%% file: main-uniform.tex
\documentclass[11pt, oneside]{amsart}
\input{header}
\title[Dehn functions of SPFs]{Dehn functions of subgroups of products of free groups \\ Part I: Uniform upper bounds}
\author{Dario Ascari}
\address{\parbox{\linewidth}{Department of Mathematics, University of the Basque Country,\\
Barrio Sarriena, Leioa, 48940, Spain\vspace{1.5pt}}}
\email{ascari.maths@gmail.com}
\author{Federica Bertolotti}
\address{Scuola Normale Superiore, Piazza dei Cavalieri 7, 56126 Pisa, Italy}
\email{federica.bertolotti@sns.it}
\author{Giovanni Italiano}
\address{\parbox{\linewidth}{Mathematical Institute, University of Oxford, \\Andrew Wiles Building,
Woodstock Road, OX2 6GG Oxford, UK}\vspace{1.5pt}}
\email{italiano@maths.ox.ac.uk}
\author{Claudio Llosa Isenrich}
\address{Faculty of Mathematics, KIT, Englerstr. 2, 76131 Karlsruhe, Germany}
\email{claudio.llosa@kit.edu}
\author{Matteo Migliorini}
\address{Faculty of Mathematics, KIT, Englerstr. 2, 76131 Karlsruhe, Germany}
\email{matteo.migliorini@kit.edu}

\keywords{Dehn function, direct product of free groups, residually free groups}
\subjclass{20F65 (20F05, 20F67, 20F69, 57M07)}

\begin{document}

\begin{abstract}
    Subgroups of direct products of finitely many finitely generated free groups form a natural class that plays an important role in geometric group theory. Its members include fundamental examples, such as the Stallings--Bieri groups. This raises the problem of understanding their geometric invariants. We prove that finitely presented subgroups of direct products of three free groups, as well as subgroups of finiteness type $\mathcal{F}_{n-1}$ in a direct product of $n$ free groups, have Dehn function bounded above by $N^9$. This gives a positive answer to a question of Dison within these important subclasses and provides new insights in the context of Bridson's conjecture stating that finitely presented subgroups of direct products of free groups have polynomially bounded Dehn function.
    To prove our results we generalise techniques for ``pushing fillings'' into normal subgroups.
\end{abstract}

\maketitle

\addtocontents{toc}{\protect\setcounter{tocdepth}{1}}

\tableofcontents

\addtocontents{toc}{\protect\setcounter{tocdepth}{1}}

\input{sections-uniform/intro-u}
\section{Preliminaries}
\input{sections-uniform/area-and-dehn-function}

\def\r{}
\input{sections-uniform/push-down}

\def\r{r}
\section{Kernels in products of free groups}\label{sec:general-strategy}

\input{sections-uniform/notation}\label{sec:notation}
\input{sections-uniform/general-case}

\input{sections-uniform/residually-free-groups}

\clearpage
\appendix
\addcontentsline{toc}{section}{\appendixname}
\addtocontents{toc}{\protect\setcounter{tocdepth}{0}}

\RenewDocumentCommand{\R}{ O{\r}O{\n}}{{\mathcal R}_{#1}^{#2}}
\RenewDocumentCommand{\Rthick}{O{\bq} O{\n} O{\r}}{{{\mathbf{\mathcal R}}^{#2}_{#3}(#1)}}
\RenewDocumentCommand{\X}{O{\n} O{\r}}{{\mathcal{X}_{#2}^{#1}}}
\def\n{3}
\input{sections-uniform/appendix-three-factors}

\def\n{n}
\input{sections-uniform/appendix-more-factors}

\def\n{4}
\input{sections-uniform/appendix-four-factors}

\bibliographystyle{alpha}
\bibliography{references}

\end{document}

%% file: header.tex
\usepackage{graphicx} %
\usepackage{amsmath, amsthm, amssymb}
\usepackage{hyperref}
\usepackage[capitalize]{cleveref}
\usepackage{xspace}
\usepackage{tikz-cd}
\usepackage{tikz}
\usepackage[colorinlistoftodos, textsize=small]{todonotes}
\usepackage{enumitem}
\usepackage{xparse}
\usepackage{mathtools}
\usepackage{etoolbox}
\usetikzlibrary{calc}
\usepackage{pgf}
\usepackage{dsfont}
\usepackage{kantlipsum} %
\usepackage{subcaption} %

\usepackage{tabularray}
\UseTblrLibrary{diagbox}
\NewColumnType{C}{Q[c,$]}

\calclayout

\newcommand{\NN}{\ensuremath{\mathbb N}}

\newcommand{\ZZ}{\ensuremath{\mathbb Z}}

\newcommand{\N}{N}

\newcommand{\inv}[1]{\overline{#1}}
\newcommand{\abs}[1]{\left|#1\right|}
\newcommand{\SB}[1][3]{\mathrm{SB}(#1)}

\renewcommand{\epsilon}{\varepsilon}

\NewDocumentCommand{\Dehn}{m O{N}}{\delta_{#1}(#2)}
\NewDocumentCommand{\radius}{m O{N}}{\rho_{#1}(#2)}
\let\asympleq\preccurlyeq
\let\asympgeq\succcurlyeq

\NewDocumentCommand{\freecopy}{m m}{F_{#1}^{(#2)}}
\NewDocumentCommand{\fpgens}{o O{\r}}{\mathcal A_{#2}\IfValueT{#1}{^{(#1)}}} %
\NewDocumentCommand{\fpgen}{m m}{a_{#1}^{(#2)}}
\NewDocumentCommand{\invfpgen}{m m}{\inv a_{#1}^{(#2)}}
\NewDocumentCommand{\invfpgens}{o O{\r}}{\inv{\mathcal A}_{#2} \IfValueT{#1}{^{(#1)}}} %

\NewDocumentCommand{\fprels}{}{\mathcal C}
\NewDocumentCommand{\kgens}{o O{\r}}{\mathcal X_{#2}\IfValueT{#1}{^{(#1)}}} %
\NewDocumentCommand{\kgen}{m m}{x_{#1}^{(#2)}}

\NewDocumentCommand{\invkgens}{o O{\r}}{\inv{\mathcal X}_#2\IfValueT{#1}{^{(#1)}}} %
\newcommand{\diaggens}{\Delta}
\NewDocumentCommand{\normal}{o m}{%
	w_{#2}%
	\IfValueT{#1}{%
		\ifthenelse{\equal{#1}{d}}{%
			^\diaggens%
		}{%
			^{(#1)}%
		}%
	}
}%
\NewDocumentCommand{\absfree}{m}{\free{\xi_1,\ldots,\xi_{#1}}}
\NewDocumentCommand{\afgen}{m}{\xi_{#1}}

\NewDocumentCommand{\krels}{}{\mathcal R}

\def\r{r}
\def\n{n}
\NewDocumentCommand{\K}{O{\n} O{\r} O{}}{{K^{#1}_{#2}(\ifblank{#3}{#2}{#3})}}
\NewDocumentCommand{\R}{ O{\r}O{\n}}{{{\mathcal R}_{#1}^{#2}}}
\NewDocumentCommand{\Rthick}{O{\bq} O{\n} O{\r}}{{{\mathbf{\mathcal R}}_{#3,#2}^{#1}}}
\NewDocumentCommand{\Aii}{m O{\r}}{{\mathcal{A}^{(#1)}_{#2}}}
\NewDocumentCommand{\A}{O{\n} O{\r}}{{\mathcal{A}_{#2}^{#1}}}
\NewDocumentCommand{\Fii}{m O{\r}}{{F^{(#1)}_{#2}}}
\NewDocumentCommand{\Xii}{m O{\r}}{{\mathcal{X}^{(#1)}_{#2}}}
\NewDocumentCommand{\mXii}{m O{\r}}{{\inv{\mathcal{X}}^{(#1)}_{#2}}}
\NewDocumentCommand{\X}{O{\n} O{\r}}{{\mathcal{X}_{#2}^{#1}}}
\NewDocumentCommand{\unor}{O{\r}}{{\mathds{1}_{#1}}}
\NewDocumentCommand{\xii}{m m}{{x^{(#1)}_{#2}}}
\NewDocumentCommand{\mxii}{m m}{{\inv x^{(#1)}_{#2}}}
\NewDocumentCommand{\aii}{m m}{{a^{(#1)}_{#2}}}

\NewDocumentCommand{\maii}{m m}{{\inv a^{(#1)}_{#2}}}
\NewDocumentCommand{\prii}{m O{\r}}{\operatorname{pr}^{(#1)}_{#2}}
\NewDocumentCommand{\thickener}{m O{\bq}}{{\kappa_#2^#1}}
\NewDocumentCommand{\pres}{ O{r} O{\ell} O{m} O{n}}{\ensuremath{\Gamma_{#2,#3,#4}^{#1}}\xspace}
\NewDocumentCommand{\prestilde}{ O{r} O{\ell} O{m} O{n}}{\ensuremath{\widetilde\Gamma_{#2,#3,#4}^{#1}}\xspace}

\NewDocumentCommand{\dirprod}{O{\r} O{\n}}{%
	\ifthenelse{\equal{#2}{3}}{%
		\freecopy{#1}{a} \times \freecopy{#1}{b} \times \freecopy{#1}{c}%
	}{%
		\freecopy{#1}{1} \times \dots \times \freecopy{#1}{#2}%
	}%
}

\newcommand{\relations}{\mathcal{R}}

\newcommand{\permut}{\sigma}
\newcommand{\interv}[3]{#1_{[#2:#3]}}
\NewDocumentCommand{\pa}{m O{\bq}}{{\mathbf x}_{#2}^{(#1)}}
\NewDocumentCommand{\mpa}{m O{\bq}}{\overline{\mathbf{x}}_{#2}^{(#1)}}
\NewDocumentCommand{\paro}{m m O{1} O{\bq}}{\pa{#1}[#4_{[#3:#2]}]}
\NewDocumentCommand{\mparo}{m m O{1} O{\bq}}{\mpa{#1}[#4_{[#3:#2]}]}
\NewDocumentCommand{\simmetriz}{m}{{\oc#1}}

\NewDocumentCommand{\pax}{O{\bq}}{{\mathbf x}_{#1}}
\NewDocumentCommand{\mpax}{O{\bq}}{\overline{\mathbf{x}}_{#1}}
\NewDocumentCommand{\parox}{m O{1} O{\bq}}{\pax[#3_{[#2:#1]}]}
\NewDocumentCommand{\mparox}{m O{1} O{\bq}}{\mpax[#3_{[#2:#1]}]}

\NewDocumentCommand{\pay}{O{\bq}}{{\mathbf y}_{#1}}
\NewDocumentCommand{\mpay}{O{\bq}}{\overline{\mathbf{y}}_{#1}}
\NewDocumentCommand{\paroy}{m O{1} O{\bq}}{\pay[#3_{[#2:#1]}]}
\NewDocumentCommand{\mparoy}{m O{1} O{\bq}}{\mpay[#3_{[#2:#1]}]}

\newcommand{\push}{\operatorname{push}}
\DeclareMathOperator{\Area}{Area}

\NewDocumentCommand{\swaprels}{}

\newcommand{\Cay}[2]{\operatorname{Cay}_{#2}(#1)}

\newcommand{\rar}{\rightarrow}
\newcommand{\ol}[1]{\overline{#1}}
\newcommand{\oc}[1]{\widehat{#1}}
\newcommand{\intsup}[1]{\lceil #1 \rceil}

\newcommand{\norma}[1]{\|#1\|}
\newcommand{\gen}[1]{\langle #1 \rangle}
\newcommand{\ppres}[2]{\langle #1\  |\  #2 \rangle}
\newcommand{\bN}{\mathbb{N}}
\newcommand{\bZ}{\mathbb{Z}}
\newcommand{\cA}{\mathcal{A}}

\newcommand{\cC}{\mathcal{C}}
\newcommand{\cQ}{\mathcal{Q}}
\newcommand{\cR}{\mathcal{R}}

\newcommand{\cX}{\mathcal{X}}

\newcommand{\bq}{\mathbf{q}}
\newcommand{\bx}{\mathbf{x}}

\newcommand{\fibration}{\psi}
\newcommand{\free}[1]{F (#1)}
\newcommand{\eqfree}{=}
\newcommand{\eqg}[1]{=_{#1}}

\newcommand{\coloneq}{\coloneqq}

\NewDocumentCommand{\simplecomms}{ O{\r} }{\mathcal R_{#1,1}}
\NewDocumentCommand{\comms}{ O{\r} }{\mathcal R_{#1,2}}
\NewDocumentCommand{\swaps}{ O{\r} }{\mathcal R_{#1,3}}
\NewDocumentCommand{\triplecomms}{ O{\r} }{\mathcal R_{#1,4}}
\NewDocumentCommand{\quadcomms}{ O{\r} }{\mathcal R_{#1,5}}

\let\originalleft\left
\let\originalright\right
\renewcommand{\left}{\mathopen{}\mathclose\bgroup\originalleft}
\renewcommand{\right}{\aftergroup\egroup\originalright}

\RequirePackage{mathtools}
\DeclarePairedDelimiter\set\{\}

\DeclarePairedDelimiterX\commsub[1]{[}{]}{#1,#1}

\NewDocumentCommand{\sequence}{O{1} m o}{#2_#1, \dots \IfValueT{#3}{, #2_#3}} %

\newcommand\generatedby[1]{\presentation{#1}{}}
\DeclarePairedDelimiterX\presentation[2]\langle\rangle{
	#1 \ifblank{#2}{}{\mid #2}
}
\NewDocumentCommand{\finitetype}{o}{\ensuremath{\IfValueTF{#1}{\mathcal F_#1}{\mathcal F}}\xspace}

\DeclarePairedDelimiter{\normalclosure}{\langle\langle}{\rangle\rangle}

\NewDocumentCommand{\shortexactsequence}{m O{} m O{} m O{}}{
	\begin{tikzcd}[ampersand replacement=\&, column sep=small]
		1 \arrow[r] \& #1 \arrow[r, "#2"] \& #3 \arrow[r, "#4"] \& #5 \arrow[r] \& 1#6
	\end{tikzcd}
}

\NewDocumentCommand{\range}{O{1} m}{%
	\ifthenelse{\equal{#1}{1} \AND \equal{#2}{3}}{%
		\set{1,2,3}%
	}{%
		\ifthenelse{\equal{#1}{1} \AND \equal{#2}{2}}{%
			\set{1,2}%
		}{%
			\set{#1, \dots, #2}%
		}%
	}
}

\newcommand{\isom}{\cong}

\newcommand{\suchthat}{\mid}
\newcommand{\Suchthat}{\ \middle|\ }

\newcommand\bigfun[5]{%
	\begin{tikzcd}[
			column sep=2em,
			row sep=1ex,
			ampersand replacement=\&
		]
		#1\colon \&[-2.5em]
		#2\vphantom{#3} \arrow[r] \&
		#3\vphantom{#2} \\
		\&
		#4\vphantom{#5}  \arrow[r,mapsto] \&
		#5\vphantom{#4}
	\end{tikzcd}%
}

\newtheorem{lemma}{Lemma}[section]
\newtheorem{proposition}[lemma]{Proposition}
\newtheorem{theorem}[lemma]{Theorem}
\newtheorem{corollary}[lemma]{Corollary}
\newtheorem{conjecture}[lemma]{Conjecture}
\theoremstyle{remark}
\newtheorem{remark}[lemma]{Remark}
\newtheorem{notation}{Notation}
\theoremstyle{definition}
\newtheorem{definition}[lemma]{Definition}
\newtheorem{question}[lemma]{Question}
\newcounter{MainTheorems}

\newtheorem{maintheorem}[MainTheorems]{Theorem}
\crefname{maintheorem}{Theorem}{Theorems}
\crefname{equation}{}{}
\crefformat{enumi}{#2\textup{(#1)}#3}
\crefmultiformat{enumi}{#2\textup{(#1)}#3}{,~#2\textup{(#1)}#3}{,~#2\textup{(#1)}#3}{ and~#2\textup{(#1)}#3}

\makeatletter
\providecommand\@dotsep{5}
\def\listtodoname{List of Todos}
\def\listoftodos{\@starttoc{tdo}\listtodoname}
\makeatother

\newcommand{\todoerror}[2][]{}
\newcommand{\todowarn}[2][]{}
\newcommand{\todoremark}[2][]{}
\newcommand{\todohint}[2][]{}
\newcommand{\done}[2][]{}
\renewcommand{\todoerror}[2][]{}
\renewcommand{\todowarn}[2][]{}
\renewcommand{\todoremark}[2][]{}
\renewcommand{\todohint}[2][]{}
\renewcommand{\done}[2][]{}

%% file: sections-uniform/intro-u.tex
\section{Introduction}

A natural problem in group theory is to understand the geometry of subgroups of direct products of groups and, in particular, of subgroups of direct products of free groups. Dehn functions provide an important and classical geometric invariant of finitely presented groups. In this work we prove:

\begin{maintheorem}\label{main:3-factors}
	The Dehn function of every finitely presented subgroup of a direct product of three free groups is bounded above by a polynomial of degree nine.
\end{maintheorem}

We will now motivate this result, place it in context, and explain further main results of our work.

\subsection{The geometry of subgroups of direct products of free groups}

Free groups are the universal objects in group theory and as such play a distinguished role. Some of the most basic constructions of new groups from old groups are free and direct products. Free products of free groups are free and by the Nielsen--Schreier Theorem the same is true for all of their subgroups. This raises the question whether direct products of free groups and their subgroups are similarly rigid. It is well-known that this is not the case. In fact, subgroups of direct products of free groups (short: SPFs) have been the source of many interesting examples in group theory. Baumslag and Roseblade \cite{BauRos-84} proved that there are uncountably
many isomorphism classes of subgroups of $F_2\times F_2$ and Mihailova \cite{Mih-68} proved that there is a finitely generated subgroup $K\leq F_2\times F_2$ for which the membership problem is unsolvable. Moreover, for every $n\in \mathbb{N}_{\geq 3}$ the first examples of groups of finiteness type $\mathcal{F}_{n-1}$ and not $\mathcal{F}_n$, constructed by Stallings ($n=3$) \cite{stallings} and Bieri ($n \geq 4$) \cite{bieri}, are subgroups of direct products of $n$ free groups; here we call a group $G$ of type $\mathcal{F}_n$ if it admits a $K(G,1)$ with finite $n$-skeleton.

On the other hand, if $K\leq F_{m_1}\times \cdots \times F_{m_n}$ is of type $\mathcal{F}_n$ for $n\geq 2$ Baumslag and Roseblade ($n=2$) \cite{BauRos-84} and Bridson, Howie, Miller and Short \cite{BHMS-09,BHMS-13} ($n\geq 3$) proved that $G$ is virtually a direct product of finitely generated free groups, giving a version of the Nielsen--Schreier Theorem under the additional assumption of strong enough finiteness properties. More generally, Bridson, Howie, Miller and Short \cite{BHMS-13} classified finitely presented SPFs in terms of their higher finiteness properties. In particular, they proved that all finitely presented SPFs can be constructed from direct products of free groups by taking finitely many fibre products over finitely generated nilpotent groups. Moreover, SPFs with high enough finiteness properties in relation to the number of factors are \emph{virtually coabelian} \cite[Corollary 3.5]{Kuckuck-14} (see also \cite{Kochloukova-10}), that is, $K\cong \ker(\psi \colon F_{m_1}\times \dots \times F_{m_n}\to \mathbb{Z}^r)$ for some product of free groups and a surjective homomorphism $\psi$. We call $r$ the corank of $K$.

These results exhibit that finitely presented SPFs form a natural and rich class that played a pivotal role in our understanding of finiteness properties of groups and thus in geometric group theory. This naturally raises the interest in their finer geometric invariants.

\subsection{Dehn functions of SPFs and residually free groups}
Dehn functions quantify finite presentability by measuring the number $\delta_G(N)$ of conjugates of relations required to determine if a word of length $N$ in the generators of $G$ represents the trivial element in $G$, where $G$ is a group given by a finite presentation.
Motivated by the fact that finitely presented SPFs are constructed using iterated fibre products from free groups and nilpotent groups, both of which have polynomially bounded Dehn function, Bridson conjectured:
\begin{conjecture}[Bridson]\label{conj:Bridson-SPF}
	Every finitely presented subgroup of a direct product of free groups has polynomially bounded Dehn function.
\end{conjecture}
More generally, Bridson conjectured that residually free groups have polynomially bounded Dehn function, due to similarities in their structure theory (see \cite{BHMS-09,BHMS-13}). Here a group $G$ is called \emph{residually free} if for every $g\in G\setminus \left\{1\right\}$ there is a homomorphism $\phi:G\to F_2$ with $\phi(g)\neq 1$.
\begin{conjecture}[Bridson]\label{conj:Bridson-RF}
	Every finitely presented residually free group has polynomially bounded Dehn function.
\end{conjecture}
We prove:
\begin{maintheorem}\label{main:equivalence-of-conjectures}
	\Cref{conj:Bridson-RF} holds if and only if \Cref{conj:Bridson-SPF} holds.
\end{maintheorem}

This provides further impetus towards understanding Dehn functions of SPFs. In his thesis \cite{Dison-08} Dison raised these conjectures as questions and conducted the first systematic study of Dehn functions of finitely presented SPFs. Predating Dison's work, a longstanding open problem was determining the precise Dehn functions of the Stallings--Bieri groups $\SB[n]$. Their Dehn function was first bounded by Gersten \cite{gersten-95} in 1995 by $N^5$ and Bridson \cite{bridson-doubles} asserted in 1999 that it is in fact quadratic. While Bridson's proof only gave a cubic upper bound, his intuition was confirmed by Dison, Elder, Riley and Young \cite{BERY-09} in 2009 for $n=3$ and by Carter and Forester \cite{carter2017stallings} in 2017 for all $n$.

In his thesis Dison proves that subgroups of $F_{m_1}\times \cdots \times F_{m_n}$ of type $\mathcal{F}_{n-1}$ satisfy a polynomial isoperimetric function, giving a positive answer to \cref{conj:Bridson-SPF} for this class of groups, which includes all finitely presented subgroups of a direct product of $3$ free groups. As we will explain, these groups are all either direct products of finitely many finitely generated free groups or commensurable to a class of groups that he denotes by $K_m^n(r)$ for integers $m,~n,~r$ with $r\leq m$, defined as kernels of morphisms $\psi\colon F_m\times \cdots \times F_m\to \mathbb{Z}^r$ from a direct product of $n$ free groups of rank $m$
whose restriction to each factor is surjective. More precisely, Dison proved that $\delta_{K_m^n(r)}(N)\preccurlyeq N^{2r+2}$ and for $n\geq \max\left\{3,2r\right\}$ he even attained a stronger bound of $N^5$.

\subsection{Uniform upper bounds}
Considering that Dison's upper bounds only depend on $r$ if the number of factors is small compared to the rank of the free abelian quotient, it is natural to ask if this is merely a relic of his techniques and if maybe it is even possible to obtain uniform upper bounds on Dehn functions of interesting classes of SPFs. Dison asked the strongest possible version of this question, which is a key motivation for our work.
\begin{question}[Dison \cite{Dison-08}]\label{qn:Dison-uniform}
	Is there a uniform polynomial $p$ such that for every finitely presented SPF $K$ we have $\delta_K(N)\preccurlyeq p(N)$?
\end{question}
Dison \cite{Dison-09} and Bridson \cite{BridsonPersonal} proved that the Dehn function of $K^3_2(2)$ is at least cubic; we will refer to $ K^3_2(2) $ as the \emph{Bridson--Dison group}. This shows that we can certainly not hope for a quadratic upper bound in general. Subsequently, in \cite{LlosaTessera} Tessera and the fourth author gave a negative answer to \cref{qn:Dison-uniform} by constructing for every natural number $n\geq 3$ a subgroup of a direct product of $n$ free groups with Dehn function $\succcurlyeq N^n$. In conjunction with Dison's results for the class $K_m^n(r)$ this naturally leads to the following restricted versions of \cref{qn:Dison-uniform}, the first of which is a slight variation of \cite[Question 3]{LlosaTessera}.

\begin{question}\label{qn:Dison-fixed-factors}
	For fixed $n\in \mathbb{N}_{\geq 3}$, is there a uniform polynomial $p(N)$ such that every finitely presented SPF $K$ in a direct product of $n$ free groups satisfies $\delta_K(N)\preccurlyeq p(N)$? %
\end{question}
\begin{question}\label{qn:Dison-fixed-finiteness-props}
	For $n\in \mathbb{N}_{\geq 3}$ and $k\in \range[2]{n-1}$, is there a uniform polynomial $p(N)$ such that every SPF of type $\mathcal{F}_k$ in a direct product of $n$ free groups satisfies $\delta_G(N)\preccurlyeq p(N)$?
\end{question}

We will prove that both questions have a positive answer in two important cases, showing that the number of factors and the finiteness properties may play a key role in understanding the Dehn functions of SPFs. Indeed, \cref{main:3-factors} gives a positive answer to \cref{qn:Dison-fixed-factors} for the three factor case, while the following more general result gives a positive answer to \cref{qn:Dison-fixed-finiteness-props} for $k=n-1$.

\begin{maintheorem}\label{main:type-F-n-1}
	If $K\leq F_{m_1}\times \cdots \times F_{m_n}$ is a finitely presented subgroup of type $\mathcal{F}_{n-1}$, then the Dehn function of $K$ is bounded above by a polynomial of degree $d$, where $d=9$ for $n\leq 3$, $d=5$ for $n=4$, and $d=4$ for $n\geq 5$.
\end{maintheorem}

\cref{tbl:dehn-functions} illustrates \cref{main:type-F-n-1}, as well as the results of our subsequent paper \cite{PreciseComputations-25}. Note that \cref{main:type-F-n-1} can be generalised to the case where the factors are limit groups (see \cref{cor:limit-groups}).

\begin{table}[h]
	\begin{tblr}{
			colspec={CCCCCCCC},
			hline{2-8} = {solid},
			vlines={2-7}{solid},
			cell{3}{2-6} = {green},
			cell{4}{4-6} = {green},
			cell{4}{4-6} = {green},
			cell{5}{5-6} = {green},
			cell{6}{6-6} = {green},
			cell{1}{1} = {r=1,c=6}{c,mode=text}
		}
		Best upper bounds                                           \\
		\diagbox{r}{n} & 3      & 4      & 5      & 6      & \dots  \\
		2              & N^4    & N^2    & N^2    & N^2    & \dots  \\
		3              & N^8    & N^5    & N^2    & N^2    & \dots  \\
		4              & N^9    & N^5    & N^4    & N^2    & \dots  \\
		5              & N^9    & N^5    & N^4    & N^4    & \dots  \\
		\vdots         & \vdots & \vdots & \vdots & \vdots & \ddots
	\end{tblr}
	\hspace{0.3cm}
	\begin{tblr}{
			colspec={CCCCCCCC},
			hline{2-8} = {solid},
			vlines={2-7}{solid},
			cell{3}{2-6} = {green},
			cell{4}{4-6} = {green},
			cell{4}{4-6} = {green},
			cell{5}{5-6} = {green},
			cell{6}{6-6} = {green},
			cell{1}{1} = {r=1,c=6}{c,mode=text}
		}
		Best lower bounds                                           \\
		\diagbox{r}{n} & 3      & 4      & 5      & 6      & \dots  \\
		2              & N^4    & N^2    & N^2    & N^2    & \dots  \\
		3              & N^4    & N^2    & N^2    & N^2    & \dots  \\
		4              & N^4    & N^2    & N^2    & N^2    & \dots  \\
		5              & N^4    & N^2    & N^2    & N^2    & \dots  \\
		\vdots         & \vdots & \vdots & \vdots & \vdots & \ddots
	\end{tblr}
	\caption{The best known upper and lower bounds for $ \K $, obtained by combining \cref{main:type-F-n-1} with the results of our subsequent paper \cite{PreciseComputations-25}. In the table we have highlighted the cases where the two bounds coincide and thus give the precise Dehn function.}
	\label{tbl:dehn-functions}
\end{table}

To prove \cref{main:3-factors,main:type-F-n-1} we first reduce to showing the result for $K_r^n(r)$, by observing that all subgroups satisfying the assumptions are virtually coabelian and then that they are commensurable with one of the $K_r^n(r)$. To obtain upper bounds on the Dehn functions of $K_r^n(r)$ we generalise a known strategy for obtaining upper bounds on Dehn functions of kernels of homomorphisms to $\mathbb{Z}$ which is known as ``pushing fillings''. The term ``pushing fillings'' was coined by Abrams, Brady, Dani, Duchin and Young \cite{ABDDY-13} and some of the underlying ideas already appeared in Gersten--Short's proof \cite{GerSho-02} that kernels of homomorphisms from hyperbolic groups onto free groups have polynomially bounded Dehn function (see \cite{Llo-24} for a generalization of Gersten--Short's result to free abelian quotients). The idea of pushing fillings is that,
given a finitely presented kernel $K$ of a surjective homomorphism $\psi \colon G \to Q$ and an area-radius pair for $G$, one fills a null-homotopic word in a generating set of $K$ by first choosing a van Kampen diagram for this word in $G$ that is minimal for the area-radius pair. One then pushes it down to a filling in $K$. This comes at the expense of replacing the original relations by larger van Kampen diagrams. However, if one manages to control the area of these van Kampen diagrams, then this provides an upper bound on the Dehn function of $K$.
When applying this technique, the first challenge lies in carefully choosing a push-down map that enables us to attain a strong bound. Once this is chosen, it allows us to replace the relations in $G$ by words in $K$ of bigger length, and we need to explain how to explicitly fill these new words. This is done by introducing a new technique for filling words, which we call \emph{doubling}, and which takes advantage of the high level of symmetry in our presentations.

\begin{remark}
	A well-known example where pushing fillings leads to very good bounds on Dehn functions are Bestvina--Brady groups \cite{Dison-08-II,ABDDY-13}. Dison's proofs in \cite{Dison-08} also rely on pushing fillings. However, he pushes them one dimension of $\mathbb{Z}^r$ at a time, which comes at the expense of an upper bound that depends on $r$. In contrast, we reveal presentations together with a pushing map that allow us to simultaneously reduce the height in all directions, thus leading to uniform bounds.
\end{remark}

Since our proofs of \cref{main:3-factors,main:type-F-n-1} rely on the fact that we know that all SPFs satisfying the assumptions are coabelian (see \cite[Corollary 3.5]{Kuckuck-14}), our work also provides positive evidence towards the following question of Tessera and the fourth author \cite[Question 4]{LlosaTessera}.
\begin{question}
	Is there a uniform polynomial $p(N)$ such that for all coabelian SPFs $K$ we have $\delta_K(N)\preccurlyeq p(N)$?
\end{question}

Finally, we mention that \cref{qn:Dison-fixed-factors} above is a slight variation of \cite[Question 3]{LlosaTessera}, which asked if the Dehn function of every finitely presented SPF in a direct product of $n$ free groups is bounded above by $N^n$. In our subsequent paper \cite{PreciseComputations-25}, we prove that a product of three free groups has a finitely presented subgroup with quartic Dehn function, giving a negative answer to \cite[Question 3]{LlosaTessera}. In light of this, we believe that \cref{qn:Dison-fixed-factors} provides a better phrasing of this question. It would be interesting to know if \cref{qn:Dison-fixed-factors} also has a positive answer for more than $3$ factors (maybe under the additional assumption that the subgroup is coabelian) and, if it does, what the asymptotic behaviour of the optimal upper bound is.

\subsection*{Structure}
The paper is structured as follows.%
\begin{itemize}
	\item In \cref{sec:preliminaries} we fix some notation and recall the definition of area and Dehn function of a group. Moreover, we give a precise formulation of the push-down strategy that will be used throughout the paper.
	\item In \cref{sec:general-strategy} we describe how to obtain \cref{main:3-factors,main:type-F-n-1} using the push-down strategy. The proof is based on what we call \emph{doubling technique}, which allows us to estimate the area of a large family of words with some symmetries. The whole section relies on several computations, which we postpone to the appendix to focus the exposition on the key steps.
	\item In \cref{sec:equivalence-of-conjectures} we prove \cref{main:equivalence-of-conjectures}, which follows from a more general result that bounds Dehn functions of subgroups of direct products of groups in terms of those of the factors and those of subgroups of direct products of free groups.

	\item Finally, in \cref{sec:appendix-three-factors,sec:appendix-5,sec:appendix-four-factors} we prove the technical lemmas required for \cref{main:3-factors,main:type-F-n-1}.
\end{itemize}

\subsection*{Guide for the reader}

\Cref{sec:preliminaries} is required for understanding the rest of the paper, as it fixes notation and describes the pushing argument, which is used many times thereafter.
\cref{sec:general-strategy,sec:residually-free-case} are fairly independent to one another, so the reader interested in only one of the main theorems can jump directly from \cref{sec:preliminaries} to the relevant section. %
Finally, the appendix does not provide additional insight, it is computation-heavy, and it should only be read after \cref{sec:preliminaries,sec:general-strategy}.

\subsection*{Acknowledgements} The first author was supported by the Basque Government grant IT1483-22.
The second author would like to thank Karlsruhe Institute of Technology for the hospitality, and was partially supported by the INdAM GNSAGA Project, CUP E55F22000270001. The third author gratefully acknowledges support from the Royal Society through the Newton International Fellowship (award number: NIF\textbackslash R1\textbackslash 231857). The fourth author would like to thank Robert Kropholler and Romain Tessera for many discussions about the topics of this work. The fourth and the fifth author gratefully acknowledge funding by the DFG 281869850 (RTG 2229).

%% file: sections-uniform/area-and-dehn-function.tex
\label{sec:preliminaries}
Throughout the whole paper, if $G$ is a group, to write the inverse of $g \in G$ we use interchangeably the notations $ \inv g$ and $g^{-1} $. The commutator of $ g, h \in G $ is denoted by $ [g,h] = gh \inv g \inv h$.

\subsection{Free groups and homomorphisms}\label{sec:substitutions}
Given a set $S$, we denote by $F(S)$ the free group with basis $S$. Every element $w\in F(S)$ can be represented by a unique reduced word in the alphabet $S\sqcup S^{-1}$ (the alphabet given by $S$ and by the formal inverses of elements of $S$). We define the \emph{length} $\abs{w}_S$ to be the number of letters of the reduced word representing $w$.

Suppose we are given a finite ordered tuple $S=(s_1,\dots,s_n)$ and an element $w=w(s_1,\dots ,s_n)\in F(S)$. Let $T$ be any other set and let $u_1,\dots,u_n\in F(T)$. Consider the unique homomorphism $\theta \colon F(S)\rar F(T)$ satisfying $\theta(s_i)=u_i$ for $i \in \range n$. Throughout the paper, we denote $w(u_1,\dots,u_n)\coloneq\theta(w)$.

\begin{definition}\label{def:homo-norm}
	Given a homomorphism $\phi\colon F(S)\rar F(T)$, with $S$ finite, define
	\[
		\norma{\phi}_{S,T}\coloneq\max\{\abs{\phi(s)}_T \suchthat s\in S\} .
	\]
\end{definition}

We have $\abs{\phi(w)}_T\le\norma{\phi}_{S,T}\abs{w}_S$ for all $w\in F(S)$. In what follows, we  write $\norma{\phi}$ and $\abs{w}$, omitting the dependence on $S,\,T$ when it is clear from the context.

\subsection{Area in free groups}\label{sec:definition-area}
Let $F$ be a free group. Given a subset $\cR\subseteq F$ and an element $w\in F$, we define the \textit{area} of $w$ as
\[
	\Area_{\mathcal R}(w) = \inf \set*{M\in\bN \Suchthat w = \prod_{i=1}^M \inv u_i R_i u_i,\,u_i \in F,\,R_i \in \mathcal R} .
\]
and in particular we set $\Area_\cR(w)=+\infty$ if $w$ does not belong to the normal subgroup generated by $\cR$. For a subset $\cQ \subseteq F$ we define
\[\Area_\cR(\cQ)=\sup\{\Area_\cR(w) \suchthat w\in \cQ\} .\]

\begin{lemma}
	For $\mathcal R,\mathcal R',\mathcal R''\subseteq F$ we have
	\[
		\Area_{\cR''}(\cR)\le\Area_{\cR''}(\cR')\cdot\Area_{\cR'}(\cR).
	\]
\end{lemma}
\begin{proof}
	Take an element $R''\in\cR''$. We can write $R''=\prod_{i=1}^k\ol{u}_iR_i'u_i$ for some $u_i\in F$, $R_i'\in\cR'$ and $k\le\Area_{\cR'}(\cR'')$. But for $i=1,\ldots,k$ we can write $R_i'=\prod_{j=1}^{h_i}\ol{v}_{ij}R_{ij}v_{ij}$ for some $v_{ij}\in F$, $R_{ij}\in\cR$ and $h_i\le\Area_{\cR}(\cR')$. It follows that $R''=\prod_{i=1}^k\prod_{j=1}^{h_i}\ol{u}_i\ol{v}_{ij}R_{ij}v_{ij}u_i$ and thus $R''$ can be written using at most $\Area_{\cR''}(\cR')\cdot\Area_{\cR'}(\cR)$ conjugates of elements of $\cR$.
\end{proof}

\begin{lemma}\label{lem:area-bound-after-homo}
	Let $\phi \colon F\rar F'$ be a homomorphism between free groups and let $\cR,\cQ\subseteq F$. Then we have
	\[
		\Area_{\phi(\cR)}(\phi(\cQ))\le\Area_{\cR}(\cQ)
	\]
\end{lemma}
\begin{proof}
	Given $w'\in \phi(\cQ)$ we write $w'=\phi(w)$ for $w\in\cQ$, and we find an identity $w=\prod_{i=1}^k \ol{u}_iR_iu_i$ with $w,u_1,\dots ,u_k\in F$ and $R_1,\dots ,R_k\in\cR\cup\{1\}$  and $k\le \Area_\cR(\cQ)$. We apply the homomorphism $\phi$ to obtain the identity $\phi(w)=\prod_{i=1}^k \phi(u_i)^{-1}\phi(R_i)\phi(u_i)$
	with $\phi(R_i)\in\phi(\cR)$ and $k\le \Area_\cR(\cQ)$. The conclusion follows.
\end{proof}

\begin{corollary}\label{lem:area-composition}
	Let $\phi \colon F\rar F'$, $ \psi \colon F' \to F''  $ be homomorphisms between free groups and let $\cR\subseteq F, \mathcal R' \subseteq F', \mathcal R'' \subseteq F''$. Then we have
	\[
		\Area_{\cR''}(\psi(\phi(\cR)))\le\Area_{\cR''}(\psi(\cR'))\cdot\Area_{\cR'}(\phi(\cR)).
	\]
\end{corollary}

\subsection{The free group over a subset of a group}
Let $G$ be a group and let $S\subseteq G$ be a subset (\emph{not} necessarily a generating set).  Then we have a natural homomorphism
\[F(S)\rar G\]
that sends each element of the basis $S$ to the corresponding element of $G$. Such map is surjective if and only if the set $S$ generates the group $G$. Given two elements $u,w\in F(S)$ we write $u \eqg{G} w$ if $u$ and $w$ project to the same element of $G$. Moreover, for $ g \in G $, we denote with $ \abs{g}_S $ the minimal length of an element $ w \in F(S) $ projecting on $g$. We set $\abs{g}_S=+\infty$ if $g$ does not belong to the subgroup generated by $S$. If $S$ is a finite generating set for $G$, then $\abs{g}_S$ coincides with the distance of $g$ from the origin in the Cayley graph $\Cay{G}{S}$.

\subsection{Dehn function of a group}
Let $G$ be a group with a generating set $S$ and let $\cR\subseteq F(S)$ be a set of elements such that $R=_G1$ for every $R\in \cR$. It follows immediately from the definitions that, for every $w\in F(S)$, if $\Area_\cR(w)<+\infty$ then $w=_G1$.

\begin{lemma}\label{lem:presentation-area}
	We have that $ \presentation S \relations $ is a presentation for the group $G$ if and only if $ \Area_{\relations}(w)<+\infty $ for every $w\in F(S)$ with $w=_G1$.
\end{lemma}
\begin{proof}
	We have that $\presentation S \cR$ is a presentation if and only if for every $w\in F(S)$ with $w=_G1$ we can write $w$ as a product of conjugates of elements of $\cR$. This happens if and only if for every $w\in F(S)$ with $w=_G1$ we have $\Area_\cR(w)<+\infty$.
\end{proof}

\begin{remark}\label{rmk:presentation-iff-areafunction-bounded}
	We will employ \cref{lem:presentation-area} to find a presentation of $\K$, for $n\geq 3$ and $r\geq 2$: in  \cref{sec:candidate-presentation} we consider a \emph{candidate presentation} for the kernel, and use it to compute an upper bound for the area of all words representing the trivial element. This strategy will both compute an upper bound for the Dehn function of $\K$ and imply \emph{a posteriori} that the candidate presentation is indeed a presentation. A more direct algorithm for computing a presentation of $\K$ can be found in \cite{Dison-08, BHMS-13}.
\end{remark}

Given a finite presentation $G = \presentation S \relations$, where $S\subseteq G$ is a finite set of generators and $\cR\subseteq F(S)$ is a finite set, we define the \emph{Dehn function} $ \delta_{G,S,\cR} \colon \NN \to \NN$ as
\[
	\delta_{G,S,\cR} (N) = \max\set{\Area_{\relations}(w) \suchthat w \in F(S), \abs{w}_S \leq N, w=_G1}.
\]

Given $ f,g \colon \NN \to \NN$, we write $ f \asympleq g $, whenever there exists a constant $C>0$ such that $ f(N) \leq C g(CN+C) + CN + C $
for every integer $ N>0$; we write $ f \asymp g $ if $ f \asympleq g $ and $ f \asympgeq g $. This is an equivalence relation on the set of all functions from $\bN$ to $\bN$.

Two different finite presentations $G = \presentation S \cR \cong \presentation {S'} {\cR'}$ give equivalent Dehn functions $\delta_{G,S,\cR} \asymp \delta_{G,S',\cR'}$. We define the \textit{Dehn function} $\delta_G$ to be the equivalence class of the functions $\delta_{G,S,\cR}$ up to the equivalence relation $\asymp$. With an abuse of notation, we will sometimes denote by $\delta_G:\bN\rar\bN$ a function in the equivalence class.

\subsection{Area-radius pairs for a group}
Let $G=\ppres{S}{\cR}$ be a finitely presented group. A pair $(\alpha,\rho)$ of functions $\alpha,\rho \colon \bN\rar\bN$ is an \textbf{area-radius pair} for the presentation if, for every $w\in F(S)$ such that $ w \eqg{G} 1 $, it is possible to write $w=\prod_{i=1}^k \ol{u}_iR_iu_i$ for some $u_1, \dots ,u_k\in F(S)$ and $R_1, \dots ,R_k\in\cR$ satisfying $k\le \alpha(\abs{w})$ and $\abs{u_1}, \dots ,\abs{u_k}\le \rho(\abs{w})$. This means that we are interested in controlling at the same time the number of relations used to fill $w$ (i.e. the area of $w$) and the length of the translations $u_i$ needed in order to write the identity $w=\prod_{i=1}^k \ol{u}_iR_iu_i$.

Note that the radius function $\rho$ in an area-radius pair $(\alpha,\rho)$ provides an upper bound on the so-called extrinsic diameter of null-homotopic words. There is also the notion of an intrinsic filling diameter, which in general is different. We refer to \cite{BriRil-09} for further details, including definitions and examples showing that they are different. %

Suppose we are given two different presentations $G=\ppres{S}{\cR}\cong\ppres{S'}{\cR'}$ for the same group: if $(\alpha,\rho)$ is an area-radius pair for $G$ with respect to the presentation $\ppres{S}{\cR}$, then there is an area-radius pair $(\alpha',\rho')$ for $G$ with respect to the presentation $\ppres{S'}{\cR'}$ satisfying $\alpha'\asymp\alpha$ and $\rho'\asymp\rho$. If $(\alpha,\rho)$ is an area-radius pair for $G$ (with respect to some presentation), then we have $\alpha\asympgeq\delta_G$.

\begin{proposition}[Papasoglu \cite{papasoglu1996asymptotic}]\label{prop:linear-radius}
	Let $G$ be a finitely presented group and suppose that $\delta_G(N)\asymp N^2$. Then there is an area-radius pair $(\alpha_G,\rho_G)$ for $G$ with $\alpha_G(N)\asymp N^2$ and $\rho_G(N)\asymp N$.
\end{proposition}
\begin{proof}
	This is proved in \cite{papasoglu1996asymptotic} on page 799.
\end{proof}

%% file: sections-uniform/push-down.tex
\subsection{Push-down map}\label{sec:push-down}
In this section we generalize the \emph{push-down} argument for fillings in kernels. %
Arguments of this kind were used by Gersten and Short \cite{GerSho-02} to prove that finitely presented kernels of homomorphisms from hyperbolic groups onto free groups have polynomially bounded Dehn function. They have subsequently been generalised and used in different contexts, including to bound Dehn functions of Bestvina--Brady groups \cite{Dison-08-II, ABDDY-13} and of subgroups of direct products of free groups \cite{Dison-08}. Here we will generalise these techniques to arbitrary quotient groups $Q$. In \cref{sec:general-strategy} we will employ our results to estimate the Dehn functions of the groups $K^n_r(r)$, defined as kernels of homomorphisms onto $Q=\ZZ^r$.

Consider the following short exact sequence
\[
	\shortexactsequence K[\iota]G[\fibration]{Q}[,]
\]
where $G = \presentation\fpgens\fprels$ is finitely presented, and $K = \generatedby\cX$ is finitely generated. We also fix a lift $ \tilde\iota \colon  \free\cX \to \free\fpgens $ that makes the following diagram commute:
\[
	\begin{tikzcd}
		& \free\cX \arrow[d] \arrow[r, "\tilde\iota"] & \free\fpgens \arrow[d]\arrow[dr, "\widetilde \psi"] \\
		1 \arrow[r] & K \arrow[r, "\iota"] & G \arrow[r, "\psi"] & Q \arrow[r] & 1,
	\end{tikzcd}
\]
where $ \widetilde \psi \colon \free \fpgens \to Q$ is defined by composition.

The general strategy is the following: given $ w \in \free{\kgens} $ representing the trivial element of $K$, we can consider its image $ \tilde w = \tilde i(w) $ as a word in the generators $ \fpgens $ that represents the trivial element of $G$. We can therefore find a disk in the Cayley complex of $G$ that bounds $ \tilde w $, and whose area is controlled by $ \delta_G $. Then, we push this filling inside the Cayley complex of $K$ via some \emph{push-down map}, to obtain a filling for the original word $w$. This filling is tessellated via the images of the relations $R$ appearing in the filling in $G$; if we can bound the area of these tiles, we get an upper bound for $ \delta_K $ (see \cref{fig:push}).

\begin{figure}
	\begin{tikzpicture}
		\node[anchor=south west, inner sep=0] (image) at (0,0) {\includegraphics{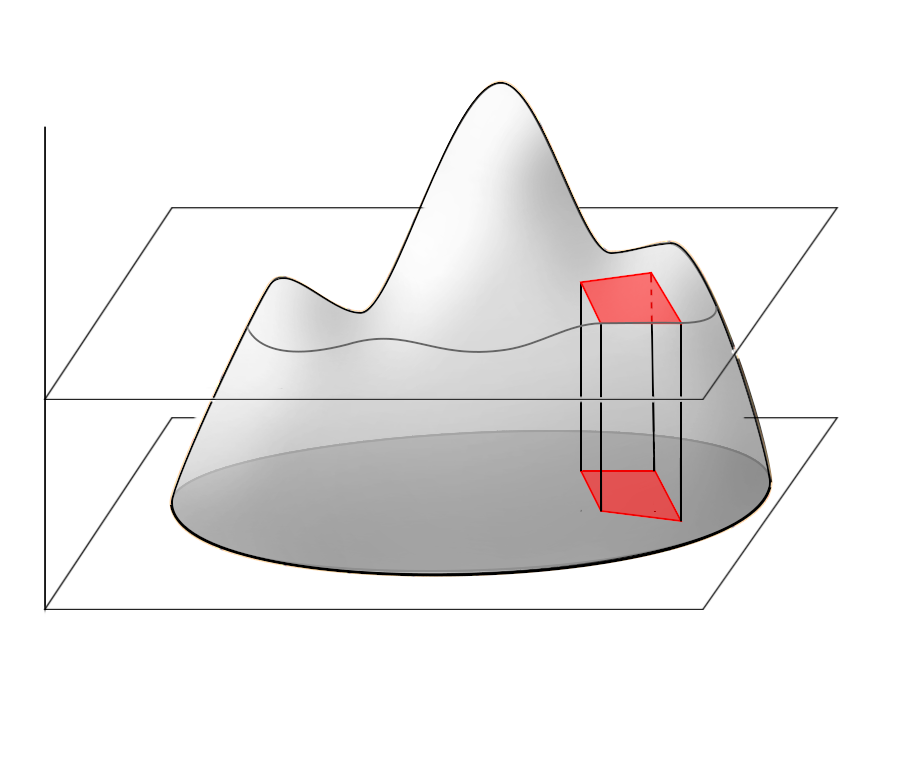}};
		\node at (6.5,1.5) {$ K$};
		\node at (6.5,5.5) {$ G$};
		\node [left] at (0.4, 1.3) {$1_Q$};
		\node [left] at (0.4, 3.1) {$q$};
		\node[right] at (6.6, 4) {$ {\widetilde\psi}^{-1}(q) $};
		\node[above left] at (0.4, 5.4) {$Q$};
		\node[left] at (4.9, 4.0) {$R$};
		\node[left] at (5.0, 2.3) {$\push_q(R)$};
	\end{tikzpicture}
	\caption{The push down map sends the boundary of every relation of $G$ to a word representing the trivial element of $K$; controlling its area in $K$ yields an upper bound for $ \delta_K $.}
	\label{fig:push}
\end{figure}
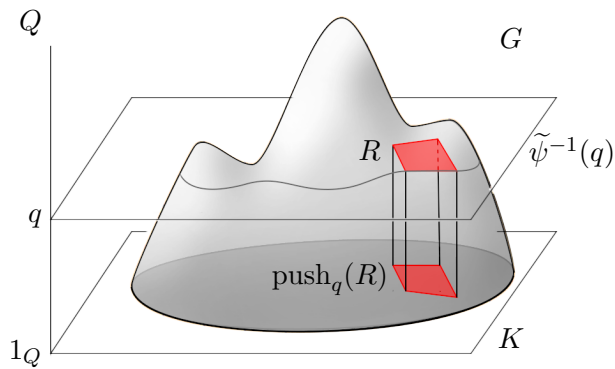

\begin{definition}\label{def:push-down}
	A \emph{push-down map} is a map (not necessarily a group homomorphism)
	\[
		\bigfun{\push}{Q \times \free\fpgens}{\free\cX}{(\bq, w)}{\push_\bq(w)}
	\]
	satisfying the following conditions:
	\begin{enumerate}
		\item For every $\bq\in Q$ and $ w, w' \in \free\fpgens $ we have
		      \[
			      \push_\bq(w \cdot w') = \push_\bq(w) \cdot \push_{\bq \cdot \widetilde\fibration(w)}(w').
		      \]
		\item For every $x\in\cX$ we have $\push_{\mathbf{1}_Q}(\tilde\iota(x))=_Kx$.
	\end{enumerate}
\end{definition}

\begin{lemma}\label{lem:push-inverse}
	A push-down map satisfies $ \push_\bq(1) = 1 $ and
	\[
		\push_\bq(\inv w) = \left(\push_{\bq \cdot \widetilde\fibration(\ol{w})}(w)\right)^{-1}
	\]
	for every $\bq\in Q$ and $ w \in \free\fpgens $.
\end{lemma}
\begin{proof}
	The first statement follows by using the first property with $ w=w'=1 $. The second statement follows again by applying the first property to $\inv w \cdot w$. %
\end{proof}

\begin{lemma}\label{lem:push-down-existence}
	For every $\bq\in Q$, choose $u_\bq\in F(\fpgens)$ with $\widetilde\psi(u_\bq)=\bq$; set $u_{\mathbf{1}_Q}=1$. For every $\bq\in Q$ and $a\in\fpgens$, choose an element $z_{\bq,a}\in F(\cX)$ such that $\tilde\iota(z_{\bq,a})=_G u_\bq\cdot a\cdot \ol{u}_{\bq\cdot\widetilde\psi(a)}$. Then there is a unique push-down map $\push\colon Q\times F(\fpgens) \rar F(\cX)$ satisfying $\push_\bq(a)=z_{\bq,a}$.
\end{lemma}
\begin{proof}
	Given a word $w'$ of length $\ell\in\bN$ in the letters $\fpgens\cup\fpgens^{-1}$, we denote by $w'[i]$ the $i$-th letter of $w'$ and by $w'[1:i]$ the word given by the initial segment of the first $i$ letters of $w'$, for $i\in\range \ell$.

	For every $\bq\in Q$ and $a\in\fpgens$, we set $\push_\bq(a)=z_{\bq,a}$ and $\push_\bq(\ol{a})=\ol{z}_{\bq \cdot \widetilde\fibration(\ol{a}),a}$. For $\bq\in Q$ and $w\in F(\cA)$, we choose a word $w'$ in the letters $\fpgens\cup\fpgens^{-1}$ representing $w$, and we set
	\[
		\push_\bq(w)=\prod_{i=1}^{\ell}\push_{\bq\cdot\widetilde\psi(w'[1:i-1])}(w'[i]),
	\]
	Notice that for every $\bq\in Q$ and $a\in\fpgens\cup\fpgens^{-1}$ we have $\push_\bq(a)\push_{\bq\cdot\widetilde\psi(a)}(\ol{a})=1$; it follows that a different choice of a word $w'$ representing the element $w\in F(\fpgens)$ gives the same value of $\push_\bq(w)$, and thus the map $\push$ is well-defined.

	By direct check we have that $\push_\bq(w \cdot w') = \push_\bq(w) \cdot \push_{\bq \cdot \widetilde\fibration(w)}(w')$. By induction on the length of the word representing $w\in F(\fpgens)$, we have that $\tilde\iota(\push_\bq(w))=_G u_{\bq}\cdot w\cdot \ol{u}_{\bq\cdot\widetilde\psi(w)}$. In particular, for $x\in\cX$ we obtain $\push_{\mathbf{1}_Q}(\tilde\iota(x))=_K x$. The conclusion follows.
\end{proof}

\begin{theorem}\label{thm:push-down}
	Suppose that we are given a short exact sequence
	\[
		\shortexactsequence KGQ
	\]
	with $ G = \presentation{\fpgens}{\mathcal C} $ finitely presented and $K=\gen{\cX}$ finitely generated. Let $ (\alpha_G, \rho_G) $ be an area-radius pair for $G$. Let $ \cR \subset \free\cX $ be a finite set of elements representing the trivial element of $K$. Let $ \push $ be a push-down map for the sequence. Suppose that $\Area_\cR(\push_{\mathbf{1}_Q}(\tilde\iota(x))\cdot\ol{x})<+\infty$ for all $x\in\cX$. Suppose that
	\[
		\max_{C \in \fprels, \abs{\bq}_{\widetilde\fibration(\fpgens )} \leq N}\Area_\relations(\push_\bq(C)) \leq f(N)
	\]
	for some function $ f \colon \NN \to \NN $. Then $ K $ is finitely presented by $ \presentation{\mathcal X}{\relations} $ and its Dehn function satisfies \[
		\Dehn{K} \asympleq \alpha_{G}(N) \cdot f(\radius{G}).
	\]
\end{theorem}

\begin{proof}
	Suppose that we are given $w \in \free\cX$ representing the trivial element in $K$, and let $w = x_1 \dots x_N$ be the reduced word representing $w$, with $x_1,\dots,x_N\in\cX\cup\cX^{-1}$. Using the first property of \cref{def:push-down} we have that
	\[
		\push_{\mathbf 1_Q}(\tilde\iota(w)) = \prod_{i=1}^N \push_{\mathbf 1_Q}(\tilde\iota(x_i)).
	\]
	In particular, if we set
	\[A=\max_{x\in\cX}\Area_\cR(\push_{\mathbf{1}_Q}(\tilde\iota(x))\cdot\ol{x}),\]
	then, by using at most $A\abs{w}= AN$ relations in $\relations$, we obtain $\push_{\mathbf 1_Q}(\tilde\iota(w))$ from $w$.
	On the other hand, $ \tilde\iota(w) $ is an element of $ \free\fpgens $ that represents the trivial element of $ G $, so we can write
	\[
		\tilde\iota(w) \eqfree \prod_{i=1}^{\alpha_G(N)} u_i C_i \inv u_i,
	\]
	with $ u_i \in \free\fpgens,\ \abs{u_i} \leq \radius{G}[\norma{\tilde \iota} \cdot N],\ C_i \in \fprels $, and thus also
	\[
		\push_{\mathbf{1}_Q}(\tilde\iota(w)) \eqfree \prod_{i=1}^{\alpha_G(N)} \push_{\mathbf{1}_Q}(u_i C_i \inv u_i).
	\]
	By \cref{lem:push-inverse}, we have that
	\begin{align*}
		\push_{\mathbf 1_Q}(u_i C_i \inv u_i) & = \push_{\mathbf 1_Q}(u_i) \cdot \push_{\widetilde\fibration(u_i)}(C_i) \cdot \push_{\widetilde\fibration(u_i)}(\inv u_i) \\
		                                      & = \push_{\mathbf 1_Q}(u_i) \cdot \push_{\widetilde\fibration(u_i)}(C_i) \cdot \left(\push_{\mathbf 1_Q}(u_i)\right)^{-1}.
	\end{align*}
	Now,
	\[\abs{\widetilde\fibration( u_i )}_{\widetilde\fibration(\fpgens )} \leq |u_i|\leq \radius{G}[\norma{\tilde \iota} \cdot N],\] so, by hypothesis, we can fill $ \push_{\widetilde \fibration(u_i)}(C_i) $ with at most $ f(\radius{G}[\norma{\tilde \iota} \cdot N]) $ relations of $ \krels $.

	Putting everything together, we obtain that it is possible to fill $ w $ with at most $\alpha_G(N) \cdot f(\radius G[\norma{\tilde \iota} \cdot N])+ AN$ relations in $ \krels $. This proves that $ \presentation\cX\krels $  is a presentation for $K$ by \cref{lem:presentation-area}, and that $\Dehn K \asympleq \alpha_G(N) \cdot f(\radius G)$.
\end{proof}

%% file: sections-uniform/notation.tex
Let $ \freecopy r1, \dots, \freecopy rn  $ denote $n \geq 3$ copies of the free group of rank $r \geq 2$.  Given a homomorphism
\[
	\fibration \colon \dirprod \to \ZZ^r
\]
such that the restriction to each factor is surjective, we are going to prove that the Dehn function of the kernel
\[
	\K=\ker\left(\fibration \right)
\]
is bounded above by $N^9$ for $n=3$, by $ N^5 $ for $n=4$, and by $N^4$ for $n \geq 5$.

In this section we describe in detail the strategy of the proof, which is the same for all $n \geq 3$ and consists of defining a push-down map for the sequence
\[
	\shortexactsequence{\K}{{\dirprod}}{\ZZ^r}
\]
as in \cref{sec:push-down}, and then proving that the area of the push-down of the relations of $\dirprod$ is bounded above by a polynomial; in this way we will be able to conclude via \cref{thm:push-down}.
To increase readability, we  postpone some technical computations that depend on the number $n$ of factors to \cref{sec:appendix-three-factors} (for $n=3$), \cref{sec:appendix-5} (for $n\geq 5$) and \cref{sec:appendix-four-factors} (for $n=4$); they are not important for understanding the main ideas of our proof and we recommend that the reader first reads this section, before looking at the appendices.

The strategy of the proof can be divided into the following steps:
\begin{enumerate}
	\item\label{enum:candidate-presentation} \textbf{Candidate presentation}: we find a set $\X$ of generators for the group $\K$, and we describe a family of relations $ \R $. We claim that $\presentation \X\R$ gives a presentation for $ \K $.
	\item\label{enum:def-push-down} \textbf{Push-down map}: we define a push-down map for the exact sequence above as described in \cref{sec:push-down}.
	\item\label{enum:doubling} \textbf{Doubling maps}: a \emph{doubling map} is a homomorphism of free groups that replaces each generator with a product of at most two generators (and which satisfies a certain symmetry condition, see \cref{def:simmetriz} below). We prove that $ \R $ is stable under doubling: if we apply a doubling map to a relation in $\R$, we obtain an element that still belongs to the normal subgroup generated by $\R$.
	\item\label{enum:power-relations} \textbf{Power maps}: for $N\in\bN$, a \emph{$N$-power map} is a homomorphism that replaces each generator $x$ with a power $x^{N_x}$ for some $\abs{N_x}\le N$ (and which satisfies a certain symmetry condition, see \cref{def:simmetriz} below). We prove that the area of \emph{$N$-power words}, obtained by applying a $N$-power map to relations in $\R$, is asymptotically bounded above by $ N^{d_n} $ (where $d_n = 7,3,2$ for $ n=3$, $ n=4 $, $ n \geq 5 $ respectively).
	\item\label{enum:thick-relation}\textbf{Thick relations}:
	applying to a relation in $\R$ a sequence of doubling maps, followed by a single power map, we obtain what we call a \emph{thick relation}.
	By employing \cref{lem:area-composition} (which tells us how to estimate the area of the composition of maps), we show that the area of these thick relations is also asymptotically bounded by $ N^{d_n} $.
	\item\label{enum:upper-bound-push}\textbf{Upper bound for the norm of the push-down map}: we prove that the push of the relations of $\dirprod$ may be filled in by a uniformly bounded number of thick relations, so the area of the push-down map is asymptotically bounded by $ N^{d_n} $.
	\item\label{enum:applic-push-thm} \textbf{Conclusion}: by employing \cref{thm:push-down}, we prove that $\presentation \X \R$ is a finite presentation for $ \K $, and that its Dehn function is asymptotically bounded by $ N^{{d_n}+2} $.
\end{enumerate}

\subsection{Candidate presentation}\label{sec:candidate-presentation}

Let $n \geq 3$ and $r \geq 2$ be two integers.
For $ \alpha \in \range n $, we consider a nonabelian free group of rank $r$
\[\freecopy r \alpha = \left\langle\fpgen 1 \alpha,\dots, \fpgen r \alpha\right\rangle;\]
we denote by $ \fpgens[\alpha] $ the ordered tuple of generators
\[\Aii \alpha =\left(\aii \alpha1,\ldots,\aii \alpha r\right).\]

Then we define
\[\A=\bigcup_{\alpha=1}^n \Aii \alpha,\]
where, in the expression, we forget about the ordering.

We define
\[
	\K=\ker\left(\fibration \colon \dirprod \to \ZZ^r\right),
\]
where $\psi$ is the surjective morphism sending $\fpgen j \alpha$ to the $j$-th basis vector $e_j$.
Note that a different choice of a map that is surjective on each factor produces a kernel isomorphic to $\K$ (see for example \cite{Dison-08}).

Denote by %
\[\prii \alpha \colon \dirprod \to \Fii \alpha\]
the projection onto the $\alpha$-th factor. With a little abuse of notation, we also denote by $ \prii \alpha $ the restriction $\prii \alpha\colon \K \to \Fii \alpha$ to the kernel.
Moreover, in the following we often regard $\Fii \alpha $ as a subgroup of $\freecopy r1 \times \dots \times \freecopy rn$ (with respect to the natural inclusion).

\subsubsection{A generating set for $ \K $}
For $ \alpha \in \range {n-1} $ and $ i \in \range r $, we set
\[\xii \alpha i = \fpgen i\alpha  \invfpgen in \in \K,\]
and we denote by $ \Xii \alpha $ the ordered tuple
\[
	\Xii \alpha=\left(\xii \alpha 1,\dots,\xii \alpha r\right).
\]

As stated in the next lemma, the elements in
\[\X=\bigcup_{\alpha = 1}^{n-1}\Xii \alpha\]
generate the group $\K$, where we forget about the ordering in $\Xii \alpha$ in this expression.

\begin{lemma}\label{lem:generator-kernel}
	The set $\X$ generates $\K$.
\end{lemma}

\begin{proof}
	Let us fix a lift
	\[ \tilde\iota \colon  \free\X \to \free\A \]
	of the canonical inclusion $\iota \colon \K \to \dirprod$
	sending $ \xii \alpha i$  to $\fpgen i \alpha \invfpgen in $, and denote by
	\[\unor = (1, \dots, 1)\]
	the ordered tuple containing $r$ copies of the trivial element.

	Fix an element $g \in \K \subset \dirprod$ and let $w=w\left(\Aii 1,\dots, \Aii n\right)$ be a word representing $g$. Set $ w' = w \left(\Xii 1,\dots,\Xii {n-1}, \unor \right) $ and
	\[v=v\left(\Aii 1,\dots, \Aii n\right)\coloneq w\left(\Aii 1,\dots, \Aii n\right)\cdot \tilde\iota \left(w\left(\Xii 1,\dots,\Xii {n-1}, \unor\right)^{-1}\right)\]
	We have that $v$ represents an element of $\dirprod$ whose projections to the first $(n-1)$ factors is trivial, and which belongs to the subgroup $\K$. It follows that $v$ represents an element belonging to the commutator subgroup $ \commsub{\freecopy rn}$. Therefore we can write
	\[v=_{\K} \prod_{i = 1}^m
		\inv z_i\left(\Aii n\right)
		\left[\aii n{s_i}, \aii n{t_i}\right]
		z_i\left(\Aii n\right),\]
	where $1\leq s_i,t_i\leq n$ and $z_i\left(\Aii n\right)$ are words in $\Fii n$.

	The word
	\[
		v'= \prod_{i = 1 }^m
		\inv z_i\left(\mXii 1\right)
		\left[ \mxii 1{s_i},\mxii 2{t_i}\right]
		z_i\left(\mXii 1\right),
	\]%
	where $ \mXii 1 $ denotes the tuple $ (\mxii 11, \dots, \mxii 1r) $, represents the same element as $v$ (as can be easily checked by projecting on each factor $\Fii \alpha$), and thus $w$ represents the same elements as $v'\cdot w'$, that is a word written in the alphabet $\X$. This concludes the proof.
\end{proof}

\subsubsection{Candidate relations}

Consider the following sets of relations:
\begin{align*}
	\simplecomms^n & \coloneq \set*{\left[x_i, y_i \right] \Suchthat
	{\begin{aligned}
				  & x_i=\kgen i\alpha ,\ y_i=\kgen i\beta               \\
				  & i \in \range r,\ \alpha \neq \beta \in \range {n-1}
			 \end{aligned}}}                                                                                       \\
	\comms^n       & \coloneq \set*{\left[x_i, y_j\ol{z}_j \right] \Suchthat
		\begin{aligned}
			 & x_i=\kgen i\alpha ,\ y_i=\kgen i\beta ,\ z_i=\kgen i\gamma        \\
			 & i \neq j \in \range r,                                            \\
			 & \alpha, \beta, \gamma  \in \range {n-1} \text{ pairwise distinct}
		\end{aligned}
	}                                                                                                                                             \\
	\swaps^n       & \coloneq \set*{[x_i^\epsilon, y_j^\delta ][x_j^\delta, y_i^\epsilon ]  \Suchthat
		\begin{aligned}
			 & x_i = \kgen i\alpha ,\ y_i = \kgen i\beta                               \\
			 & i \neq j \in \range r,                                                  \\
			 & \alpha \neq \beta  \in \range {n-1} ,\ \epsilon, \delta \in \set{\pm 1}
		\end{aligned}
	}                                                                                                                                             \\
	\triplecomms^n & \coloneq \set*{\left[x_i, [y_j^\epsilon, x_k^\delta\ol{y}_k^\delta ]\right]  \Suchthat
		\begin{aligned}
			 & x_i = \kgen i\alpha ,\ y_i = \kgen i\beta                               \\
			 & i, j, k \in \range r \text{ pairwise distinct},                         \\
			 & \alpha \neq \beta  \in \range {n-1} ,\ \epsilon, \delta \in \set{\pm 1}
		\end{aligned}
	}                                                                                                                                             \\
	\quadcomms^n   & \coloneq \set*{\left[ [ x_i^\epsilon, x_k^\delta\inv y_k^\delta ] , [ y_j^\sigma, x_h^\tau\inv y_h^\tau ] \right]  \Suchthat
		\begin{aligned}
			 & x_i=\kgen i\alpha ,\ y_i=\kgen i\beta              \\
			 & i, j, k, h \in \range r \text{ pairwise distinct}, \\
			 & \alpha \neq \beta            \in \range {n-1},     \\
			 & \epsilon, \delta,\sigma,\tau \in \set{\pm 1}
		\end{aligned}
	}
\end{align*}

For every $r \geq 2, n \geq 3$, all the words inside these sets represent the trivial element of $ \K $. Note that for small values of $ n $ and $r$, some of these sets may be empty, due to the requirement that indices should be distinct (if $n=3$ then $ \comms^n $ is empty, if $r=2$ then $\triplecomms^n,\quadcomms^n$ are empty, if $r=3$ then $\quadcomms^n$ is empty). Moreover, some of these sets become redundant when $n$ is big enough: when $n=4$ the relations in $\swaps^n,\quadcomms^n$ can be obtained from $\simplecomms^n,\comms^n, \triplecomms^n$, and when $ n \geq 5 $ the relations in $ \swaps^n, \triplecomms^n, \quadcomms^n$ can be obtained from the relations in $ \simplecomms^n $ and $ \comms^n $ (this will become clear later).

We sum up the relations as follows:

\begin{align*}
	\mathcal R_r^3 & = \simplecomms^3 \cup \swaps^3 \cup \triplecomms^3 \cup \quadcomms^3                           \\
	\mathcal R_r^4 & = \simplecomms^4 \cup \comms^4 \cup \triplecomms^4                                             \\
	\mathcal R_r^n & = \simplecomms^n \cup \comms^n                                       & \text{ for $n \geq 5$.}
\end{align*}

We will prove that the group $\K$ is presented by
\[\presentation \X \R  \]
for every integer $n \geq 3$ and $r \geq 2$.%

\begin{notation}
	From now on until the end of the section we fix $n \geq 3$. This allows us to drop the superscript $n$, so that the symbols $\mathcal A_r, \mathcal R_r, \mathcal R_{r,i}$ will denote respectively $\A,\R,\mathcal R_{r,i}^n$, for any integer $r\ge 2$ and $i \in \range 5$.
\end{notation}

\RenewDocumentCommand{\R}{O{\r}}{{{\mathcal R}_{#1}}}
\RenewDocumentCommand{\Rthick}{O{\m} O{\r}}{{{\mathbf{\mathcal R}}_{#2}({m}})}
\RenewDocumentCommand{\A}{O{\r}}{{\mathcal{A}_{#1}}}
\RenewDocumentCommand{\X}{O{\r}}{{\mathcal{X}_{#1}}}

\subsection{The push-down map}\label{sec:push-down-for-our-kernels}

We define a \emph{push-down map} for the sequence
\[
	\shortexactsequence{\K}{{\dirprod}}{\ZZ^r}
\]
as described in \cref{sec:push-down}, using the construction described by \cref{lem:push-down-existence}.

We start by defining
\[u_\bq=\left(a_1^{(n)}\right)^{q_1}\cdots\left(a_r^{(n)}\right)^{q_r}\in F(\A)\]
for every $\bq\in \bZ^r$.

For every $a=\aii \alpha j\in \A$ and $\bq\in \ZZ^r$, we now define $z_{\bq,a}:=\push_\bq(a) \in F(\X)$ such that $\tilde\iota(z_{\bq,a})$ represents the element $u_{\bq}a \inv u_{\bq \cdot \widetilde\psi(a)}$ by setting
\[
	z_{\bq,a}= \left(\ol{x}^{(\sigma(\alpha))}_1\right)^{q_1}\cdots \left(\ol{x}^{(\sigma(\alpha))}_{j-1}\right)^{q_{j-1}} \cdot x^{(\alpha)}_j\cdot \left(x^{(\sigma(\alpha))}_{j-1}\right)^{q_{j-1}}\cdots \left(x^{(\sigma(\alpha))}_1\right)^{q_1}
\]
if $\alpha\in \range {n-1}$, where $\permut\colon\range {n-1}\to \range {n-1}$ is a fixed fixed-point-free auxiliary map, and
\[
	z_{\bq,a}=\left(\ol{x}^{(2)}_1\right)^{q_1}\cdots \left(\ol{x}^{(2)}_r\right)^{q_r}\cdot
	\ol{x}^{(1)}_j\cdot
	\left(x^{(2)}_r\right)^{q_r}\cdots \left(x^{(2)}_j\right)^{q_j}\cdot
	x^{(1)}_j\cdot
	\left(x^{(2)}_{j-1}\right)^{q_{j-1}}\cdots \left(x^{(2)}_1\right)^{q_1}
\]
if $\alpha=n$. Indeed, if $\alpha\in \range {n-1}$, the $\aii {\sigma(\alpha)} i$ commute with all the letters appearing in the expression $u_{\bq}a\inv u_{\bq \cdot \widetilde\psi(a)}$ (since $\sigma(\alpha)\neq \alpha$) and we can employ the alphabet $\Aii {\sigma(\alpha)}$ to rewrite $u_{\bq}a\inv u_{\bq \cdot \widetilde\psi(a)}$ as the above word in the letters $\xii {\sigma(\alpha)} i = \aii {\sigma(\alpha)} i \maii ni$, and a similar argument applies if $\alpha=n$.

By \cref{lem:push-down-existence} there is a unique extension of such map to a push-down map $\push\colon \ZZ^r\times F(\A)\rar F(\X)$. This extension can be computed from the above definitions using the property that $\push_\bq(w \cdot w') = \push_\bq(w) \cdot \push_{\bq \cdot \widetilde\fibration(w)}(w')$.

%% file: sections-uniform/general-case.tex
\newcommand{\dtab}{0.15cm}

\subsection{Doubling maps}

We consider homomorphisms between the free groups $F(\cX_r)$ satisfying a certain symmetry condition (see \cref{def:simmetriz} below), related to the symmetry of our (candidate) presentation for $K^n_r(r)$. By applying these symmetric homomorphisms to the relations in $\cR_r$ we can get large families of elements of $F(\cX_r)$, which still belong to the normal subgroup generated by $\cR_r$. In order to prove this, we first deal with homomorphisms of norm $1$ (\cref{prop:symmetry} below), then with homomorphisms of norm $2$ (\cref{prop:homo-norm2} below), and finally we provide a statement for generic homomorphisms (\cref{prop:upperbound-with-norm-homo} below).

Recall that $\absfree k$ denotes the free group with basis $\afgen 1,\ldots ,\afgen k$.
\begin{definition}\label{def:simmetriz}
	Let $r,r'\ge1$ be two integers. Given a homomorphism $\phi \colon \absfree r \rar \absfree{r'}$ of free groups, define the homomorphism
	\[
		\simmetriz \phi\colon F\left(\Xii 1,\ldots, \Xii {n-1}\right) \rar F\left(\Xii 1[r'],\ldots, \Xii {n-1}[r']\right)
	\]
	as follows: for $i \in \range r$, if $\phi(\afgen i)=w_i(\afgen 1,\ldots,\afgen{r'})$, then we set
	\[\simmetriz\phi\left(\xii \alpha i\right)=w_i\left(\xii \alpha 1,\ldots,\xii \alpha {r'}\right).\]
\end{definition}

Let us recall from \cref{def:homo-norm} that the norm $\|\phi\|$ of a homomorphism of free groups $\phi \colon \absfree r \rar \absfree{r'}$ is the maximum of the lengths of the elements $\phi(\afgen 1),\dots ,\phi(\afgen r)$ as reduced words in $\afgen 1,\dots,\afgen {r'}$ (and their inverses). Notice that $\norma{\simmetriz\phi}=\norma{\phi}$.

\begin{proposition}\label{prop:symmetry}
	There exists a constant $A_1>0$ such that the following holds:
	for all integers $r,r'\ge1$ and every homomorphism of free groups $\phi \colon \absfree r \rar \absfree{r'}$  with $\norma{\phi}\le 1$, it holds that
	\[
		\Area_{\R [r']}\left(\simmetriz\phi(\R) \right)\le A_1.
	\]
\end{proposition}
\begin{proof}
	Follows from \cref{prop:symmetry-3,prop:ass:simmetry-4,prop:ass:simmetry-n-big}.
\end{proof}

\begin{proposition}
	\label{lem:basic-doubling}
	There exists a constant $A_2>0$ such that the following holds:
	for every integer $r\ge1$, let $\rho_r \colon \absfree r\rar \absfree{r+1}$ be the homomorphism given by $\rho_r(\afgen 1)=\afgen 1\afgen 2$ and $\rho_r(\afgen i)=\afgen {i+1}$ for $i=2,\dots ,r$. Then we have
	\[
		\Area_{\R[r+1]}(\simmetriz\rho(\R))\le A_2.
	\]
\end{proposition}
\begin{proof}
	Follows from \cref{lem:basic-doubling-3,prop:ass:doubling-4,prop:ass:doubling-more-factors}.
\end{proof}

\begin{lemma}\label{lem:raddoppio-4-var}
	Consider the homomorphism $\theta \colon \absfree{4}\rar \absfree{8}$ given by $\theta(\afgen i)=\afgen {2i-1}\afgen {2i}$ for $i=1,2,3,4$. Then we have
	\[
		\Area_{\R[8]}\left(\simmetriz{\theta}(\R[4])\right)\le A_3
	\]
	for some constant $A_3$.
\end{lemma}
\begin{proof}
	We can easily write
	\[
		\theta= \mu_4\circ\rho_7\circ\mu_3\circ\rho_6\circ \mu_2\circ\rho_5\circ\mu_1\circ\rho_4
	\]
	where $\rho_4,\dots, \rho_7$ are the homomorphisms defined in \cref{lem:basic-doubling} and $\mu_1,\dots,\mu_4$ are homomorphisms of norm $1$. We use \cref{lem:area-composition,prop:symmetry,lem:basic-doubling} to estimate
	\[
		\Area_{\R[8]}\left(\simmetriz{\theta}(\R[4])\right)=\Area_{\R[8]}\left(\simmetriz{\mu_4} \simmetriz{\rho_7} \simmetriz{\mu_3} \simmetriz{\rho_6} \simmetriz{\mu_2} \simmetriz{\rho_5} \simmetriz{\mu_1} \simmetriz{\rho_4}(\R[4])\right)\le A_3\\
	\]
	for some constant $A_3$.
\end{proof}

\begin{proposition}[Doubling]\label{prop:homo-norm2}
	Let $r,r'\ge1$ be integers and let $\phi\colon \absfree{r}\rar \absfree{r'}$ be a homomorphism with $\norma{\phi}\le 2$. Then we have
	\[
		\Area_{\R[r']}(\simmetriz\phi(\R))\le A
	\]
	for some constant $A$ independent of $r,r',\phi$.
\end{proposition}
\begin{proof}
	Let us first prove the result for the homomorphism
	\[\psi \colon \absfree{r}\rar \absfree{2r}\]
	given by $\psi(\afgen i)=\afgen{2i-1}\afgen{2i}$ for $i \in \range r$.
	Let $R\in\R $
	and let $1\le i_1<\ldots<i_k\le r$ be the subscripts involved in the relation $R$, where $1\le k\le 4$. Consider the morphism
	\[\mu \colon \absfree{r}\rar \absfree{4}\]
	given by $\mu(\afgen {i_j})=\afgen j$ for $j=1,\ldots,k$ and $\mu(\afgen i)=1$ for $i\not=i_1,\ldots,i_k$; observe that $\norma{\mu}\le 1$. Consider the morphism
	\[\lambda \colon \absfree{8}\rar \absfree{2r}\]
	given by $\lambda(\afgen {2j-1})=\afgen {2i_j-1}$ and $\lambda(\afgen {2j})=\afgen {2i_j}$ for $j=1,\ldots,k$ and $\lambda(\afgen i)=1$ otherwise; observe that $\norma{\lambda}\le 1$.
	By definition, %
	we have that \[\simmetriz{\psi}(R)=\simmetriz{\lambda} \left(\simmetriz{\theta}\left(\simmetriz{\mu}(R)\right)\right)\]
	where $\theta \colon \absfree{4}\rar \absfree{8}$ is the homomorphism of  \cref{lem:raddoppio-4-var}. Therefore, by  \cref{lem:area-composition} we have
	\begin{align*}
		\Area_{\R[2r]}\left(\simmetriz\psi(R)\right)\le &
		\Area_{\R[2r]}\left(\simmetriz\lambda(\R[8])\right)\Area_{\R[8]}\left(\simmetriz\theta(\R[4])\right)\Area_{\R[4]}\left(\simmetriz\mu(\R)\right) \\
		\le                                             & (A_1)^2A_3
	\end{align*}
	where $A_1,A_3$ are the constants given by \cref{prop:symmetry,lem:raddoppio-4-var} respectively.
	It follows that
	\[\Area_{\R[2r]}\left(\simmetriz\psi(\cR_r)\right)\le (A_1)^2A_3.\]

	For the general case, it is enough to observe that we can decompose any homomorphism
	\[
		\phi \colon \absfree{r}\rar \absfree{r'}
	\]
	with $\|\phi\|\leq 2$ as a composition $\phi=\eta\circ\psi$, where
	$\psi \colon \absfree{r}\rar \absfree{2r}$ is the homomorphism defined above and $\eta \colon \absfree{2r}\rar \absfree{r'}$ is a homomorphism with $\norma{\eta}\le 1$. Thus, by \cref{lem:area-composition} we have
	\[		\Area_{\R[r']}\left(\simmetriz\phi(\R)\right)\le\Area_{\R[r']}\left(\simmetriz\eta(\R[2r])\right)\Area_{\R[2r]}\left(\simmetriz\psi(\R) \right)\le (A_1)^3A_3
	\]
	(using \cref{prop:symmetry} once again) and the conclusion follows.
\end{proof}

As a consequence, we get an estimate of the area of a word obtained by applying any homomorphism $\phi$ to a relation in $\R$ in terms of the norm $\|\phi\|$.
\begin{proposition}\label{prop:upperbound-with-norm-homo}
	Let $r,r'\ge1$ be integers and consider a homomorphism $\phi \colon \absfree{r}\rar \absfree{r'}$  of free groups. Then we have
	\[
		\Area_{\R[r']}\left(\simmetriz\phi(\R)\right)\le B\norma{\phi}^B,
	\]
	for some constant $B$ independent of $r,r',\phi$.
\end{proposition}
\begin{proof}
	Any homomorphism $\phi\colon \absfree{r}\rar \absfree{r'}$ can be decomposed as a composition of at most $\intsup{\log_2(\norma{\phi})}$ homomorphisms of norm $2$ and a single homomorphism of norm $1$ (thus of norm $\leq 2$). We can then estimate $\Area_{\R[r']}\left(\simmetriz\phi(\R)\right)$ by using \cref{lem:area-composition,prop:homo-norm2}:
	\[
		\Area_{\R[r']}\left(\simmetriz\phi(\R)\right)\le A^{\intsup{\log_2(\norma{\phi})}}\cdot A\le A^2 A^{\log_2(\norma{\phi})}= A^2\norma{\phi}^{\log_2(A)}
	\]
	where $A$ is the constant of \cref{prop:homo-norm2}. We conclude by setting $B =\max \{A^2,\log_2 A \}$.
\end{proof}

\subsection{Power maps}
A \emph{$N$-power map} is an endomorphism $\widehat{\phi}$ of $F(\cX_r)$, obtained as in \cref{def:simmetriz} from a morphism $\phi$ of $\absfree{r}$, and sending each generator $\xii \alpha i \in \X$ to a power $(\xii \alpha i)^{N_i}$, for some integer $N_i$ with $|N_i|\leq N$. By applying a $N$-power map to a relation in $\R$ we obtain another element of $F(\cX_r)$, which belongs to the normal subgroup generated by $\R$. By \cref{prop:upperbound-with-norm-homo} we can immediately obtain a bound, polynomial in $N$, on the area of this new element. However, the resulting exponent is quite big, and thus we prefer to provide an independent and stronger bound here.

\begin{proposition}\label{prop:powers}
	There is a constant $C>0$ such that the following happens: let $r \geq 2$ and $N_1,\ldots,N_r$ be integers, and consider the homomorphism
	\[\omega=\omega_{N_1,\ldots,N_r}\colon \absfree{r}\rar \absfree{r}\]
	given by $\omega(\afgen i)=\afgen i^{N_i}$ for $i \in \range r$. Then,
	\[
		\Area_{\R}\left(\simmetriz\omega(\R)\right)\le C(\max\{\abs {N_1},\ldots,\abs {N_r}\})^{d_n},
	\]
	where $d_3 = 7$, $d_4=3$ and $d_n=2$ for $n\ge 5$.
\end{proposition}

\begin{proof}
	Follows from \cref{ass:powers-3,prop:powers-4,prop:ass:powers-more-factors}.
\end{proof}

\subsection{Thick relations} A \emph{thick relation} is an element of $F(\cX_r)$ obtained by applying a certain (symmetric) homomorphism to a relation in $\R$. Once again, thick relations are words belonging to the normal subgroup in $\free \X$ generated by $\R$. The reason why we are interested in these relations is that we can use them to estimate the area of the push of the relations of $\dirprod$ (see \cref{prop:thick}).

\begin{definition}\label{def:thickening-map}
	For $\bq=(q_1,\ldots,q_r),\bq'=(q_1',\ldots,q_r')\in\bZ^r$, define the homomorphism
	\[
		\kappa_{\bq,\bq'} \colon \absfree{r+2}\rar \absfree{r}
	\]
	given by $\afgen i\mapsto \afgen i$ for $i \in \range r$ and $\afgen {r+1}\mapsto \afgen 1^{q_1}\cdots \afgen r^{q_r}$ and $\afgen {r+2}\mapsto \afgen 1^{q_1'}\cdots \afgen r^{q_r'}$.
\end{definition}

For $\bq\in\bZ^r$ we denote by $\abs{\bq}=\max\{\abs{q_1},\ldots,\abs{q_r}\}$.
\begin{definition}[Thick relations]
	For $m>0$, we define
	\[
		\Rthick \coloneqq\bigcup_{\|\phi\|\leq 1}\ \bigcup_{\abs{\bq},\abs{\bq'}\leq m+1} \simmetriz{\kappa_{\bq,\bq'}}\left(\simmetriz\phi(\R[r +2])\right)
	\]
	where we consider homomorphisms  $\phi\colon \absfree{r+2}\to \absfree{r+2}$ and tuples $\bq,\bq'\in \bZ^r$.
\end{definition}

\begin{proposition}\label{prop:bound-for-thick-relation}
	Let $n\ge 3$, $r\ge2$, $m\ge1$ be integers. Then, we have
	\[
		\Area_{\R}(\Rthick)\le D_rm^{d_n}
	\]
	for some constant $D_r$ depending on $r$ (but independent of $n,m$) and for $d_3= 7$, $d_4= 3$ and $d_n= 2$ for $n\ge 5$.
\end{proposition}
\begin{proof}
	We fix a morphism $\phi\colon \absfree{r+2}\to\absfree{r+2}$ with $\norma{\phi}\leq 1$, two tuples  $\bq=(q_1,\ldots,q_r),\bq'=(q_1',\ldots,q_r')\in \ZZ^r$ and consider the homomorphism
	$\kappa_{\bq,\bq'} \colon \absfree{r+2}\rar \absfree{r}$ as in \cref{def:thickening-map}.
	Let
	\[\alpha \colon \absfree{r+2}\rar \absfree{3r}\]
	defined by $\alpha(\afgen i)=\afgen i$, for $i \in \range r$, $\alpha(\afgen {r+1})=\afgen {r+1}\cdots \afgen {2r}$, and $\alpha(\afgen {r+2})=\afgen {2r+1}\cdots \afgen {3r}$. Consider
	\[\omega=\omega_{1,\ldots,1,q_1,\ldots,q_r,q_1',\ldots,q_r'}\colon \absfree{3r}\rar \absfree{3r}\]
	as in \cref{prop:powers}, that is, $\omega(\afgen i)=\afgen i$, $\omega(\afgen {r+i})=\afgen {r+i}^{q_i}$ and $\omega(\afgen {2r+i})=\afgen {2r+i}^{q_i'}$ for $i \in \range r$, and
	\[\beta \colon \absfree{3r}\rar \absfree{r}\]
	given by $\beta(\afgen i)=\beta(\afgen {r+i})=\beta(\afgen {2r+i})=\afgen i$ for $i \in \range r$.

	By direct check we have $\kappa_{\bq,\bq'}\circ \phi=\beta\circ\omega\circ\alpha\circ \phi$, so by \cref{lem:area-composition} we obtain, for every $R \in \R$
	\begin{align*}
		\Area_{\R}(\Rthick)\le &
		A_1 \cdot C(\max\{1,\abs{q_1},\ldots,\abs{q_r},\abs{q_1'},\ldots,\abs{q_r'} \})^{d_n}\cdot B\norma{\alpha}^B\cdot A_1 \\
		\le                    & A_1^2BCr^B(m+1)^{d_n}
	\end{align*}%
	where $A_1,B,C$ are the constants of \cref{prop:symmetry,prop:powers,prop:upperbound-with-norm-homo}. The conclusion follows.
\end{proof}

We are able to fill pushes of relations of $\dirprod$ by using a uniformly bounded number of thick relations.

\begin{proposition}\label{prop:thick}
	There is a constant $E$ (independent of $r$) such that, for every $\bq\in\bZ^r$ with $\abs{\bq}\le m$, $\alpha \neq \beta \in \range{n}$ and $1\le i,j\le r$, we have the following:
	\[\Area_{\Rthick}\left(\push_\bq \left(\left[\aii \alpha i,\aii \beta j \right]\right)\right)\le E . \]
\end{proposition}

\begin{proof}
	The assertion follows from \cref{prop:ass:thick-3,prop:ass:thick-4,prop:ass:thick-more-factors}.
\end{proof}

\subsection{Conclusion}

We are finally able to give an upper bound to the Dehn function of $\K$. Note that it only depends on the constant $d_n$ defined in \cref{prop:powers}.
\begin{theorem}\label{thm:if-assumption-then-bound-Dehn-function}
	The group $\K$ is presented by \(\presentation \X \R\) and
	\[
		\delta_{\K} (N)\asympleq N^{d_n+2}
	\]
	where $d_3 = 7$, $d_4=3$ and $d_n=2$ for $n\ge 5$.
\end{theorem}

\begin{proof} %
	Recall that for
	\[\cC_r =\left \{ \left[\aii \alpha i,\aii \beta j \right] \Suchthat i, j \in \range r, \alpha \neq \beta \in \range{n}\right \},\]%
	$\dirprod$ is finitely presented by $\presentation \A  {\cC_r} $.
	By \cref{lem:area-composition} we have, for every $\bq\in \bZ^r$ with $\abs{\bq}\leq m$, that
	\[
		\Area_{\R}(\push_\bq(\cC_r))\leq
		\Area_{\R}(\Rthick)\cdot \Area_{\Rthick}(\push_\bq(\cC_r))\leq
		E \cdot D_r m^{d_n},
	\]
	where $E, D_r,d_n$ are the constants given in \cref{prop:thick,prop:bound-for-thick-relation}.

	But $\delta_{\dirprod}(N)\asympleq N^2$, and thus by \cref{prop:linear-radius} we have that there is an area-radius pair $(\alpha,\rho)$ for $\dirprod$ with $\alpha(N)\asymp N^2$ and $\rho\asymp N$. The conclusion follows by \cref{thm:push-down}.
\end{proof}

\begin{remark}
	For $ \K[3][2] $, resp.~ $ \K[3][3]$, Dison proved stronger upper bounds of $ N^6 $, resp.~ $ N^8 $, see \cite[Theorem 13.3(2)]{Dison-08}. For $\K[3][2]$ we will prove in a subsequent paper \cite{PreciseComputations-25} that in fact the Dehn function is quartic. Moreover, a careful analysis of our proof of \cref{thm:if-assumption-then-bound-Dehn-function} for the special case of $\K[3][3]$ allows us to recover Dison's bound for this example, using the fact that one family of relations appearing in the proof of \cref{thm:if-assumption-then-bound-Dehn-function} is empty for $r = 3$.
\end{remark}

%% file: sections-uniform/residually-free-groups.tex
\section{Generalisations to other SPFs and residually free groups}\label{sec:equivalence-of-conjectures}
\label{sec:residually-free-case}

In this section we will first explain how to generalise the upper bounds on the Dehn functions of the $K_r^n(r)$ from \cref{sec:general-strategy} to a larger class of SPFs by deducing inequalities between the Dehn functions of groups in this class. This will allow us to prove \cref{main:3-factors,main:type-F-n-1}. We will then explain how, more generally, Bridson's \cref{conj:Bridson-RF} about residually free groups can be reduced to Bridson's \cref{conj:Bridson-SPF} about SPFs, proving \cref{main:equivalence-of-conjectures} and highlighting the importance of understanding Dehn functions of SPFs.

\subsection{Free groups of different ranks}\label{sec:free-rank}
Given two fixed positive integers $r\ge 2$ and $n\geq 3$, we now consider the product of $n$ free groups $ F_{m_1}, \dots, F_{m_n}$, with possibly different ranks $m_i\geq r$. As before, let $\psi \colon F_{m_1}\times \dots \times F_{m_n} \to \ZZ^r$ be a morphism which is surjective on each factor. Denote its kernel by
\[
	K_{m_1, \dots, m_n}(r)= \ker\left(\psi \colon F_{m_1}\times \dots \times F_{m_n} \to \ZZ^r\right).
\]

We are going to prove that the group $K_{m_1, \ldots, m_n}(r)$ has the same Dehn function as the group $\K$, for every positive integers $n\geq 3$ and $r\geq 2$ and for any choice of $m_1,\ldots, m_n$, with $m_i\geq r$ for all $i\in\range n$.
In order to do that, we are going to prove two preliminary lemmas. The first one  allows some comparisons when the ranks increase, and the second one gives a comparison bound when the ranks decrease.

\begin{lemma}\label{prop:go-up}
	There are infinitely many positive integers $m>r$ such that the Dehn function of the group $K^n_m(r)$ satisfies
	\[
		\Dehn{K^n_m(r)} \asymp \Dehn{K^n_r(r)}.
	\]
\end{lemma}
\begin{proof}
	Let $H$ be an index $k$ subgroup of $F_r$, then it is a well-known fact that $H$ is isomorphic to a free group of rank $m=m_{k,r}=k(r-1)+1$.
	The subgroup $H \times \dots \times H < F_r \times \dots \times F_r$ has finite index $k^n$. Let $\psi'$ be the restriction of the morphism $\psi$ to the subgroup $H \times \dots \times H$, and denote by $K=\ker \psi'$ and $S= \operatorname{im} \psi'<\bZ^r$ the kernel and the image of $\psi'$.
	\[
		\begin{tikzcd}
			& K \arrow[r, hookrightarrow] \arrow[d, hookrightarrow, dashed] & H\times \dots \times H \arrow[d, hookrightarrow] \arrow[r, tail, twoheadrightarrow, "\psi'"] & S \arrow[d, hookrightarrow] \\
			& \K \arrow[r, hookrightarrow] & F_{r}\times \dots \times F_{r} \arrow[r, twoheadrightarrow, "\psi"] & \ZZ^r%
		\end{tikzcd}
	\]
	By standard diagram chasing, one can check that the kernel $K$ is included in the kernel $\K$; in particular, $K=\K\cap (H \times \dots \times H)$ has finite index in $\K$. Moreover, $S$ is a finite index subgroup of $\ZZ^r$, and it is therefore isomorphic to $\ZZ^r$ itself. Combining this information with the fact that $H\cong F_m$, we deduce that $K$ is isomorphic to the group ${K^n_m(r)}$.

	In conclusion, we have found a finite index subgroup isomorphic to $K^n_m(r)$ inside $\K$, so these two groups share the same Dehn function. Different choices of $k$ yield different values of $m$, thus the proposition is proved.
\end{proof}

Recall that a morphism $r\colon G\to H$ of groups is called a retraction if there is an injective morphism $\iota\colon H\to G$ such that $r\circ \iota={\rm id}_H$. We will require the following well-known fact about Dehn functions and retractions:
\begin{lemma}\label{lemma:retr}
	Let $r\colon G\to H$ be a retraction of groups. Then
	$$\Dehn{H} \asympleq \Dehn{G}.$$
\end{lemma}

\begin{lemma}\label{prop:go-down}
	If $m_1, \dots, m_n$ and $m_1', \dots, m_n'$ are two $n$-tuples of positive integers satisfying $m_i'\geq m_i \geq r$%
	, then
	\[
		\Dehn{K_{m_1, \dots, m_n}(r)} \asympleq \Dehn{K_{m_1', \dots, m_n'}(r)}.
	\]
\end{lemma}

\begin{proof}
	Since the morphism $\psi \colon F_{m_1}\times \dots \times F_{m_n} \to \ZZ^r$ is surjective on every factor, we can suppose (up to a change of basis on the free groups) that for every $ \alpha \in \range n $ the $i$-th generator $a_i^{(\alpha)}$ of $F_{m_\alpha}$ is mapped by $\psi$ onto the $i$-th generator $e_i$ of $\ZZ^r$ when $i \in\range r$, and onto the trivial element if $ i \in \range[r+1]{m_\alpha} $.  Similarly, we can suppose that we have chosen a morphism $\psi' \colon F_{m_1'}\times \dots \times F_{m_n'}\to \ZZ^r$ that behaves the same way on the generators $a_{j}^{(\alpha)}$ of the $F_{m'_\alpha}$.

	For every $\alpha\in \{1,\dots,n\}$ we define a retraction $F_{m'_\alpha} \to F_{m_\alpha}$ of the natural inclusion $F_{m_\alpha}\to F_{m'_\alpha}$ by mapping $\aii \alpha i$ to $\aii \alpha i$ if $i \leq m_\alpha$, and to $1$  if $i>m_\alpha$. This induces a retraction $g \colon F_{m'_1}\times \dots \times F_{m'_n} \to F_{m_1}\times \dots \times F_{m_n}$ that fits into the commutative diagram %
	\[
		\begin{tikzcd}
			& 1 \arrow[r] & K_{m_1', \dots, m_n'}(r)  \arrow[d,"\bar{g}", dashed]\arrow[r, hookrightarrow] & F_{m'_1}\times \dots \times F_{m'_n} \arrow[d, rightarrow ,"g"] \arrow[r, twoheadrightarrow, "\psi'"] & \ZZ^r \arrow[d, leftrightarrow, "="] \arrow[r] & 1 \\
			& 1 \arrow[r] & K_{m_1, \dots, m_n}(r) \arrow[r, hookrightarrow] & F_{m_1}\times \dots \times F_{m_n} \arrow[r, twoheadrightarrow, "\psi"] & \ZZ^r \arrow[r] & 1.%
		\end{tikzcd}
	\] %

	The commutativity of the diagram implies that the image of the restriction of the morphism $g$ to the kernel $K_{m_1', \dots, m_n'}(r)$ is precisely the kernel $K_{m_1, \dots, m_n}(r)$; and therefore $g$ restricts to a retraction $\overline{g} \colon K_{m_1', \dots, m_n'}(r) \to K_{m_1, \dots, m_n}(r)$. Thus, the assertion follows from \cref{lemma:retr}.
\end{proof}

\begin{theorem}\label{thm:main-free-groups-different-ranks}
	The group $K_{m_1, \dots, m_n}(r)$ has the same Dehn function as the group $\K$, for all integers $n\geq 3$ and $r\geq 2$, and every $n$-tuple of integers $m_1,\dots, m_n$ such that $m_i\geq r$ for each $i$.%
\end{theorem}
\begin{proof}
	By \cref{prop:go-up}, there exists an integer $m$ such that $m\geq m_i$ for every $i$, and such that the group $K^n_m(r)$ has the same Dehn function as the group $\K$.
	By applying \cref{prop:go-down} twice, we get that
	\[
		\Dehn{K_r^n(r)} \asympleq \Dehn{K_{m_1, \dots, m_n}(r)} \asympleq \Dehn{K_m^n(r)}.%
	\]
	Since the first and the last term are asymptotically equivalent by the assumption on $m$, these are all asymptotic equivalences, and the theorem is therefore proved.
\end{proof}

\subsection{Free abelian groups of different ranks}
The goal of this section is to prove the following inequality when we change the rank of the free abelian quotient group. It shows that, for fixed $n$, the Dehn functions of the $K_r^n(r)$ are non-decreasing in $r$. In particular, in the 3-factor case it will imply that $\delta_{K^3_r(r)}(N)\succcurlyeq N^4$ for every $r\geq 2$.
\begin{theorem}\label{thm:monotonicity-dehn-function}
	Let $n\geq 3$, $r'\geq r\geq 1$, $m_1,\cdots, m_n\geq \max\left\{r,2\right\}$ and $m_1',\dots, m_n'\geq \max\left\{r',2\right\}$ be integers. Then $\delta_{K_{m_1',\dots,m_n'}(r')}(N)\succcurlyeq \delta_{K_{m_1,\dots,m_n}(r)}(N)$.
\end{theorem}
\begin{proof}
	By \cite{carter2017stallings} if $r=1$ we have $\delta_{K_{m_1,\dots,m_n}(r)}(N)\asymp N^2$ and the statement is trivially true. Thus, we may assume that $r'\geq r\geq 2$ and, by \cref{thm:main-free-groups-different-ranks}, it then suffices to show that $\delta_{K_{r'}^n(r')}(N)\succcurlyeq \delta_{K_r^n(r)}(N)$.

	This follows from \cref{lemma:retr} and the observation that the retractions $F_{r'}^{(\alpha)}\to F_{r}^{(\alpha)}$ (resp. $\ZZ^{r'}\to \ZZ^r$) of the natural inclusions $F_r^{(\alpha)}\hookrightarrow F_{r'}^{(\alpha)}$ (resp. $\ZZ^r\hookrightarrow \ZZ^{r'}$) defined by $a_i^{(\alpha)}\mapsto a_i^{(\alpha)}$ if $i\in \range{r}$ and $a_i^{(\alpha)}\mapsto 1$ if $i>r$ (resp. $e_i\mapsto e_i$ if $i\in \range{r}$ and $e_i\mapsto 0$ if $i>r$) induce a commutative diagram
	\[
		\begin{tikzcd}
			& 1 \arrow[r] & K_{r'}^n(r')  \arrow[d]\arrow[r, hookrightarrow] & F_{r'}^{(1)}\times \dots \times F_{r'}^{(n)} \arrow[d] \arrow[r, twoheadrightarrow] & \ZZ^{r'} \arrow[d] \arrow[r] & 1 \\
			& 1 \arrow[r] & K_{r}^n(r) \arrow[r, hookrightarrow] & F_{r}^{(1)}\times \dots \times F_{r}^{(n)}\arrow[r,twoheadrightarrow] & \ZZ^r \arrow[r] & 1,%
		\end{tikzcd}
	\]
	where the vertical maps are retractions.%
\end{proof}

\subsection{Proof of \cref{main:3-factors,main:type-F-n-1}}\label{sec:proofs-of-mainthms-upper-bounds}

We now have all ingredients to prove \cref{main:3-factors,main:type-F-n-1}.

For a direct product $G_1\times \ldots \times G_n$ of groups $G_1,\ldots, G_n$ and $1\leq i_1<\ldots <i_k\leq n$ denote by $p_{i_1,\ldots,i_k} \colon G_1\times \cdots \times G_n\to G_{i_1}\times \cdots \times G_{i_k}$ the projection. We will often identify $G_i$ with its corresponding subgroup in $G_1\times \ldots \times G_n$. For a subgroup $H\leq G_1\times \ldots\times G_n$ we say that
\begin{itemize}
	\item $H$ is \emph{full} if $H\cap G_i \neq 1$ for all $ i \in \range n $, where we identify $G_i$ with the corresponding subgroup of $G_1\times \ldots \times G_n$;%
	\item $H$ is \emph{subdirect} if $p_i(H)=G_i$ for all $i\in \left\{1,\ldots, n\right\}$;
	\item $H$ has the \emph{VSP property (virtual surjection to pairs)} if $p_{i_1,i_2}(H)\leq G_{i_1}\times G_{i_2}$ has finite index for all $1\leq i_1 < i_2 \leq n$.
\end{itemize}

\begin{proof}[Proof of \cref{main:3-factors}]
	Since a group is of type $\mathcal{F}_2$ if and only if it is finitely presented, \cref{main:3-factors} is a direct consequence of \cref{main:type-F-n-1}.
\end{proof}

We are left with proving \cref{main:type-F-n-1}.
\begin{proof}[Proof of \cref{main:type-F-n-1}]
	Let $G\leq F_{m_1}\times \dots \times F_{m_n}$ be a finitely presented subgroup of a direct product of free groups of type $\mathcal{F}_{n-1}$. If $G$ is of type $\mathcal{F}_n$, then \cite[Theorem A]{BHMS-02} implies that $G$ is virtually a direct product of at most $n$ free groups. We may thus assume that $G$ is not of type $\mathcal{F}_{n}$. Then $G\cap F_{m_i}\neq 1$ for $i\in \range{n}$; else the projection away from a factor with trivial intersection would define an embedding of $G$ in a direct product with less factors and we could conclude as above.

	After replacing the $F_{m_i}$ by the finitely generated free groups $p_i(G)$ we may thus assume that $G\leq F_{m_1}\times \dots \times F_{m_n}$ is full subdirect. Moreover, we may assume that the $F_{m_i}$ are non-abelian. Indeed, if this is not the case, then it follows from \cite[Theorem 7.2.21]{Geo-08} and the fact that finitely generated abelian groups are of type $\mathcal{F}_{\infty}$ that the projection of $G$ to the direct product of the non-abelian factors is a full subdirect product of type $\mathcal{F}_{n-1}$ in a direct product of at most $n-1$ free groups and we can again conclude as above.

	Finally, we can argue similar as in the proof of \cite[Theorem 5.1]{KrophollerLlosa}, by combining  \cite[Corollary 3.5]{Kuckuck-14} and \cite[Corollary 5.4]{Llo-20}, that there is a finite index subgroup of $G$, which is isomorphic to $K_{m_1',\ldots,m_n'}(r)$ for some $m_1',\ldots,m_r'\geq 2$ and $r\geq 1$.

	The assertion now follows from \cref{thm:if-assumption-then-bound-Dehn-function} if $r\geq 2$ and from \cite[Corollary 4.3]{carter2017stallings} if $r=1$. This completes the proof.
\end{proof}

\subsection{Reduction from residually free groups to SPFs}

In this section we generalise an unpublished argument by Tessera and the fourth author, which allows us to prove \cref{main:equivalence-of-conjectures}.

Let $H\leq G_1\times \ldots \times G_n=:G$ be a full subdirect product of finitely presented groups $G_1=\left\langle \mathcal{X}_1\mid \mathcal{R}_1\right\rangle ,\ldots, G_n=\left\langle \mathcal{X}_n\mid \mathcal{R}_n\right\rangle$, let $\Phi=\left(\phi_1,\ldots,\phi_n\right)\colon F(\mathcal{X}_1)\times \cdots \times F(\mathcal{X}_n)\to G_1\times \cdots \times G_n$ be the canonical quotient homomorphism, and let $\widetilde{H}\coloneq\Phi^{-1}(H)\leq F(\mathcal{X}_1)\times \cdots \times F(\mathcal{X}_n)$. This is summarized by the following commutative diagram:
\[
	\begin{tikzcd}
		\widetilde{H} \arrow[r, hook] \arrow[d, "\Phi"] & F(\mathcal{X}_1)\times \cdots \times F(\mathcal{X}_n) \arrow[d, "\Phi"]\\
		H \arrow[r, hook] & G_1 \times \dots \times G_n.
	\end{tikzcd}
\]
Clearly, $\widetilde{H}$ is full subdirect inside $F(\mathcal{X}_1)\times \cdots \times F(\mathcal{X}_n)$.

\begin{lemma}\label{lem:vsp-of-preimage}
	If $H$ satisfies the VSP property, then $\widetilde{H}$ satisfies the VSP property.
\end{lemma}
\begin{proof}
	Assume that $H$ satisfies the VSP property and let $1\leq i_1 < i_2 \leq n$, say $i_1=1$ and $i_2=2$. 	Let $(a_{1},a_{2})$ be an element of the finite index subgroup
	\[
		(\phi_{1},\phi_{2})^{-1}\left(p_{1,2}(H)\right)\leq F(\mathcal{X}_{1})\times F(\mathcal{X}_{2}).
	\]
	Then there is an element $(g_1,\dots, g_n)\in H$ with $(g_1,g_2)=(\phi_1(a_1),\phi_2(a_2))$. Since $\Phi$ is surjective this implies that we can complete $a_1$ and $a_2$ to an element $(a_1,a_2,\dots, a_n)$ of $\widetilde{H}$. Thus, $\widetilde{H}$ has the VSP property as a subgroup of $F(\mathcal{X}_1)\times \cdots \times F(\mathcal{X}_n)$.
\end{proof}

\begin{proposition}\label{prop:VSP-Dehn}
	If $H$ satisfies the VSP property, then $\widetilde{H}$ is finitely presented and there is a finite presentation for $H$ and a constant $C\geq 1$ for which the Dehn function of $H$ is bounded above by $\delta_{\widetilde{H}}\left(C\cdot \max\left\{N,(\delta_{G_1}(CN))^2,\cdots,(\delta_{G_n}(CN))^2\right\}\right)$.
\end{proposition}
\begin{proof}
	By \cref{lem:vsp-of-preimage}, $\widetilde{H}\leq F(\mathcal{X}_1)\times \cdots \times F(\mathcal{X}_r)$ has the VSP property. It follows from \cite[Theorem A]{BHMS-13} that $H$ and $\widetilde{H}$ are finitely presented. Let $\widetilde{H}=\left\langle \mathcal{Y}\mid \mathcal{S}\right\rangle$ be a finite presentation for $\widetilde H$. By definition of $\Phi$, $ \mathcal{R}_i\subset F(\mathcal X_i)\cap \widetilde{H}$. Thus, for every $r_i\in \mathcal{R}_i$ there is a word $w_{r_i}(\mathcal{Y})$ that represents the element $r_i$ of $F(\mathcal X_i)$.

	Denote by $\mathcal{T}_i= \left\{w_{r_i}(\mathcal{Y})\mid r_i\in \mathcal{R}_i\right\}$ and by $\mathcal{T}\coloneqq \bigcup_{i=1}^n \mathcal{T}_i$. Since $\widetilde{H}\leq F(\mathcal{X}_1)\times \cdots \times F(\mathcal{X}_n)$ is subdirect, the normal subgroup $K_i\coloneq \normalclosure*{ \mathcal{T}_i}_{\widetilde{H}}$ of $\widetilde{H}$ coincides with the normal subgroup $\normalclosure{\mathcal{R}_i}_{F(\mathcal{X}_i)}$ of $F(\mathcal{X}_i)$. In particular, $\ker(\Phi)= K\coloneq K_1\times \cdots \times K_n=\normalclosure{ \mathcal{T}}_{\widetilde{H}}$ implying that $H=\left\langle \mathcal{Y}\mid \mathcal{S}\cup \mathcal{T}\right\rangle$.

	Let now $w(\mathcal{Y})$ be a word of length $\leq N$ that is null-homotopic in $H$. Then there is a constant $C_1>0$ (that only depends on our chosen presentation for $\widetilde{H}$) and words $v_i(\mathcal{X}_i)$, $1\leq i \leq n$, of length $\leq C_1\cdot N$ such that $v_i(\mathcal{X}_i)=_{F(\mathcal{X}_i)} p_i(w(\mathcal{Y}))$. Since $\phi_i(v_i)=_{G_i}1$ we can freely write $v_i(\mathcal{X}_i)$ as a product $v_i(\mathcal{X}_i)= \prod_{j=1}^{k_i} \overline{u}_{i,j}(\mathcal{X}_i)\cdot r_{i,j}(\mathcal{X}_i) \cdot u_{i,j}(\mathcal{X}_i)$ of at most $\delta_{G_i}(C_1 N)$ relations. A standard argument shows that, moreover, there is a constant $C_2>0$ that only depends on our chosen presentations for the $G_i$ such that we may assume that the $u_{i,j}(\mathcal{X}_i)$ are words of length at most $C_2\delta_{G_i}(C_1 N)$.

	Using again the subdirectness of $\widetilde{H}$ and the definition of $\mathcal{T}_i$, we observe that there is a constant $C_3>0$ that only depends on our chosen presentations for $\widetilde{H}$ and the $G_i$ such that for all $i,j$ there is a word $\nu_{i,j}(\mathcal{Y})$ of length at most $C_3 |u_{i,j}(\mathcal{X}_i)|$ such that $\overline{u}_{i,j}(\mathcal{X}_i)\cdot r_{i,j}(\mathcal{X}_i) \cdot u_{i,j}(\mathcal{X}_i)=_{\widetilde{H}} \overline{\nu}_{i,j}(\mathcal{Y}) w_{r_{i,j}}(\mathcal{Y}) \nu_{i,j}(\mathcal{Y})$; for this we observe that, for a suitable choice of $C_3$, for every letter in $x_i\in \mathcal{X}_i$ there is a word in $\mathcal{Y}$ of length at most $C_3$ whose projection to $F(\mathcal{X}_i)$ coincides with the group element represented by $x_i$ and then choose $\nu_{i,j}(\mathcal{Y})$ to be the concatenation of such words corresponding to the letters of $u_{i,j}$.

	Let $\nu_i(\mathcal{Y})\coloneq \prod_{j=1}^{k_i} \overline{\nu}_{i,j}(\mathcal{Y})\cdot w_{r_{i,j}}(\mathcal{Y}) \cdot \nu_{i,j}(\mathcal{Y})$. Then $\nu_i(\mathcal{Y})=_{\widetilde{H}}  v_i(\mathcal{X}_i)$ and the word length of $\nu_i$ is bounded above by
	\[
		\delta_{G_i}(C_1 \cdot N) \cdot \left(C_4 + 2 \cdot C_2 \cdot C_3 \cdot \delta_{G_i}(C_1\cdot N)\right),
	\]
	where we denote by $C_4$ the maximum of the lengths of the words $w_{i,j}(\mathcal{Y})$.

	By construction, the word $w(\mathcal{Y})\cdot \overline{\nu}_1(\mathcal{Y}) \cdots \overline{\nu}_n(\mathcal{Y})$ is null-homotopic in $\widetilde{H}$. Our above estimates then imply that
	\begin{align*}
		{\rm Area}_{\mathcal{S}}(w) & \leq \delta_{\widetilde{H}}\left(N + \sum_{i=1}^n \delta_{G_i}(C_1 \cdot N) \cdot \left(C_4 + 2 \cdot C_2 \cdot C_3 \cdot \delta_{G_i}(C_1\cdot N)\right)\right) \\
		                            & \leq \delta_{\widetilde{H}}\left(C\cdot \max\left\{N,(\delta_{G_1}(CN))^2,\ldots,(\delta_{G_n}(CN))^2\right\}\right),
	\end{align*}
	where $C=\max\left\{C_1, (n+1)\cdot (C_4+2 \cdot C_2\cdot C_3) \right\}$.
	This completes the proof.
\end{proof}

We can now prove \cref{main:equivalence-of-conjectures}. For this recall that $G$ is a limit group (or fully residually free group) if for every finite subset $S\subset G$ there is a homomorphism $\phi\colon G\to F_2$ such that $\phi|_S$ is injective. Residually free group are the generalisation of this notion, where this property only needs to be satisfied for all $2$-element sets $S$ containing the neutral element.
\begin{proof}[{Proof of \cref{main:equivalence-of-conjectures}}]
	Since subgroups of direct products of free groups are residually free, we only need to prove that \Cref{conj:Bridson-SPF} implies \Cref{conj:Bridson-RF}. By \cite[Theorem D]{BHMS-13}, it suffices to consider the case of a full subdirect product of a direct product of finitely many limit groups that has the VSP property. Since limit groups are ${\rm CAT}(0)$ by \cite{AliBes-06}, they have quadratic Dehn function. The assertion then follows by combining \Cref{conj:Bridson-SPF}, \Cref{prop:VSP-Dehn} and the fact that compositions of polynomially bounded functions are polynomially bounded.
\end{proof}

We can deduce the following generalisations of \cref{main:3-factors,main:type-F-n-1}.

\begin{corollary}\label{cor:limit-groups}
	Let $H\leq G_1\times \cdots \times G_n$ be a subgroup of a direct product of $n$ limit groups of type $\mathcal{F}_{n-1}$. Then $H$ has Dehn function bounded above by $N^{2d_n+4}$, where $d_n=7$ if $n\leq 3$, $d_4=3$ and $d_n=2$ if $n\geq 5$. In particular, every finitely presented subgroup of a direct product of at most three limit groups has Dehn function bounded above by $N^{18}$.
\end{corollary}
\begin{proof}
	We can use \cite[Theorem A]{BHMS-09} and the fact that the arguments in the proof of \cite[Theorem 5.1]{KrophollerLlosa} also apply to full subdirect products of limit groups (see \cite[Remark 5.1]{Llo-20}) to argue as in the proof of \cref{main:type-F-n-1} that we may assume that $H\cong \ker(\psi)$ for some homomorphisms $\psi\colon G_1\times \cdots \times G_n\to \ZZ ^r$, with $r\geq 1$, which is surjective on factors. Then the subgroup $\widetilde{H}\leq F(\mathcal{X}_1)\times \cdots \times F(\mathcal{X}_n)$ defined as above is the kernel of the composition $\psi\circ\Phi \colon F(\mathcal{X}_1)\times \cdots \times F(\mathcal{X}_n)\to \ZZ^r$, which is also surjective on factors. In particular, the conclusion of \cref{main:type-F-n-1} applies to $\widetilde{H}$. Thus, $\delta_{\widetilde{H}}(N)\preccurlyeq N^{d_n+2}$ and the assertion follows from \cref{prop:VSP-Dehn}.
\end{proof}

\begin{remark}
	Since every finitely generated residually free group is a subgroup of a direct product of finitely many limit groups and conversely every subgroup of such a product is residually free, \cref{cor:limit-groups} shows that as a consequence of our work, we obtain a positive answer to \cref{conj:Bridson-RF} for further classes of residually free groups.
\end{remark}

%% file: sections-uniform/appendix-three-factors.tex
\section{Three factors}
\label{sec:appendix-three-factors}
In this appendix we prove \cref{prop:symmetry,lem:basic-doubling,prop:powers} for $\K$.
Recall from \cref{sec:general-strategy} that $\K$ is generated by
\[\X=\set{x_1,\ldots x_r,y_1,\ldots y_r},\]
where we have set $ x_i = \kgen i1 $ and $ y_j = \kgen j2 $,

We consider the set of trivial words
\begin{align*}
	\simplecomms & \coloneq \set*{\left[x_i, y_i \right] \Suchthat
		i \in \range r
	}                                                                                                                     \\
	\swaps       & \coloneq \set*{\left[x_i^\epsilon, y_j^\delta \right]\left[x_j^\delta, y_i^\epsilon \right]  \Suchthat
		i \neq j \in \range r, \ \epsilon, \delta \in \set{\pm 1}
	}                                                                                                                     \\
	\triplecomms & \coloneq \set*{
		\begin{aligned}
			 & \left[x_i, \left[y_j^\epsilon, x_k^\delta\ol{y}_k^\delta \right]\right] \\
			 & \left[y_i, \left[x_j^\epsilon, x_k^\delta\ol{y}_k^\delta \right]\right]
		\end{aligned}
		\Suchthat
		\begin{aligned}
			 & i, j, k \in \range r \text{ pairwise distinct}, \\
			 & \epsilon, \delta \in \set{\pm 1}
		\end{aligned}
	}                                                                                                                     \\
	\quadcomms   & \coloneq \set*{
		\left[\left[x_i^\epsilon, x_k^\delta\inv y_k^\delta \right], \left[y_j^\sigma, x_h^\tau\inv y_h^\tau\right]\right]
		\Suchthat
		\begin{aligned}
			 & i, j, k, h \in \range r \text{ pairwise distinct}, \\
			 & \epsilon, \delta,\sigma,\tau \in \set{\pm 1}
		\end{aligned}
	}
\end{align*}
and we claim
\[\K=\presentation\X\R.\]
for $\R=\simplecomms\cup\swaps\cup\triplecomms\cup\quadcomms$.

\begin{remark}\label{rmk:exponents-wlog-1}
	In most of the proofs, we assume that all the exponents involved are equal to $1$, in order to ease notation. The other cases are analogous, and often they can be deduced directly by applying the automorphism of $ \K $ that sends $ x_i, y_i $ to $ \inv x_i, \inv y_i $ for some chosen $ i \in \range r $, and fixes $ x_j, y_j $ for $ j \ne i $.
\end{remark}

\begin{remark}
	A lot of computations that follow make silent use of the algebraic identities $[uv,w]=u[v,w]\inv u [u,w]$, $[u,vw]=[u,v]v[u,w]\inv v$ and $ [wu\inv w, v] = w[u, \inv w v w] \inv w $.
\end{remark}

\subsection{Useful trivial words}
We start by proving the following result, which says that if in each set of relations we remove the hypothesis that all the indices are distinct, we get an equivalent presentation. We also prove that some variants of the relations are trivial.
\begin{lemma}\label{lem:same-index-3}
	There is a constant $C$ (independent of $r$) such that
	for all $i,j,k,l\in\range r$ (not necessarily distinct) and $\epsilon,\delta,\sigma,\tau\in \set{\pm1}$, we have:
	\begin{enumerate}
		\item\label{indeq1-3} $\Area_{\R}([x_i^\epsilon,y_j^\delta][x_j^\delta,y_i^\epsilon])\leq C$;
		\item\label{symrel4-3} $\Area_{\R}([x_i^\epsilon,x_j^\delta\ol{y}_j^\delta][x_j^\delta,x_i^\epsilon\ol{y}_i^\epsilon])\le C$;
		\item\label{symrel5-3} $\Area_{\R}([y_i^\epsilon,x_j^\delta\ol{y}_j^\delta][y_j^\delta,x_i^\epsilon\ol{y}_i^\epsilon])\le C$;
		\item\label{indeq2-3} $\Area_{\R}([x_i,[y_j^\epsilon,y_k^\delta\ol{x}_k^\delta]])\leq C$;
		\item\label{indeq3-3} $\Area_{\R}([y_i,[x_j^\epsilon,x_k^\delta\ol{y}_k^\delta]])\leq C$;
		\item\label{indeq4-3} $\Area_{\R}([[x_i^\epsilon,x_k^\delta\ol{y}_k^\delta],[y_j^\sigma,x_h^\tau\ol{y}_h^\tau]])\leq C$;
		\item\label{symrel1-3} $\Area_{\R}([x_k^\sigma\ol{y}_k^\sigma,[x_i^\epsilon,y_j^\delta]])\le C$;
		\item\label{symrel2-3} $\Area_{\R}([[x_i^\epsilon,y_j^\delta],[x_k^\sigma,x_h^\tau\ol{y}_h^\tau]])\le C$;
		\item\label{symrel3-3} $\Area_{\R}([[x_i^\epsilon,y_j^\delta],[y_k^\sigma,x_h^\tau\ol{y}_h^\tau]])\le C$.
	\end{enumerate}
\end{lemma}

\begin{proof}
	We prove each item separately.
	\begin{enumerate}
		\item If $ i \neq j $, the expression is a relation. If $i=j$, the conclusion follows since $[x_i,y_i] \in \R$.

		\item Assume $ \epsilon=\delta=1 $ (see \cref{rmk:exponents-wlog-1}). We observe that using the relation $[\ol{y}_j,\ol{x}_i][\ol{y}_i,\ol{x}_j]$ we obtain %
		      \[
			      [x_i,x_j\ol{y}_j][x_j,x_i\ol{y}_i]= x_ix_j[\ol{y}_j,\ol{x}_i]\ol{y}_i\ol{x}_jy_i\ol{x}_i = x_ix_j[\ol{x}_j,\ol{y}_i]\ol{y}_i\ol{x}_jy_i\ol{x}_i=1.
		      \]

		\item Analogous to the previous case.

		\item If $ i,j,k $ are pairwise distinct, the expression belongs to the presentation. Assume $ j \neq k$, otherwise the conclusion is trivial; so either $ i=j $ or $ i=k $. We may assume without loss of generality that $ i=j $ by applying \cref{symrel5-3}.
		      Moreover, by combining \cref{rmk:exponents-wlog-1} with the fact that $[\inv x_i,[y_i^\epsilon,y_k^\delta\ol{x}_k^\delta]]$ is conjugate to $ [x_i,[y_i^\epsilon,y_k^\delta\ol{x}_k^\delta]]$, we may assume $ \epsilon=\delta=1 $.

		      Therefore, consider
		      \[
			      [x_i,[y_i,y_k\ol{x}_k]] = [x_i, y_i] y_i y_k \inv x_k \inv y_i x_k \inv y_k [y_k \inv x_k  y_i x_k \inv y_k, x_i] y_k \inv x_k  y_i x_k \inv y_k \inv y_i.
		      \]
		      Note that
		      \[
			      [y_k\ol{x}_ky_ix_k\ol{y}_k, x_i]
			      = x_i y_k \inv x_i [x_i, \inv y_k][\inv x_k, y_i] y_i x_i [\inv x_i, \inv y_k][\inv x_k, \inv y_i] \inv y_i \inv y_k \inv x_i
		      \]
		      which is trivial by $[y_i,\ol{x}_k][\ol{y}_k,x_i], [\inv x_i, \inv y_k][\inv x_k, \inv y_i]$ and $[x_i,y_i]$.
		\item Analogous to the previous item.
		\item If $ i,j,k,h $ are pairwise distinct, then the expression is a relation. If either $ i=k $ or $ j=h $, then the conclusion is trivial: so assume that one of $ i,k $ is equal to one either $ j $ or $h$. By applying \cref{symrel4-3,symrel5-3} we may assume that $ i=j $.

		      Suppose that $ \epsilon=\sigma $. Consider the expression $[[x_i^\epsilon,x_k^\delta \ol{y}_k^\delta],[y_i^\sigma,x_h^\tau\ol{y}_h^\tau]]$ and note that using \cref{indeq2-3}

		      it becomes
		      \[
			      [x_i^\epsilon \ol{y}_k^\delta \ol{x}_i ^\epsilon y_k^\delta \ol{x}_k^\delta ,[y_i^\sigma ,x_h^\tau \ol{y}_h^\tau ]]
		      \]
		      which, using $[x_i^\epsilon ,\ol{y}_k^\delta ][\ol{x}_k^\delta ,y_i^\epsilon ]$, becomes
		      \[
			      [y_i^\epsilon \ol{x}_k^\delta \ol{y}_i^\epsilon ,[y_i^\sigma,x_h^\tau\ol{y}_h^\tau]].
		      \]
		      Now we use $ \epsilon = \sigma $ to rewrite it as
		      \[
			      y_i^\epsilon[\ol{x}_k^\delta,[x_h^\tau\ol{y}_h^\tau,\ol{y}_i^\epsilon]]\ol{y}_i^\epsilon
		      \]
		      and the conclusion follows.

		      Suppose instead that $ \epsilon = -\sigma $. Starting again from $[[x_i^\epsilon,x_k^\delta \ol{y}_k^\delta],[y_i^\sigma,x_h^\tau\ol{y}_h^\tau]]$ and using \cref{indeq3-3} twice we obtain
		      \[
			      [[x_i^\epsilon,x_k^\delta \ol{y}_k^\delta],[x_h^\tau\ol{y}_h^\tau, \inv y_i^\sigma]].
		      \]
		      By the same arguments as before we reach
		      \[
			      [y_i^\epsilon \ol{x}_k^\delta \ol{y}_i^\epsilon ,[x_h^\tau\ol{y}_h^\tau, \inv y_i^\sigma]]
		      \]
		      that, using $ \epsilon=-\sigma $, can be rewritten as
		      \[
			      y_i^\epsilon[\ol{x}_k^\delta,[\inv y_i^\epsilon,x_h^\tau\ol{y}_h^\tau]]\ol{y}_i^\epsilon
		      \]
		      and again we conclude.

		\item Assume $ \sigma=\epsilon=\delta=1 $. We observe that
		      \[
			      \ol{y}_j\ol{x}_i[x_k\ol{y}_k,[x_i,y_j]]x_iy_j=
			      \ol{y}_j[\ol{x}_i,x_k\ol{y}_k]y_j[\ol{y}_j,x_k\ol{y}_k][x_k\ol{y}_k,\ol{x}_i]\ol{x}_i[x_k\ol{y}_k,\ol{y}_j]x_i
		      \]
		      and the conclusion follows by \cref{indeq2-3,indeq3-3,indeq4-3}.%

		      The other cases are completely analogous, by replacing $ x_i \leadsto x_i^\epsilon $, $ y_j \leadsto y_j^\delta $, $ x_k \leadsto x_k^\sigma $ and $ y_k \leadsto y_k^\sigma $ in the proof above.
		\item Assume $ \sigma=\epsilon=\delta=\tau=1 $. We observe that
		      \[
			      [[x_i,y_j],[x_k,x_h\ol{y}_h]]=[y_i[\ol{y}_ix_i,y_j]y_j\ol{y}_i\ol{y}_j,[x_k,x_h\ol{y}_h]]
		      \]
		      and the conclusion follows by \cref{indeq2-3,indeq4-3}. The other cases are analogous.
		\item Analogous to the previous case.
	\end{enumerate}
\end{proof}

\subsection{Proof of \cref{prop:symmetry} for three factors}

Recall from \cref{def:homo-norm} that for a homomorphism
$\phi\colon\absfree r\to \absfree{r'}$
the norm $\norma{\phi}$ is the maximum of the lengths of $\phi(\afgen 1),\ldots,\phi(\afgen r)$ as reduced words in $\afgen 1,\ldots, \afgen{r'}$ (and their inverses). Notice that $\norma{\oc\phi}=\norma{\phi}$.

\begin{proposition}[\cref{prop:symmetry} for $n=3$]\label{prop:symmetry-3}
	There exists a constant $A_1>0$ such that the following holds:
	for all integers $r,r'\ge1$ and every homomorphism of free groups $\phi \colon \absfree r\rar\absfree {r'}$  with $\norma{\phi}\le 1$, we have
	\[
		\Area_{\R [r']}\left(\simmetriz\phi(\R) \right)\le A_1.
	\]
\end{proposition}

\begin{proof}
	Let $r\geq 1$, $n\geq 3$ be integers.
	Let $R\in\R$ and let $i_1,\ldots,{i_k}$ be the indices involved in $R$, where $1\le k\le 4$ and $1\le i_1<\ldots<i_k\le r$.
	If $\phi(\afgen {i_j})=1$ for some $1\le j\le k$ then $\oc\phi(R)=1$ and we are done.
	Otherwise, the conclusion follows by \cref{lem:same-index-3}.
\end{proof}

\subsection{Proof of \cref{lem:basic-doubling} for three factors}
The purpose of this section is to prove the following proposition.
\begin{proposition}[\cref{lem:basic-doubling} for $n=3$]\label{lem:basic-doubling-3}
	There exists a constant $A_2>0$ such that the following holds:
	for every integer $r\ge1$, consider the homomorphism $\rho_r \colon \absfree r\rar \absfree {r+1}$ given by $\rho_r(\afgen 1)=\afgen1\afgen2$ and $\rho_r(\afgen i)={\afgen {i+1}}$ for $i=2,\dots ,r$. Then we have
	\[
		\Area_{\R[r+1]}(\simmetriz\rho(\R))\le A_2.
	\]
\end{proposition}
We split the proof into several lemmas, one for every type of relation.

\begin{lemma}[Doubling for $\simplecomms$]\label{double:small-commutator-3}
	There is a constant $C$ (independent of $r$) such that, for every $i,i'\in \range r$, we have
	\[\Area_{\R}([x_ix_{i'},y_iy_{i'}])\le C.\]
\end{lemma}
\begin{proof}
	We have
	\[\ol{y}_i[x_ix_{i'},y_iy_{i'}]y_i=(\ol{y}_ix_i)x_{i'}y_i(y_{i'}\ol{x}_{i'})\ol{x}_i\ol{y}_{i'},\]
	which, using $[x_i,y_i]$ and $[x_{i'},y_{i'}]$ and conjugating, becomes
	$$[\ol{y}_i,x_{i'}][y_{i'},\ol{x}_i]$$
	and the conclusion follows.
\end{proof}

\begin{lemma}[Doubling for $\swaps$]\label{double:double-commutator-3}
	There is a constant $C$ (independent of $r$) such that, for every $i,i',j\in\range r$ and $\epsilon,\delta\in \set{\pm1}$, we have
	\[\Area_{\R}([(x_i x_{i'})^\epsilon,y_j^\delta][x_j^\delta,(y_i y_{i'})^\epsilon])\le C.\]
\end{lemma}
\begin{proof}
	We have
	\[[x_ix_{i'},y_j][x_j,y_iy_{i'}]=x_i[x_{i'},y_j]\ol{x}_i[x_i,y_j][x_j,y_i]y_i[x_j,y_{i'}]\ol{y}_i\]
	and we can use the relation $[x_i,y_j][x_j,y_i]$ and conjugate to obtain
	\[\ol{y}_ix_i[x_{i'},y_j]\ol{x}_iy_i[x_j,y_{i'}].\]
	We conclude by applying the the trivial word $[\ol{x}_i y_i,[x_{i'},y_j]]$ (\cref{lem:same-index-3} \cref{symrel1-3}) and the relation $[x_{i'},y_j][x_j,y_{i'}]$.
\end{proof}

\begin{lemma}[Doubling for $\triplecomms$ - Part I]\label{double:triple-commutator-x-3}
	There is a constant $C$ (independent of $r$) such that, for every $ i,i',j,j',k,k'\in \range r$ and $\epsilon,\delta\in \set{\pm1}$, we have:
	\begin{enumerate}
		\item\label{tcx1-3} $\Area_{\R}([x_ix_{i'},[y_j^\epsilon,x_k^\delta \ol{y}_k^\delta]])\le C$;
		\item\label{tcx2-3} $\Area_{\R}([x_i,[(y_j y_{j'})^\epsilon,x_k^\delta\ol{y}_k^\delta])\le C$;
		\item\label{tcx3-3} $\Area_{\R}([x_i,[y_j^\epsilon,(x_k x_{k'})^\delta(\ol{y}_{k'}\ol{y}_k)^\delta]])\le C$.
	\end{enumerate}
\end{lemma}
\begin{proof}
	We distinguish the three cases.
	\begin{enumerate}
		\item Immediate from the relations in $\triplecomms$.
		\item We have
		      \[
			      [x_i,[y_jy_{j'},x_k\ol{y}_k]]=[x_i,y_j[y_{j'},x_k\ol{y}_k]\ol{y}_j[y_j,x_k\ol{y}_k]]
		      \]
		      and we can use $[x_i,[y_j,x_k\ol{y}_k]]$ and conjugate to obtain
		      \[
			      [\ol{y}_jx_iy_j,[y_{j'},x_k\ol{y}_k]]
		      \]
		      which follows using $[x_i,[y_{j'},x_k\ol{y}_k]]$ and $[[\ol{x}_i,\ol{y}_j],[y_{j'},x_k\ol{y}_k]]$; for this last identity, we apply \cref{lem:same-index-3} \cref{symrel3-3}.

		\item We consider
		      \[
			      [x_i,[y_j,x_kx_{k'}\ol{y}_{k'}\ol{y}_k]]
		      \]
		      and we use $[x_kx_{k'},y_ky_{k'}]$ (using \cref{double:small-commutator-3}) in order to obtain
		      \[
			      [x_i,[y_j,\ol{y}_{k'}\ol{y}_kx_kx_{k'}]]
		      \]
		      Using $[x_k,y_k]$ and $[x_{k'},y_{k'}]$ this becomes
		      \[
			      [x_i,[y_j,[\ol{y}_{k'},x_k\ol{y}_k]x_k\ol{y}_kx_{k'}\ol{y}_{k'}]].
		      \]
		      This is equal to
		      \[
			      [x_i,[y_j,[\ol{y}_{k'},x_k\ol{y}_k]][\ol{y}_{k'},x_k\ol{y}_k][y_j,x_k\ol{y}_kx_{k'}\ol{y}_{k'}][x_k\ol{y}_k,\ol{y}_{k'}]],
		      \]
		      which is a product of conjugates of the words $[x_i,[\ol{y}_{k'},x_k\ol{y}_k]]$, $[x_i,[y_j,[\ol{y}_{k'},x_k\ol{y}_k]]]$, and $[x_i,[y_j,x_k\ol{y}_kx_{k'}\ol{y}_{k'}]]$.

		      In order to compute the area of $[x_i,[y_j,[\ol{y}_{k'},x_k\ol{y}_k]]]=1$ we observe that
		      \[
			      [x_i,[y_j,[\ol{y}_{k'},x_k\ol{y}_k]]]=
			      [x_i,y_j[\ol{y}_{k'},x_k\ol{y}_k]\ol{y}_j[x_k\ol{y}_k,\ol{y}_k']]
		      \]
		      which, using $[x_i,[\ol{y}_{k'},x_k\ol{y}_k]]$ and conjugating, becomes
		      \[
			      [\ol{y}_jx_iy_j,[\ol{y}_{k'},x_k\ol{y}_k]]
		      \]
		      which follows using $[x_i,[\ol{y}_{k'},x_k\ol{y}_k]]$ and $[[\ol{x}_i,\ol{y}_j],[\ol{y}_{k'},x_k\ol{y}_k]]$; for this last identity we use \cref{lem:same-index-3}\cref{symrel3-3}.

		      To obtain $[x_i,[y_j,x_k\ol{y}_kx_{k'}\ol{y}_{k'}]]=1$ we observe that
		      \[
			      [x_i,[y_j,x_k\ol{y}_kx_{k'}\ol{y}_{k'}]]=
			      [x_i,[y_j,x_k\ol{y}_k]x_k\ol{y}_k[y_j,x_{k'}\ol{y}_{k'}]y_k\ol{x}_k]]
		      \]
		      which, using $[x_i,[y_j,x_k\ol{y}_k]]$ and conjugating, becomes
		      \[
			      [y_k\ol{x}_kx_ix_k\ol{y}_k,[y_j,x_{k'}\ol{y}_{k'}]]]
		      \]
		      which follows using $[x_k,y_k]$, $[x_i,[y_j,x_{k'}\ol{y}_{k'}]]$ and $[[\ol{x}_i,\ol{x}_ky_k],[y_j,x_{k'}\ol{y}_{k'}]]$.
	\end{enumerate}
\end{proof}

\begin{lemma}[Doubling for $\triplecomms$ - Part II]\label{double:triple-commutator-y-3}
	There is a constant $C$ (independent of $r$) such that, for every $ i,i',j,j',k,k'\in\range r$ and $\epsilon,\delta\in \set{\pm1}$, we have:
	\begin{enumerate}
		\item\label{tcy1-3} $\Area_{\R}([y_iy_{i'},[x_j^\epsilon,x_k^\delta \ol{y}_k^\delta]])\le C$;
		\item\label{tcy2-3} $\Area_{\R}([y_i,[(x_j x_{j'})^\epsilon,x_k^\delta\ol{y}_k^\delta])\le C$;
		\item\label{tcy3-3} $\Area_{\R}([y_i,[x_j^\epsilon,(x_k x_{k'})^\delta(\ol{y}_{k'}\ol{y}_k)^\delta]])\le C$.
	\end{enumerate}
\end{lemma}
\begin{proof}
	Analogous to \cref{double:triple-commutator-x-3}.
\end{proof}

\begin{lemma}[Doubling for $\quadcomms$]\label{double:quadruple-commutator-3}
	There is a constant $C$ (independent on $r$) such that, for every $ i,i',j,j',k,k',h,h'\in\range r$ and $\epsilon,\delta,\sigma,\tau\in\set{\pm1}$, we have:
	\begin{enumerate}
		\item\label{bs1-3} $\Area_{\R}([[(x_i x_{i'})^\epsilon,x_k^\delta \ol{y}_k^\delta],[y_j^\sigma,x_h^\tau\ol{y}_h^\tau]])\le C$;
		\item\label{bs2-3} $\Area_{\R}([[x_i^\epsilon,x_k^\delta \ol{y}_k^\delta],[(y_j y_{j'})^\sigma,x_h^\tau\ol{y}_h^\tau]])\le C$;
		\item\label{bs3-3} $\Area_{\R}([[x_i^\epsilon,(x_k x_{k'})^\delta (\ol{y}_{k'} \ol{y}_k)^\delta],[y_j^\sigma,x_h^\tau\ol{y}_h^\tau]])\le C$;
		\item\label{bs4-3} $\Area_{\R}([[x_i^\epsilon,x_k^\delta \ol{y}_k^\delta],[y_j^\sigma,(x_h x_{h'})^\tau (\ol{y}_{h'} \ol{y}_h)^\tau]])\le C$.
	\end{enumerate}
\end{lemma}
\begin{proof}
	We distinguish the four cases.
	\begin{enumerate}
		\item  We observe that
		      \[
			      [[x_ix_{i'},x_k\ol{y}_k],[y_j,x_h\ol{y}_h]]=
			      [x_i[x_{i'},x_k\ol{y}_k]\ol{x}_i[x_i,x_k\ol{y}_k],[y_j,x_h\ol{y}_h]]
		      \]
		      and the conclusion follows.
		\item Analogous to the previous item.
		\item We observe that
		      \begin{multline*}
			      [[x_i,x_kx_{k'}\ol{y}_{k'}\ol{y}_k],[y_j,x_h\ol{y}_h]]=\\
			      =[x_ix_k\ol{x}_i[x_i,x_{k'}\ol{y}_{k'}](x_{k'}\ol{y}_{k'}[x_i,\ol{y}_k]y_{k'}\ol{x}_{k'})\ol{x}_k,[y_j,x_h\ol{y}_h]]
		      \end{multline*}
		      which, using $[x_{k'}\ol{y}_{k'},[x_i,\ol{y}_k]]$ (obtained from \cref{lem:same-index-3}), becomes
		      \[
			      [x_ix_k\ol{x}_i[x_i,x_{k'}\ol{y}_{k'}][x_i,\ol{y}_k]\ol{x}_k,[y_j,x_h\ol{y}_h]]
		      \]
		      and the conclusion follows using \cref{lem:same-index-3,lem:same-index-3}.
		\item Analogous to the previous item.
	\end{enumerate}
\end{proof}

\begin{proof}[Proof of \cref{lem:basic-doubling-3}]
	Let $R\in\cR_r$ and let ${i_1},\ldots,{i_k}$ be the indices involved in $R$, where $1\le k\le 4$ and $1\le i_1<\dots <i_k\le r$. If $i_1>1$ then $\oc\rho(R)\in\cR_{r+1}$ and we are done. If $i_1=1$ then the conclusion follows by
	\cref{double:small-commutator-3,double:double-commutator-3,double:triple-commutator-x-3,double:triple-commutator-y-3,double:quadruple-commutator-3}.
\end{proof}

\subsection{Proof of \cref{prop:powers} for three factors}
In this proposition we will prove \cref{prop:powers}, that we recall here.
\begin{proposition}[\cref{prop:powers} for $n=3$]\label{ass:powers-3}
	There exists a constant $C>0$ such that the following happens: let $r \geq 2$ and $N_1,\ldots,N_r$ be integers, and consider the homomorphism
	\[\omega=\omega_{N_1,\ldots,N_r}\colon \absfree r \rar \absfree r\]
	given by $\omega(\afgen i)=\afgen i^{N_i}$ for $i \in \range r$. Then
	\[
		\Area_{\R}\left(\simmetriz\omega(\R)\right)\le C(\max\{\abs{N_1},\dots,\abs{N_r}\})^7 .
	\]
\end{proposition}

\begin{lemma}[Power relations for $\simplecomms$]\label{thin:simple-commutator-3}
	There is a constant $C$ (independent of $r$) such that, for all integers $\N$ and all $ i\in\range r$, we have
	\[\Area_{\R }([x_i^\N, y_i^{\N}])\le C\N^2.\]
\end{lemma}
\begin{proof}
	Immediate.
\end{proof}

The following lemmas compute the area of some expressions that are useful for proving power relations.
\begin{lemma}[Power relations for $\swaps$ and similar]\label{thin:double-commutator-3}
	There is a constant $C$ (independent of $r$) such that, for all integers $\N$, $ i,j\in\range r$ and $\epsilon\in\set{\pm 1}$, we have:
	\begin{enumerate}
		\item\label{pow-swap-1-3} $\Area_{\R}([x_i^\epsilon, y_j^{\N}][x_j^{\N}, y_i^\epsilon])\le C\N^2$;
		\item\label{pow-swap-2-3} $\Area_{\R}([x_i^\epsilon, x_j^{\N}\ol{y}_j^{\N}][x_j^{\N}, x_i^\epsilon\ol{y}_i^\epsilon])\le C\N^2$;
		\item\label{pow-swap-3-3} $\Area_{\R}([y_i^\epsilon, y_j^{\N}\ol{x}_j^{\N}][y_j^{\N}, x_i^\epsilon\ol{y}_i^\epsilon])\le C\N^2$.
	\end{enumerate}
\end{lemma}
\begin{proof} Assume $ N > 0 $ (the other case is analogous). We observe that
	\[
		[x_i, y_j^\N][x_j^\N, y_i] = ([x_i,y_j]y_j)^\N \ol{y}_j^\N x_j^\N (\ol{x}_j[x_j,y_i])^\N
	\]
	which, using $ O(\N^2) $ times $ [x_j, y_j] $, becomes
	\[
		([x_i,y_j]y_j)^\N (x_j\ol{y}_j)^\N (\ol{x}_j[x_j,y_i])^\N.
	\]
	But for $L \in \range \N$ we have
	\[
		[x_i,y_j]y_j (x_j\ol{y}_j)^L \ol{x}_j[x_j, y_i] = [x_i,y_j] (x_j \ol{y}_j)^{L-1} [x_j, y_i] = (x_i \ol{y}_j)^{L-1}
	\]
	by using $ O(L) $ times $ [x_j, y_j] $, $ O(L) $ times $ [x_j\ol{y}_j, [x_i,y_j]] $ (obtained from \cref{lem:same-index-3}) and once $ [x_i,y_j][x_j,y_i] $. This proves \cref{pow-swap-1-3}. The other items are analogous.
\end{proof}

\begin{lemma}[Power relations for $\quadcomms$ and similar]\label{thin:quadruple-commutator-3}
	There is a constant $C$ (independent of $r$) such that, for all integers $\N$, all $i,j,k,h\in \range r$ and all $\epsilon,\delta,\sigma,\tau\in\set{\pm 1}$, we have:
	\begin{enumerate}
		\item\label{it:qc1-3} $\Area_{\R}([[x_i^{\N},x_k^{\sigma}\ol{y}_k^{\sigma}],[y_j^{\delta},x_h^{\tau}\ol{y}_h^{\tau}]])\le C\abs\N$;
		\item\label{it:qc2-3} $\Area_{\R}([[x_i^{\epsilon},x_k^{\N}\ol{y}_k^{\N}],[y_j^{\delta},x_h^{\tau}\ol{y}_h^{\tau}]])\le C\abs\N^2$;
		\item $\Area_{\R}([[x_i^{\epsilon},x_k^{\sigma}\ol{y}_k^{\sigma}],[y_j^{\N},x_h^{\tau}\ol{y}_h^{\tau}]])\le C\abs\N$;
		\item $\Area_{\R}([[x_i^{\epsilon},x_k^{\sigma}\ol{y}_k^{\sigma}],[y_j^{\delta},x_h^{\N}\ol{y}_h^{\N}]])\le C\abs\N^2$;

		\item $\Area_{\R}([[x_i^{\N},y_j^{\delta}],[x_k^{\sigma},x_h^{\tau}\ol{y}_h^{\tau}]])\le C\abs\N^2$;
		\item $\Area_{\R}([[x_i^{\epsilon},y_j^{\N}],[x_k^{\sigma},x_h^{\tau}\ol{y}_h^{\tau}] ])\le C\abs\N$;
		\item $\Area_{\R}([[x_i^{\epsilon},y_j^{\delta}],[x_k^{\N},x_h^{\tau}\ol{y}_h^{\tau}]])\le C\abs\N^2$;
		\item $\Area_{\R}([[x_i^{\epsilon},y_j^{\delta}],[x_k^{\sigma},x_h^{\N}\ol{y}_h^{\N}]])\le C\abs\N^2$;

		\item $\Area_{\R}([[x_i^{\N},y_j^{\delta}],[y_k^{\sigma},x_h^{\tau}\ol{y}_h^{\tau}]])\le C\abs\N$;
		\item $\Area_{\R}([[x_i^{\epsilon},y_j^{\N}],[y_k^{\sigma},x_h^{\tau}\ol{y}_h^{\tau}]])\le C\abs\N^2$;
		\item $\Area_{\R}([[x_i^{\epsilon},y_j^{\delta}],[y_k^{\N},x_h^{\tau}\ol{y}_h^{\tau}]])\le C\abs\N^2$;
		\item $\Area_{\R}([[x_i^{\epsilon},y_j^{\delta}],[y_k^{\sigma},x_h^{\N}\ol{y}_h^{\N}]])\le C\abs\N^2$.
	\end{enumerate}
\end{lemma}
\begin{proof}
	For \cref{it:qc1-3}, note that
	\[
		[[x_i^\N,x_k\ol{y}_k],[y_j,x_h\ol{y}_h]]=
		[x_i^\N([x_k\ol{y}_k,\ol{x}_i]\ol{x}_i)^\N,[y_j,x_h\ol{y}_h]]
	\]
	which follows by applying $O(\N)$ times the relation $[x_i,[y_j,x_h\ol{y}_h]]$ and $O(\N)$ times the relation $[[x_k\ol{y}_k,\ol{x}_i],[y_j,x_h\ol{y}_h]]$. \cref{it:qc2-3} is analogous, with the only difference that at the beginning we need to substitute $[x_i,x_k^\N\ol{y}_k^\N]$ with $[x_i\ol{y}_i,x_k^\N]$ using \cref{thin:double-commutator-3}. The other items are analogous.
\end{proof}

\begin{lemma}[Power relation for $\triplecomms$]\label{thin:triple-commutator-3}
	There is a constant $C$ (independent of $r$) such that, for all integers $\N$, all $i,j,k \in\range r$ and all $\epsilon\in \set{\pm 1}$, we have:
	\begin{enumerate}
		\item\label{it:tc1-3} $\Area_{\R}([x_i,[y_j^{\N},x_k^{\epsilon}\ol{y}_k^{\epsilon}]])\le C\abs\N^3$;
		\item $\Area_{\R}([x_i,[y_j^{\epsilon},x_k^{\N}\ol{y}_k^{\N}]])\le C\abs\N^3$;

		\item $\Area_{\R}([y_i,[x_j^{\N},x_k^{\epsilon}\ol{y}_k^{\epsilon}]])\le C\abs\N^3$;
		\item $\Area_{\R}([y_i,[x_j^{\epsilon},x_k^{\N}\ol{y}_k^{\N}]])\le C\abs\N^3$;

		\item $\Area_{\R}([x_i\ol{y}_i,[x_j^{\epsilon},y_k^{\N}]])\le C\abs\N^3$.
	\end{enumerate}
\end{lemma}
\begin{proof}
	We assume $N>0$ and $ \epsilon = 1 $, the other cases being similar. For \cref{it:tc1-3} we note that
	\[
		[y_j^\N,x_k\ol{y}_k]=
		(y_j^{\N-1}[y_j,x_k\ol{y}_k]\ol{y}_j^{\N-1})(y_j^{\N-2}[y_j,x_k\ol{y}_k]\ol{y}_j^{\N-2})\dots([y_j,x_k\ol{y}_k]),
	\]
	and thus it is enough to prove that, for every $L=0,\dots,N-1$, the area of $[x_i,y_j^L[y_j,x_k\ol{y}_k]\ol{y}_j^L]$ is at most $O(L^2)$. But $[x_i,y_j^L[y_j,x_k\ol{y}_k]\ol{y}_j^L]$ is conjugate to
	\[
		[\ol{y}_j^Lx_iy_j^L,[y_j,x_k\ol{y}_k]]
	\]
	which follows using $[x_i,[y_j,x_k\ol{y}_k]]$ once and $[[\ol{x}_i,\ol{y}_j^L],[y_j,x_k\ol{y}_k]]$ (with area $O(L^2)$ by \cref{thin:quadruple-commutator-3}). The other parts are analogous.
\end{proof}

\begin{lemma}\label{prop:thick-3}
	There is a constant $C$ (independent of $r$) such that, for all integers $\N,M,P,Q$ and all $i,j,k,h\in \range r$, we have:
	\begin{enumerate}
		\item\label{thick:c-3} $\Area_{\R}([x_i^{N},y_i^{N}])\le C\abs N^2$;
		\item\label{thick:dc-3} $\Area_{\R}([x_i^{N},y_j^{M}][x_j^{M},y_i^{N}])\le C\abs M^2\abs N^3$;
		\item\label{thick:tc-3} $\Area_{\R}([x_i^N,[y_j^{M},x_k^{P}\ol{y}_k^{P}]])\le C\abs N\left(\abs M^3+\abs M^2\abs P^3\right)$;
		\item\label{thick:qc-3} $\Area_{\R}([[x_i^{N},x_k^{P}\ol{y}_k^{P}],[y_j^{M},x_h^{Q}\ol{y}_h^{Q}]])\le C\abs N\abs P\left(\abs P+\abs M^3+\abs M^2\abs Q^3\right)$.
	\end{enumerate}
\end{lemma}
\begin{proof}
	We prove each item separately.
	\begin{enumerate}
		\item Immediate.
		\item We assume that $M>1$, the case $M\leq 0$ being similar. We observe that
		      \[
			      [x_i^N,y_j^M][x_j^M,y_i^N]=
			      ([x_i^N,y_j]y_j)^M\ol{y}_j^Mx_j^M(\ol{x}_j[x_j,y_i^N])^M
		      \]
		      which using $O(M^2)$ times $[x_j,y_j]$ becomes
		      \[
			      ([x_i^N,y_j]y_j)^M(x_j\ol{y}_j)^M(\ol{x}_j[x_j,y_i^N])^M.
		      \]
		      But for $L=1,\dots ,M$ we have
		      \begin{align*}
			      ([x_i^N,y_j]y_j)(x_j\ol{y}_j)^L(\ol{x}_j[x_j,y_i^N]) & =
			      [x_i^N,y_j](x_j\ol{y}_j)^{L-1}[x_j,y_i^N]                                   \\
			                                                           & =(x_j\ol{y}_j)^{L-1}
		      \end{align*}
		      using $O(L)$ times $[x_j,y_j]$ and $O(L)$ times $[x_j\ol{y}_j,[x_i^N,y_j]]$ (whose area is at most $O(N^3)$ by \cref{thin:triple-commutator-3}) and once $[x_i^N,y_j][x_j,y_i^N]$ (which has area at most $O(N^2)$ by \cref{thin:double-commutator-3}). This proves \cref{thick:dc-3}.

		\item  We assume that $M>1$, the case $M\leq 0$ being similar. We note that
		      \[
			      [y_j^M,x_k^P\ol{y}_k^P] = (y_j^{M-1}[y_j,x_k^P\ol{y}_k^P]\ol{y}_j^{M-1})(y_j^{M-2}[y_j,x_k^P\ol{y}_k^P]\ol{y}_j^{M-2})\dots([y_j,x_k^P\ol{y}_k^P])
		      \]
		      so it is enough to prove that, for every $L=1,\dots,M$, the area of $[x_i,y_j^L[y_j,x_k^P\ol{y}_k^P]\ol{y}_j^L]$ is at most $O(L^2+LP^3)$. But $[x_i,y_j^L[y_j,x_k^P\ol{y}_k^P]\ol{y}_j^L]$ is conjugate to
		      \[
			      [\ol{y}_j^Lx_iy_j^L,[y_j,x_k^P\ol{y}_k^P]]
		      \]
		      which, using $[x_i,[y_j,x_k^P\ol{y}_k^P]]$ (whose area is at most $O(P^3)$ by \cref{thin:triple-commutator-3}), becomes
		      \[
			      [[\ol{x}_i,\ol{y}_j^L],[y_j,x_k^P\ol{y}_k^P]].
		      \]
		      Using $[\ol{x}_i,\ol{y}_j^L][\ol{x}_j^L,\ol{y}_i]$, whose area is bounded from above by $O(L^2)$ by \cref{thin:double-commutator-3}), we get
		      \[
			      [[\ol{y}_i,\ol{x}_j^L],[y_j,x_k^P\ol{y}_k^P]]=
			      [([\ol{y}_i,\ol{x}_j]\ol{x}_j)^Lx_j^L,[y_j,x_k^P\ol{y}_k^P]]
		      \]
		      which follows using $O(L)$ times $[x_j,[y_j,x_k^P\ol{y}_k^P]]$, of area bounded by $O(P^3)$ (\cref{thin:triple-commutator-3}) and $O(L)$ times $[[\ol{y}_i,\ol{x}_j],[y_j,x_k^P\ol{y}_k^P]]$ (which has area $\leq O(P^2)$ by \cref{thin:quadruple-commutator-3}). Finally, in order to prove \cref{thick:tc-3}, we observe that $[x_i^N,[y_j^{M\epsilon},x_k^{P\delta}\ol{y}_k^{P\delta}]]$ follows by applying $N$ times $[x_i,[y_j^{M\epsilon},x_k^{P\delta}\ol{y}_k^{P\delta}]]$.

		\item We note that
		      \[
			      [[x_i^N,x_k^P\ol{y}_k^P],[y_j^M,x_h^Q\ol{y}_h^Q]]=
			      [x_i^N(\ol{x}_i[x_i,x_k^P\ol{y}_k^P])^N,[y_j^M,x_h^Q\ol{y}_h^Q]]
		      \]
		      which, using $O(N)$ times $[x_i,x_k^P\ol{y}_k^P][x_k^P,x_i\ol{y}_i]$ (which has area $\leq O(P^2)$ by \cref{thin:double-commutator-3}), becomes
		      \[
			      [x_i^N(\ol{x}_i[x_i\ol{y}_i,x_k^P])^N,[y_j^M,x_h^Q\ol{y}_h^Q]]
		      \]
		      which follows by applying $O(N)$ times the element $[x_i,[y_j^M,x_h^Q\ol{y}_h^Q]]]$, whose area is bounded by $O(M^3+M^2Q^3)$ by \cref{thick:tc-3}, and $O(N)$ times $[[x_i\ol{y}_i,x_k^P],[y_j^M,x_h^Q\ol{y}_h^Q]]$, whose area we will estimate now.

		      We note that
		      \[
			      [[x_i\ol{y}_i,x_k^P],[y_j^M,x_h^Q\ol{y}_h^Q]]=
			      [([x_i\ol{y}_i,x_k]x_k)^P\ol{x}_k^P,[y_j^M,x_h^Q\ol{y}_h^Q]]
		      \]
		      which follows using $O(P)$ times the element $[x_k,[y_j^M,x_h^Q\ol{y}_h^Q]]$ (which has area at most $O(M^3+M^2Q^3)$ by \cref{thick:tc-3}) and $O(P)$ times $[[x_i\ol{y}_i,x_k],[y_j^M,x_h^Q\ol{y}_h^Q]]$, whose area we estimate next.

		      We note that
		      \[
			      [[x_i\ol{y}_i,x_k],[y_j^M,x_h^Q\ol{y}_h^Q]]=
			      [[x_i\ol{y}_i,x_k],y_j^M(\ol{y}_j[y_j,x_h^Q\ol{y}_h^Q])^M]
		      \]
		      which follows by applying $O(M)$ times the element $[[x_i\ol{y}_i,x_k],y_j]$ and $O(M)$ times the element $[[x_i\ol{y}_i,x_k],[y_j,x_h^Q\ol{y}_h^Q]]$ (whose area is bounded  by $O(Q^2)$ by \cref{thin:quadruple-commutator-3}). This proves  \cref{thick:qc-3}. \qedhere
	\end{enumerate}
\end{proof}

\begin{proof}[Proof of \cref{ass:powers-3}]
	It is a direct consequence of \cref{prop:thick-3}.
\end{proof}

\subsection{Proof of \cref{prop:thick} for three factors}
We recall (and introduce) the following notation (translated to the three factor-setting, where $\bq=(q_1,\ldots,q_r),\bq'=(q_1',\ldots,q'_r)\in\ZZ^r$:
\begin{itemize}
	\item $\abs{\bq}=\max\{\abs{q_1},\dots ,\abs{q_r}\}$;
	\item $\A=\set{a_i,b_i,c_i,\text{ for } i\in\range r}$;
	\item $\mathcal C_r= \set{[a_i,b_j],[a_i,c_k],[b_j,c_k] \text{ for } i,j,k\in\range r}$;
	\item $\dirprod=\presentation \A {\mathcal{C}_r} $;
	\item $\interv \bq ij=(0, \dots ,0,q_i, \dots ,q_j,0, \dots ,0) \in \ZZ^r$;
	\item $\pax = x_1^{-q_1}\cdots x_r^{-q_r}$;
	\item $\pay = y_1^{-q_1}\cdots y_r^{-q_r}$;
	\item a push down map for the short exact sequence
	      \[1\rar \K\rar \dirprod\rar \ZZ^r\rar 1\]
	      is given by
	      \begin{align*}
		      \push_\bq (a_i) & =
		      y_1^{-q_1}\cdots y_{i-1}^{-q_{i-1}}x_iy_{i-1}^{q_{i-1}}\cdots y_1^{q_1} = \paroy {i-1} x_i \mparoy {i-1}; \\
		      \push_\bq (b_i) & =
		      x_1^{-q_1}\cdots x_{i-1}^{-q_{i-1}}y_ix_{i-1}^{q_{i-1}}\cdots x_1^{q_1} = \parox {i-1} y_i \mparox {i-1}; \\
		      \push_\bq (c_i) & =
		      y_1^{-q_1}\cdots y_{r}^{-q_{r}}\inv x_i y_{r}^{q_{r}}\cdots y_i^{q_i}x_iy_{i-1}^{q_{i-1}}\cdots y_1^{q_1} \\
		                      & = \pay \inv x_i\mparoy r[i]x_i\mparoy {i-1};
	      \end{align*}
	      and it satisfies $\push_\bq\left(uw\right)=
		      \push_\bq\left(u\right)\push_{\bq+\widetilde\psi(u)}(w)$ for any $u,w \in \free \A$, where the homomorphism $\widetilde \psi\colon F(\A)\to \ZZ^r$ is given by the composition
	      $F(\A)\longrightarrow \dirprod \overset{\phi^n_r}{\longrightarrow} \ZZ^r$.
	\item $\kappa_{\bq, \bq'} \colon \absfree {r+2}\rar \absfree r$ is defined as the homomorphism sending $\afgen {r+1}$ to ${\afgen 1}^{q_1}\cdots \afgen r^{q_r}$, $\afgen {r+2}$ to ${\afgen 1}^{q'_1} \cdots \afgen r ^{q'_r}$  and $\afgen i$ to $\afgen i$, for every $i\in \range [1]r$;
	\item Given $m>0$, the set of $m$-thick relations is
	      \[
		      \Rthick[m] \coloneqq\bigcup_{\|\phi\|\leq 1}\ \bigcup_{\abs{\bq},\abs{\bq'}\leq m+1} \simmetriz{\kappa_{\bq,\bq'}}\left(\simmetriz\phi(\R[r +2])\right).
	      \]
\end{itemize}

The aim of this section is to prove the following.
\begin{proposition}[\cref{prop:thick} for $n=3$]\label{prop:ass:thick-3}
	There is a constant $E$ (independent of $r$) such that, for every $\bq\in\bZ^r$ with $\abs{\bq}\leq m$ and $i,j,k\in\range r$, the following hold:%
	\begin{itemize}
		\item 	$\Area_{\Rthick[m]}\left(\push_\bq \left(\left[a_i,b_j \right]\right)\right)\le E$;
		\item 	$\Area_{\Rthick[m]}\left(\push_\bq \left(\left[a_i ,c_k \right]\right)\right)\le E$;
		\item 	$\Area_{\Rthick[m]}\left(\push_\bq \left(\left[b_j,c_k \right]\right)\right)\le E$.
	\end{itemize}
\end{proposition}

We start with a preliminary lemma.
\begin{lemma}\label{lem:thick-word-in-terms-of-thick-rel-3}
	For every $\bq, \bq' \in\bZ^r$, $i,j\in\range r$, and $ \epsilon, \delta, \sigma, \tau \in \set{\pm 1} $  we have
	\begin{enumerate}
		\item\label{tw1} $\Area_{\Rthick[m]}([x_i\inv y_i, [x_j^\epsilon,\pay^\delta ]])\leq C$;
		\item\label{tw3} $\Area_{\Rthick[m]}([\pax ,[(y_i \inv x_i)^\epsilon,\pay[\bq']^\delta]])\leq C$;
		\item\label{tw4} $\Area_{\Rthick[m]} ([ \pay \mpax, [x_j^\epsilon ,\pay[\bq']^\delta ]])\leq C$;
		\item\label{tw5} $\Area_{\Rthick[m]}([[(x_i\inv y_i)^\epsilon,\pay^\delta],[ x_j^\tau, \pay[\bq']^\sigma ]])\leq C$;
	\end{enumerate}
	for some constant $C$ independent of $r,\bq,\bq', i,j$.
\end{lemma}

\begin{proof}
	Since the proof is similar for each statement, we only prove the second (since it contains both $ \bq $ and $ \bq' $).

	Consider the expression $ [x_{r+1}, [(y_i \inv x_i)^\epsilon, y_{r+2}^\delta]] $. By \cref{lem:same-index-3}, this can be written as a product of a finite number of conjugates of relations in $\R[r+2]$, i.e.
	\[
		\left[x_{r+1}, \left[(y_i \inv x_i)^\epsilon, y_{r+2}^\delta\right]\right] =\prod_{k=1}^\ell \inv z_k R_k z_k
	\]
	where $z_s\in \free{\X[][r+2]}$ and $R_s\in \R[r+2]$. Let
	$\kappa_{\bq, \bq'}\colon \absfree {r+2}\to \absfree{r}$
	as defined above. Then applying $\widehat{\kappa}_{\bq,\bq'}$ to both sides of the equation gives
	\[
		\left[\pax ,\left[(y_i \inv x_i)^\epsilon,\pay[\bq']^\delta\right]\right] = \prod_{k=1}^\ell  \simmetriz\kappa_{\bq,\bq'}(\inv z_k) \simmetriz \kappa_{\bq,\bq'}(R_k) \simmetriz\kappa_{\bq,\bq'}(z_k),
	\]
	where, by definition, $\simmetriz \kappa_{\bq, \bq'}(R_k)\in\Rthick[m] $. This proves the claim.
\end{proof}

We split the proof of \cref{prop:ass:thick-3} into three lemmas, one for every item.
\begin{lemma}\label{lem:push-ab-3}
	For every $m\in \bZ$, $\bq\in\bZ^r$ with $|\bq|\leq m$,  and $i,j \in \range r$, we have
	\[
		\Area_{\Rthick[m]}(\push_\bq(a_ib_j\ol{a}_i\ol{b}_j))\le C
	\]
	for some constant $C$ independent of $m,r,\bq,i,j$.
\end{lemma}
\begin{proof}
	We split the proof in three cases.
	\begin{itemize}
		\item if $1\le i<j\le r$, then $\push_\bq(a_ib_j\ol{a}_i\ol{b}_j)$ is
		      \begin{equation*}
			      (x_i[\ol{x}_i,\paroy{i-1}])
			      (\pax [{\bq_{[1:j-1]}+e_i}]y_j\mpax [{\bq_{[1:j-1]}+e_i}])
			      ([\paroy{i-1},\inv x_i]\inv x_i)
			      (\parox {j-1}\inv y_j\mparox {j-1}),
		      \end{equation*}
		      (where we set $\paroy {i-1}=1$ if $i = 1$). Using $[\inv x_i,\paroy{i-1}][\parox {i-1},\inv y_i]\in\Rthick[m]$ twice, and that $\bx_{\bq_{[1:j-1]}+e_i}=\parox {i-1}\inv x_i\parox {j-1}[i]$, we obtain
		      \begin{multline*}
			      x_i\inv y_i\parox {i-1} y_i\inv x_i
			      \parox {j-1}[i] y_j \mparox {j-1}[i] x_i
			      \inv y_i \mparox {i-1} y_i \inv x_i
			      \parox {j-1}\inv y_j\mparox {j-1}=\\
			      x_i\inv y_i\parox {j-1} y_i\inv x_i
			      [[x_i\inv y_i,\mparox {j-1}[i]],y_j]
			      y_j x_i\inv y_i \mparox {j-1} y_i \inv x_i
			      \parox {j-1}\inv y_j\mparox {j-1}.
		      \end{multline*}
		      Using $[[x_i\inv y_i,\mparox {j-1}[i]],y_j]\in \Rthick[m]$,
		      one gets
		      \[
			      \parox {j-1}{} [[\mparox{j-1},x_i\inv y_i],y_j]\mparox {j-1},
		      \]
		      which is a conjugate of an element inside $\Rthick[m]$.

		\item
		      If $1\leq j<i\leq r$ the proof is analogous.

		\item If $1\leq i =j\leq r $, then $\push_\bq(a_ib_j\ol{a}_i\ol{b}_j)$ is
		      \[
			      x_i[\ol{x}_i,\paroy{i-1}]
			      [\pax [{\bq_{[1:i-1]}}],y_i] y_i
			      \inv x_i [x_i, \paroy{i-1}]
			      [\parox {i-1},\inv y_i]\inv y_i,
		      \]
		      and the thick relations $ [y_i \inv x_i, [\pax [{\bq_{[1:i-1]}}],y_i]]$, $ [y_i \inv x_i, [\ol{x}_i,\paroy{i-1}]] $, $ [\pax [{\bq_{[1:i-1]}}],y_i] [x_i, \paroy{i-1}]$, $ [\ol{x}_i,\paroy{i-1}] [\parox {i-1},\inv y_i]$, and $ [x_i, y_i]$ imply that it is trivial.

	\end{itemize}
\end{proof}

\begin{lemma}\label{lem:push-ac-3}
	For every $m\in\bZ$, $\bq\in\bZ^r$ with $|\bq|\leq m$, and $i,j\in \range r$, we have
	\[
		\Area_{\Rthick[m]}(\push_\bq(a_ic_j\ol{a}_i\ol{c}_j))\le C
	\]
	for some constant $C$ independent of $m,r,\bq,i,j$.
\end{lemma}

\begin{proof}
	We split the proof in three cases:
	\begin{itemize}
		\item If $1\leq i<j\le r$,  then $\push_\bq(a_ic_j\ol{a}_i\ol{c}_j)$ is
		      \begin{multline*}
			      \paroy {i-1} x_i\inv y_i
			      \paroy r[i] \inv x_j\mparoy r[j]x_j\mparoy {j-1}[i]y_i
			      \inv x_i \paroy {j-1}[i]
			      \inv x_j\paroy r[j]x_j\mpay =\\
			      =\paroy {j-1}{}
			      \left[x_i\inv y_i [y_i \inv x_i, \mparoy {j-1}[i]], [\paroy r[j], \inv x_j]\right] \mparoy {j-1}[1]
		      \end{multline*}
		      (where, if $i=1$, we set $\paroy {i-1}=1$), which becomes trivial via \cref{lem:thick-word-in-terms-of-thick-rel-3}.

		\item If $1 \leq j < i\le r$, then $\push_\bq(a_ic_j\ol{a}_i\ol{c}_j)$ is
		      \begin{multline*}
			      \paroy {i-1} x_i
			      \inv y_i \paroy r[i] \inv x_j\mparoy r[i]y_i \mparoy{i-1}[j] x_j
			      \inv y_j \paroy {i-1}[j] \inv x_i \mparoy {i-1}[j] y_j
			      \inv x_j\paroy r[j]x_j\mpay =\\
			      \paroy {i-1} x_i
			      \inv y_i \paroy r[i] \inv x_j\mparoy r[i]y_i \inv x_i
			      \left[x_i,x_j\inv y_j[y_j \inv x_j,\mparoy{i-1}[j]]\right]
			      \paroy r[i] x_j\mpay,
		      \end{multline*}

		      which, using $[x_i,[y_j \inv x_j,\mparoy{i-1}[j]]], \ [[x_i,x_j\inv y_j],
			      \paroy r[i] ]\in\Rthick[m]$, becomes
		      \begin{multline*}
			      \paroy {i-1} x_i
			      \inv y_i \paroy r[i] \inv x_j\mparoy r[i]y_i \inv x_i \paroy r[i]
			      \left[x_i,x_j\inv y_j\right]
			      x_j\mpay=\\
			      =
			      \pay  \inv x_j
			      \left[x_j, x_i\inv y_i [y_i\inv x_i,\mparoy r[i]]\right]
			      \left[x_i,x_j\inv y_j\right]
			      x_j\mpay,
		      \end{multline*}
		      which is trivial via $[x_j, [y_i\inv x_i,\mparoy r[i]]]\in\Rthick[m]$
		      and \cref{lem:same-index-3}.

		\item If $1 \leq i = j \leq r$,   then $\push_\bq(a_ic_i\ol{a}_i\ol{c}_i)$ is
		      \[
			      \paroy {i -1}\left[x_i \inv y_i,[\paroy r[i],\inv x_i]\inv x_i\right]\mparoy {i -1},
		      \]
		      which is trivial by $[x_i,y_i]$ and \cref{lem:thick-word-in-terms-of-thick-rel-3}.
	\end{itemize}

\end{proof}

\begin{lemma}\label{lem:push-bc}
	For every $\bq\in\bZ^r$ with $\abs{\bq}\leq m$ and $1\le i,j\le r$, we have
	\[
		\Area_{\Rthick[m]}(\push_\bq(b_ic_j\ol{b}_i\ol{c}_j))\le C
	\]
	for some constant $C$ independent of $m,r,\bq,i,j$.
\end{lemma}
\begin{proof}
	We only need to show that we can exchange $ x $ and $ y $ in the definition of $ \push_\bq(c_j) $ by using a bounded number of thick relations, that is,
	\[
		\Area_{\Rthick[m]}( \pay \inv x_i\mparoy r[i]x_i\mparoy {i-1}  \parox{i-1} \inv y_i \parox r[i] y_i \mpax ) \leq E
	\]
	for some universal constant $E$. After that, we can conclude in the same way as in \cref{lem:push-ac-3} by exploiting the symmetry in $ x $ and $y$.

	The expression above can be rewritten as
	\[
		\paroy {i-1}{} [\paroy r[i], \inv x_i]\mparoy {i-1}  \parox{i-1}{} [\inv y_i, \parox r[i]] \mparox{i-1}
	\]
	which becomes trivial after using the words $ [\paroy r[i], \inv x_i][\inv y_i, \parox r[i]] $ and $[ [\paroy r[i], \inv x_i], \mparoy {i-1}  \parox{i-1}{}]  $ (see \cref{lem:thick-word-in-terms-of-thick-rel-3} \cref{tw4}).
\end{proof}

%% file: sections-uniform/appendix-more-factors.tex
\section{Five factors or more}
\label{sec:appendix-5}

In this appendix we prove \cref{prop:symmetry,lem:basic-doubling,prop:powers,prop:thick} for $\K$ when the number of factors $n$ is at least $5$.
We fix an $n\geq 5$ throughout the whole section (so that we do not need to specify the number of factors in every set we define).
Recall from \cref{sec:general-strategy} that $\K$ is generated by
\[\X=\bigcup_{\alpha = 1}^{n-1}\Xii \alpha,\]
where
\[
	\Xii \alpha=\left(\xii \alpha 1,\dots,\xii \alpha r\right) .
\]

We consider the sets of trivial words
\begin{align*}
	\simplecomms^n & \coloneq \set*{\left[\xii \alpha i, \xii \beta i\right] \Suchthat
	i \in \range r,\ \alpha \neq \beta \in \range {n-1}}                                               \\
	\comms^n       & \coloneq \set*{\left[\xii \alpha i, \xii \beta j \mxii \gamma j \right] \Suchthat
		\begin{aligned}
			 & i \neq j \in \range r,                                            \\
			 & \alpha, \beta, \gamma  \in \range {n-1} \text{ pairwise distinct}
		\end{aligned}
	}
\end{align*}
and we claim
\[\K=\presentation \X \R ,\]
where
\[
	\R = \simplecomms^n \cup \comms^n.
\]

\subsection{Proof of \cref{prop:symmetry} for more factors}
Recall from \cref{def:homo-norm} that for a homomorphism
$\phi\colon\absfree r\to \absfree{r'}$
the norm $\norma{\phi}$ is the maximum of the lengths of $\phi(\afgen 1),\ldots,\phi(\afgen r)$ as reduced words in $\afgen 1,\ldots, \afgen{r'}$ (and their inverses). Notice that $\norma{\oc\phi}=\norma{\phi}$.

\begin{proposition}[\cref{prop:symmetry} for $n\geq 5$]\label{prop:ass:simmetry-n-big}
	There exists a constant $A_1>0$ such that the following holds:
	for all integers $r,r'\ge1$ and every homomorphism of free groups $\phi \colon \absfree r\rar\absfree {r'}$  with $\norma{\phi}\le 1$, it holds that
	\[
		\Area_{\R [r']}\left(\simmetriz\phi(\R) \right)\le A_1.
	\]
\end{proposition}

The proposition is a consequence of the following lemma.
\begin{lemma}\label{lem:same-index-n-grande}
	There is a constant $C$ (independent of $r$) such that the following holds.
	For every $i,j\in\range r$ (not necessarily distinct), $\alpha\neq \beta\in \range {n-1}$ and $\epsilon,\delta\in \set{\pm1}$, we have
	\[
		\Area_{\R}\left(
		\left[\left(\xii \alpha i\right)^\epsilon,\left(\xii \beta j\right)^\delta\left( \mxii \gamma j\right)^\delta\right]
		\right)\leq C.\]
\end{lemma}

\begin{proof}
	We observe that, up to applying the commutator $\left[\xii \beta j,\xii \gamma j\right]$ (and renaming $\gamma$ and $\beta$ accordingly), we can assume $\epsilon=\delta=1$.

	When $i\neq j$, the statement is trivial.
	If $i=j$, then
	\[
		\left[\xii \alpha i,\xii \beta i\mxii \gamma i\right]=
		\left[\xii \alpha i,\xii \beta i\right]\left(\xii \beta i\xii \alpha i\left[\mxii \gamma i,\mxii \alpha i\right]  \mxii \alpha i \mxii \beta i\right),\]
	proving the statement.
\end{proof}

\begin{proof}[Proof of \cref{prop:ass:simmetry-n-big}]
	Fix two integers $r,r'\geq 1$.
	Let $R\in\R$. If $R\in \simplecomms$, then $\simmetriz\phi(R)$ is either trivial or in $\simplecomms$, and there is nothing to prove.
	If $R\in \comms$, then let $i_1,{i_2}$ be the indices involved in $R$, where $1\le i_1<i_2\le r$.
	If $\phi(x_{i_j})=1$ for some $1\le j\le 2$ then $\simmetriz\phi(R)=1$ and we are done.
	Otherwise, the conclusion follows by \cref{lem:same-index-n-grande}.
\end{proof}

\subsection{Proof of \cref{lem:basic-doubling} for more factors}
The purpose of this section is to prove the following proposition.
\begin{proposition}[\cref{lem:basic-doubling} for $n\geq 5$]\label{prop:ass:doubling-more-factors}
	There exists a constant $A_2>0$ such that the following holds:
	for every integer $r\ge1$, consider the homomorphism $\rho_r \colon \absfree r\rar \absfree {r+1}$ given by $\rho_r(\afgen 1)=\afgen1\afgen2$ and $\rho_r(\afgen i)={\afgen {i+1}}$ for $i \in \range [2]r$. Then we have
	\[
		\Area_{\R[r+1]}(\simmetriz\rho(\R))\le A_2.
	\]
\end{proposition}
The proposition is a consequence of the following two lemmas.

\begin{lemma}[Doubling for $\simplecomms$]\label{double:small-commutator-n-grande}
	There is a constant $C$ (independent of $r$) such that, for every $i,i'\in \range  r$ and $\alpha,\beta \in \range{n-1}$, we have
	\[\Area_{\R}\left(\left[\xii \alpha i\xii \alpha{i'},\xii \beta i\xii \beta {i'}\right]\right)\le C .\]
\end{lemma}

\begin{proof}
	For each $k\in\range r$, we set $x_k=\xii\alpha k$, $y_k=\xii\beta k$ and $z_{i'}=\xii \gamma {i'}$ for some $\gamma\in \range{n-1} \setminus\set{\alpha,\beta}$.
	Using the relation $[x_{i'},y_{i'}]$, the element $[x_ix_{i'},y_iy_{i'}]$ becomes
	\[
		x_ix_{i'}y_i\inv x_{i'}y_{i'}\inv x_i\inv y_{i'} \inv y_{i}=
		x_ix_{i'}(y_i\inv x_{i'}z_{i'})(\inv z_{i'}y_{i'}\inv x_i)\inv y_{i'} \inv y_{i}.
	\]
	Using $[y_i,z_{i'}\inv x_{i'}]$ and $[x_i,y_{i'}\inv z_{i'}]$, we deduce
	\[
		x_iz_{i'}y_i\inv x_i\inv z_{i'} \inv y_{i},
	\]
	which is trivial with $[z_{i'},y_i \inv x_i]$ and $[x_i,y_i]$.
\end{proof}

\begin{lemma}[Doubling for $\comms$]\label{double:commutator-more-factors}
	There is a constant $C$ (independent of $r$) such that, for every $i,i',j\in \range  r$ pairwise distinct, and $\alpha,\beta, \gamma\in \range {n-1}$ pairwise distinct, we have
	\[\Area_{\R}\left(\left[\xii \alpha j,\xii \beta i\xii \beta{i'} \mxii \gamma {i'}\mxii \gamma i\right]\right)\leq C.\]
\end{lemma}
\begin{proof}
	For each $k\in\range r$, we set $x_k=\xii\alpha k$, $y_k=\xii\beta k$ and $z_k= \xii \gamma k$.
	We observe that
	\begin{equation*}
		[x_j,y_iy_{i'}\inv z_{i'}\inv z_i]
		=[x_j,y_i \inv z_i [z_i,y_{i'} \inv z_{i'}]y_{i'}\inv z_{i'}].
	\end{equation*}
	Using $[x_j,y_i\inv z_i]$ this becomes
	$[x_j,y_{i'}\inv z_{i'}]$,
	\begin{equation*}
		y_i\inv z_i([x_j,[z_i,y_{i'} \inv z_{i'}]])z_i\inv y_i=
		y_i\inv z_i(x_jz_iy_{i'}\inv z_{i'}\inv z_i[z_{i'}\inv y_{i'},\inv x_j]\inv x_jz_iz_{i'}\inv y_{i'}\inv z_i)z_i\inv y_i.
	\end{equation*}
	By conjugating by $y_i\inv z_i$ and applying the relation $[x_j,y_{i'}\inv z_{i'}]$ one gets
	\[
		x_jz_iy_{i'}\inv z_{i'}\inv z_i\inv x_jz_iz_{i'}\inv y_{i'}\inv z_i=
		x_jz_iy_{i'}\inv z_{i'}(\inv z_i[\inv x_j,z_i\inv w_i]z_i)z_{i'}\inv y_{i'}\inv w_i\inv x_jw_i[\inv w_i x_jw_i,y_{i'}\inv z_{i'}]\inv z_i
	\]
	where $w_i=\xii \delta i$ for some $\delta \in \range{n-1}\setminus\set{\alpha,\beta,\gamma}$. And now, using several relations in $\comms$, one gets
	\[      x_jz_i\inv w_i\inv x_jw_i\inv z_i \]
	which is once more a relation in $\comms$.
\end{proof}

\begin{proof}[Proof of \cref{prop:ass:doubling-more-factors}]
	Let $R\in\R$. If $1$ is not an index involved in $R$, then $\simmetriz\rho(R)\in\R[r+1]$ and we are done. Otherwise, the conclusion follows by \cref{double:small-commutator-n-grande,double:commutator-more-factors}.
\end{proof}

\subsection{Proof of \cref{prop:powers}}
In this section we aim to prove the following proposition.
\begin{proposition}[\cref{prop:powers} for $n \geq 5$]\label{prop:ass:powers-more-factors}
	There exists a constant  $C>0$ such that the following happens: let $r \geq 2$ and $N_1,\ldots,N_r$ be integers and consider the homomorphism
	\[\omega=\omega_{N_1,\ldots,N_r}\colon \absfree r\rar  \absfree r\]
	given by $\omega(\afgen i)=\afgen i^{N_i}$ for $i \in \range r$. Then,
	\[
		\Area_{\R}\left(\simmetriz\omega(\R)\right)\le C(\max\{\abs {N_1},\ldots, \abs {N_r}\})^{2} .
	\]
\end{proposition}

As usual, we split the proof into two lemmas, one for each set of relations.

\begin{lemma}[Power relations for $\simplecomms$]\label{power:simple-commutator-more-factors}
	For all integers $\N$, all $ i\in\range r$, and all $\alpha\neq \beta \in \range{n-1}$, we have:
	\[\Area_{\R }\left(\left[\left(\xii \alpha i\right)^\N, \left(\xii\beta i\right)^\N\right]\right)\le \abs\N^2.\]
\end{lemma}
\begin{proof}
	Immediate.
\end{proof}
\begin{lemma}[Power relations for $\comms$]\label{power:commutators-more-factors}
	There is a constant $C$ (independent of $r$) such that, for all integers $\N$, all $\alpha,\beta,\gamma\in \range{n-1}$ pairwise distinct, and all $ i\neq j\in\range r$, we have
	\[ \Area_{\R} \left(\left[\left(\xii \alpha i\right)^M,\left(\xii \beta j\right)^\N\left(\mxii \gamma j\right)^\N\right]\right)\leq C\left(\abs M\abs \N+\abs \N^2\right) . \]
\end{lemma}
\begin{proof}
	Let $x_i=\xii \alpha i$, $y_j=\xii \beta j$ and $z_j=\xii \gamma j$. We observe that the element $[x_i^M,y_j^\N\inv z_j^\N]$, after using $O(\N^2)$ times the relation $[y_j ,z_j ]$, becomes $ [x_i,(y_j\inv z_j)^N]$, which becomes trivial by applying $\abs M\abs \N$ times the relation $ [x_i,y_j\inv z_j]$.
\end{proof}

\begin{proof}[Proof of \cref{prop:ass:powers-more-factors}]
	The result is a direct consequence of \cref{power:simple-commutator-more-factors,power:commutators-more-factors}.
\end{proof}

\subsection{Proof of \cref{prop:thick} for more factors}\label{sec:push-more-factors}
Recall the following notation, where $\bq=(q_1,\ldots,q_r)\in\ZZ^r$, $\abs{\bq}=\max\{\abs{q_1},\ldots,\abs{q_r}\}$, and $\permut\colon\range {n-1}\to \range {n-1}$ is any map without fixed points:
\begin{itemize}
	\item $\A=\left\{\aii \alpha i,\text{ for } i\in\range r,\ \alpha\in \range{n}\right\}$;
	\item $\mathcal C_r= \left\{\left[\aii \alpha i,\aii \beta j\right], \text{ for } i,j\in\range r,\ \alpha,\beta \in \range {n}\right\}$;
	\item $\dirprod=\presentation \X {\mathcal{C}_r} $;
	\item $\interv \bq ij=(0,\ldots,0,q_i,\ldots,q_j,0,\ldots,0) \in \ZZ^r$;
	\item $\pa \alpha  = \left(\xii \alpha 1 \right)^{-q_1}\cdots \left(\xii \alpha r \right)^{-q_r}$;
	\item a push down map for the short exact sequence
	      \[1\rar \K\rar \dirprod\rar \ZZ^r\rar 1\]
	      is defined on the generators (where $\alpha \in \range {n-1}$) by
	      \begin{align*}
		      \push_\bq \left(\aii \alpha j\right) & =
		      \paro {\permut (\alpha)}{j -1} \xii \alpha j\mparo {\permut (\alpha )}{j -1}; \\
		      \push_\bq \left(\aii nj\right)       & =
		      \pa 2 \mxii 1j \mparo 2r[j]\xii 1j \mparo 2{j -1}.
	      \end{align*}
	      It satisfies $\push_\bq\left(uw\right)=
		      \push_\bq\left(u\right)\push_{\bq+\widetilde\psi(u)}(w)$ for any $u,w \in \free \A$, where the homomorphism $\widetilde\psi \colon F(\A)\to \ZZ^r$ is given by the composition
	      $F(\A)\longrightarrow \dirprod \overset{\phi^n_r}{\longrightarrow} \ZZ^r$.
	\item $\kappa_{\bq, \bq'} \colon \absfree {r+2}\rar \absfree r$ is defined as the homomorphism sending $\afgen {r+1}$ to ${\afgen 1}^{q_1}\cdots \afgen r^{q_r}$, $\afgen {r+2}$ to ${\afgen 1}^{q'_1} \cdots \afgen r ^{q'_r}$  and $\afgen i$ to $\afgen i$, for every $i\in \range [1]r$;
	\item Given $m>0$, the set of $m$-thick relation is
	      \[
		      \Rthick[m] \coloneqq\bigcup_{\|\phi\|\leq 1}\ \bigcup_{\abs{\bq},\abs{\bq'}\leq m+1} \simmetriz{\kappa_{\bq,\bq'}}\left(\simmetriz\phi(\R[r +2])\right).
	      \]%
\end{itemize}

The aim of this section is to prove the following.
\begin{proposition}[\cref{prop:thick} for $n\geq 5$]\label{prop:ass:thick-more-factors}
	There is a constant $E$ (independent of $r$) such that, for every $\bq\in\bZ^r$ with $\abs{\bq}\leq m$, $\alpha\neq \beta\in \range n$  and $i,j\in\range r$, the following holds:
	\[\Area_{\Rthick[m]}\left(\push_\bq \left(\left[\aii\alpha i,\aii\beta j \right]\right)\right)\le E  . \]
\end{proposition}

We start with the following lemma stating that, up to applying a finite number of thick relations, the push-down map does not depend on the choice of the fixed-point-free map $\sigma$.
\begin{lemma}\label{lem:cambio-lettera-more-factors}
	There exists a constant $C$ (independent of $r$) such that,
	for every $\alpha,\beta,\gamma \in \range {n -1}$ pairwise distinct, every positive integer $m$ and every $\bq, \bq' \in \ZZ$ satisfying $\abs \bq, \abs{\bq'} \leq m$, it holds that
	\[
		\Area_{\Rthick[m]}\left(\left(\pa {\beta} \pa{\alpha}[\bq'] \mpa {\beta}\right) \left(\pa {\gamma} \pa {\alpha}[\bq'] \mpa {\gamma}\right)^{-1} \right) \leq C .
	\]
\end{lemma}

\begin{proof}
	It is enough to observe that
	\[\left(\pa {\beta} \pa\alpha[\bq'] \mpa {\beta}\right) \left(\pa {\gamma} \pa\alpha[\bq'] \mpa {\gamma}\right)^{-1} =
		\pa {\beta}\left[\pa\alpha[\bq'],\pa {\gamma}\mpa {\beta}\left[\pa \beta,\mpa \gamma\right] \right] \mpa {\beta} ,\]
	and then use that $\left[\pa\alpha[\bq'],\pa {\beta}\mpa {\gamma} \right], \left[\pa \beta,\mpa \gamma\right] \in \Rthick[m]$.
\end{proof}

\begin{lemma}\label{lem:cambio-n-lettera-more-factors}
	There exists a constant $C$ (independent of $r$) such that,
	for every $\alpha\ne\beta \in \range {n -1}$, every positive integer $m$ and every $\bq\in \ZZ$ satisfying $|\bq|\leq m$, it holds that
	\[
		\Area_{\Rthick[m]}\left(\left( \pa 2 \mxii 1j \mparo 2r[j]\xii 1j \mparo 2{j -1}, \right) \left( \pa \beta \mxii \alpha j \mparo \beta r[j]\xii \alpha j \mparo \beta {j -1}, \right)^{-1} \right) \leq C.
	\]
\end{lemma}
\begin{proof}
	If $ \alpha \neq 2 $, then we can apply \cref{lem:cambio-lettera-more-factors} to $\pa 2 (\mxii 1j \mparo 2r[j]\xii 1j) \mparo 2{j -1}$ to obtain%
	\[
		\pa 2 \mxii \alpha j \mparo 2 r[j]\xii \alpha j \mparo 2 {j -1}=
		\pa 2 \mxii \alpha j \mpa 2 \paro 2 {j-1}\xii \alpha j \mparo 2 {j -1}.
	\]
	Using \cref{lem:cambio-lettera-more-factors} once more, this is equivalent to
	\[
		\pa \beta \mxii \alpha j \mpa \beta \paro \beta {j-1}\xii \alpha j \mparo \beta {j -1} .
	\]
	Thus, $\left(\pa 2 (\mxii 1j \mparo 2r[j]\xii 1j) \mparo 2{j -1}\right)
		\left(\pa \beta \mxii \alpha j \mpa \beta \paro \beta {j-1}\xii \alpha j \mparo \beta {j -1}\right)^{-1} $
	can be filled with a finite number of thick relations.

	If $\alpha=2$ and $\beta \neq 1 $, we can apply the same transformations as above in reversed order, starting by replacing $ 2 $ by $\beta$ and after that $ 1 $ by $ \alpha $.

	If $ \alpha=2 $ and $ \beta=1 $, then we similarly first replace $1$ by $3$, then $2$ by $1=\beta$, and finally $3$ by $2=\alpha$.
\end{proof}

\begin{proof}[Proof of \cref{prop:ass:thick-more-factors}]
	First, assume that $\alpha,\beta\neq n$ and $i < j$. The term $\push_\bq \left(\left[\aii \alpha i,\aii \beta j\right]\right)$ is then given by
	\[
		\left(\paro {\permut(\alpha)} {i-1} \xii {\alpha}i \mparo {\permut(\alpha)}{i-1}\right)
		\left(\paro {\permut(\beta)} {j-1}[1] [e_i +\bq] \xii {\beta}j \mparo {\permut(\beta)}{j-1} [1][e_i+\bq]\right)
		\left(\paro {\permut(\alpha)} {i-1} \mxii \alpha i \mparo {\permut(\alpha)}{i-1}\right)
		\left(\paro {\permut(\beta)} {j-1} \mxii {\beta}j \mparo {\permut(\beta)}{j-1}\right).
	\]
	Thanks to  \cref{lem:cambio-lettera-more-factors}, we may assume (after applying a fixed number of thick relations) $\gamma\coloneq\permut(\alpha)=\permut (\beta)\in \range {n-1}\setminus \set{\alpha,\beta}$.
	Thus the previous expression becomes
	\[
		\paro {\gamma} {i-1} \xii {\alpha}i \mxii {\gamma}i
		\left(\paro {\gamma} {j-1}[i] \xii {\beta}j  \mparo {\gamma}{j -1}[i]\right) \xii {\gamma}i
		\mxii \alpha i\paro {\gamma} {j-1} [i] \mxii {\beta}j \mparo {\gamma}{j-1}.
	\]
	Let $\delta \in \range{n-1}\setminus\set{\alpha,\beta,\gamma}$ (this choice is possible due to our hypothesis $n\geq 5$).
	After applying \cref{lem:cambio-lettera-more-factors}, we obtain
	\[
		\paro {\gamma} {i-1} \xii {\alpha}i \mxii {\gamma}i
		\left(\paro {\delta} {j-1}[i] \xii {\beta}j  \mparo {\delta}{j -1}[i]\right) \xii {\gamma}i
		\mxii \alpha i \paro {\gamma} {j-1} [i] \mxii {\beta}j \mparo {\gamma}{j-1}.
	\]
	Using $\left[\xii {\alpha}i \mxii {\gamma}i, \paro {\delta} {j-1}[i],\right]$,
	$\left[\xii {\alpha}i \mxii {\gamma}i,	\xii {\beta}j\right]\in\Rthick[m]$,
	this becomes
	\[
		\paro {\gamma} {i-1}
		\left(\paro {\delta} {j-1}[i] \xii {\beta}j  \mparo {\delta}{j -1}[i]\right)
		\paro {\gamma} {j-1} [i] \mxii {\beta}j \mparo {\gamma}{j-1},
	\]
	which becomes trivial after applying  \cref{lem:cambio-lettera-more-factors} once more.

	If $i>j$, then it is enough to notice that $\left[\aii \alpha i,\aii \beta j\right]=\left[\aii \beta j,\aii \alpha i\right]^{-1}$ and, thus, the statement follows from the previous case.

	If $i=j$, then by applying \cref{lem:cambio-lettera-more-factors} as in the previous cases, we get that
	$\push_\bq \left(\left[\aii \alpha i, \aii \beta i\right]\right)$ is equivalent to
	\[
		\left(\paro {\gamma} {i-1} \xii {\alpha}i \mparo {\gamma}{i-1}\right)
		\left(\paro {\gamma} {i-1}[1] [\bq] \xii {\beta}i \mparo {\gamma}{j-1} [1][\bq]\right)
		\left(\paro {\gamma} {i-1} \mxii \alpha i \mparo {\gamma}{i-1}\right)
		\left(\paro {\gamma} {i-1} \mxii {\beta}i \mparo {\gamma}{i-1}\right),
	\]
	which is a conjugate of a relation in $\R$.

	We are left with considering the case $ \alpha < \beta = n $. Assume $ i < j $. By using \cref{lem:cambio-n-lettera-more-factors} we may replace $ \push_{ \bq }(\aii n j) $ by $ \pa {\permut(\alpha)}\mxii \gamma j \mparo {\permut(\alpha)}r[j] \xii \gamma j \mparo {\permut(\alpha)}{j-1} $, where $ \gamma \in \range {n-1} $ is different from $ \alpha, \permut(\alpha) $.
	Then $ \push_\bq\left([\aii \alpha i, \aii n j ]\right) $ is

	\begin{multline*}
		\left(\paro {\permut(\alpha)} {i-1} \xii \alpha i \mparo {\permut(\alpha)}{i-1}\right)
		\left(\pa {\permut(\alpha)}[e_i+\bq]  \mxii \gamma j \mparo {\permut(\alpha)}r[j]\xii \gamma j \mparo {\permut(\alpha)}{j-1}[1][e_i + \bq] \right) \cdot \\
		\cdot\left(\paro {\permut(\alpha)} {i-1} \mxii \alpha  i \mparo {\permut(\alpha)}{i-1}\right)
		\left( \paro {\permut(\alpha)}{j-1}\mxii \gamma j \paro {\permut(\alpha)}r[j]  \xii \gamma j  \mpa {\permut(\alpha)} \right),
	\end{multline*}
	which can be rewritten as
	\[
		\paro {\permut(\alpha)}{i-1}[1] \left[\xii \alpha i \mxii {\permut(\alpha)} i , \paro {\permut(\alpha)} {j-1}[i]\left[\paro {\permut(\alpha)} r[j], \mxii \gamma j\right] \mparo {\permut(\alpha)}{j-1}[i] \right] \mparo {\permut(\alpha)}{i-1}[1] .
	\]

	Using $ [\mparo {\permut(\alpha)} r[j] \paro {\alpha} r[j], \mxii \gamma j ] $, one obtains
	\[
		\paro{\permut(\alpha)}{i-1}[1] \left[\xii \alpha i \mxii {\permut(\alpha)} i , \paro {\permut(\alpha)} {j-1}[i]\left[\paro {\alpha} r[j], \mxii \gamma j\right] \mparo {\permut(\alpha)}{j-1}[i] \right] \mparo{\permut(\alpha)}{i-1}[1] .
	\]
	For $\delta \in \range{n-1}\setminus\set{\alpha,\beta,\gamma}$ we apply the thick relations $ [\mparo {\permut(\alpha)} {j-1}[i] \paro {\delta} {j-1}[i],  \paro {\alpha} r[j]] $ and $ [\mparo {\permut(\alpha)} {j-1}[i] \paro {\delta} {j-1}[i],  \mxii \gamma j ] $, to deduce
	\[ \paro{\permut(\alpha)}{i-1}[1] \left[\xii \alpha i \mxii {\permut(\alpha)} i , \paro {\delta} {j-1}[i]\left[\paro {\alpha} r[j], \mxii \gamma j\right] \mparo {\delta}{j-1}[i] \right] \mparo{\permut(\alpha)}{i-1}[1] ,                           \]
	which, using $ [\mparo {\alpha} r[j] \paro {\delta} r[j], \mxii \gamma j ] $, becomes
	\[
		\paro{\permut(\alpha)}{i-1}[1] \left[\xii \alpha i \mxii {\permut(\alpha)} i , \paro {\delta} {j-1}[i]\left[\paro {\delta} r[j], \mxii \gamma j\right] \mparo {\delta}{j-1}[i] \right] \mparo{\permut(\alpha)}{i-1}[1]     .
	\]
	Then we conclude since $ \xii \alpha i \mxii {\permut(\alpha)} i  $ commutes with each factor of the second term of the commutator via thick relations.

	The cases $ i>j $ and $ i=j $ are similar.

\end{proof}

%% file: sections-uniform/appendix-four-factors.tex
\section{Four factors}
\label{sec:appendix-four-factors}

We conclude with the remaining case, that is when the number of factors is $n=4$. We left it for last since it presents some similarities with both the case $ n=3 $ and $ n \geq 5 $, so we will be able to reuse some arguments from these two sections.

Our purpose is to prove \cref{prop:symmetry,lem:basic-doubling,prop:powers,prop:thick} for $\K$.
Recall from \cref{sec:general-strategy} that $\K$ is generated by
\[\X=\Xii 1 \cup \Xii 2 \cup \Xii 3\]
where
\[
	\Xii \alpha=\left(\xii \alpha 1,\dots,\xii \alpha r\right) .
\]
for $ \alpha \in \range 3$.

We consider the sets of trivial words
\begin{align*}
	\simplecomms & \coloneq \set*{\left[\xii \alpha i, \xii \beta i\right] \Suchthat
	i \in \range r,\ \alpha \neq \beta \in \range 3}                                                 \\
	\comms       & \coloneq \set*{\left[\xii \alpha i, \xii \beta j \mxii \gamma j \right] \Suchthat
		\begin{aligned}
			 & i \neq j \in \range r,                   \\
			 & \set{\alpha, \beta, \gamma} = \range {3}
		\end{aligned}
	}                                                                                                \\
	\triplecomms & \coloneq \set*{
		\left[\xii \alpha i, \left[\left(\xii \beta j\right)^\epsilon, \left(\xii \alpha k\right)^\delta \left(\mxii \beta k\right)^\delta \right]\right]
		\Suchthat
		\begin{aligned}
			 & i, j, k \in \range r \text{ pairwise distinct},                   \\
			 & \epsilon, \delta \in \set{\pm 1},\ \alpha \neq \beta \in \range 3
		\end{aligned}
	}
\end{align*}
and we claim
\[\K=\presentation \X \R ,\]
where
\[
	\R = \simplecomms^n \cup \comms^n.
\]

Since the number of factors is small, we will often write $ x_i, y_j, z_k $ instead of $ \xii 1 i, \xii 2 j, \xii 3 k $ to keep a simplify the notation, especially in the proofs that involve a lot of computations.

\subsection{Useful trivial words}
We start by noting that all the relations we had in the three factors case are still trivial in the presentation given above. This allows us to recycle several proofs from the three factor case (\cref{sec:appendix-three-factors}).

\begin{lemma}\label{lem:trivial-words-4}
	There is a constant $C$ (independent of $r$) such that
	for every $\alpha, \beta, \gamma \in \range 3$ pairwise distinct,
	$ i,j,k,h \in \range r $ pairwise distinct,
	and $\epsilon,\delta,\sigma,\tau\in \set{\pm1}$, we have:
	\begin{enumerate}
		\item $
			      \Area_{\R}\left(\left[\left(\xii \alpha i\right)^\epsilon,\left(\xii \beta j\right)^\delta\right]\left[\left(\xii \alpha j\right)^\delta,\left(\xii \beta i\right)^\epsilon\right]\right)\leq C $
		      for $ i\neq j \in \range r $;
		\item $
			      \Area_{\R}\left(\left[\xii \gamma i, \left[\left(\xii \beta j\right)^\epsilon, \left(\xii \alpha k\right)^\delta \left(\mxii \beta k\right)^\delta \right]\right]\right)\leq C$;
		\item $
			      \Area_{\R}\left(\left[\left[\left(\xii \alpha i\right)^\epsilon,\left(\xii \alpha k\right) ^\delta\left(\mxii \beta k\right)^\delta\right],\left[\left(\xii \beta j\right)^\sigma,\left(\xii \alpha h\right)^\tau\left(\mxii \beta h\right)^\tau\right]\right]\right)\leq C$.
	\end{enumerate}

\end{lemma}

\begin{proof}
	For sake of simplicity, assume $ \delta=\epsilon=\tau=\sigma=1 $ and $ \alpha=1, \beta=2, \gamma=3 $ (the other cases are completely analogous).
	\begin{enumerate}
		\item Note that
		      \[
			      [x_i, y_j][x_j, y_i] = x_i z_j (\inv z_j y_j) \inv x_i (\inv y_j z_j) (\inv z_j x_j) y_i (\inv x_j z_j) \inv z_j \inv y_i
		      \]
		      Using $ [\inv z_j y_j, x_i] $ and $ [\inv z_j x_j, y_i] $ this becomes
		      \[
			      y_i [\inv y_i x_i, z_j] \inv y_i
		      \]
		      which is conjugate to a relation, proving this assertion.

		\item We have
		      \[
			      [y_j,x_k \inv y_k]=[y_j,z_k \inv y_k],
		      \]
		      using the relations $[x_k,y_k]$, $[\inv y_j, \inv z_k x_k]$ and $[z_k,y_k]$.

		      Thus, the element
		      \[
			      [z_i, [y_j, x_k \inv y_k]]
		      \]
		      becomes $ [z_i, [y_j, z_k \inv y_k]] $, which is a relation of $ \triplecomms $. This proves the second assertion.
		\item
		      Since $\inv x_i$ commutes with $y_k\inv z_k$, we obtain
		      \[
			      [x_i, x_k \inv y_k] = [x_i, x_k \inv z_k],
		      \]
		      implying that
		      \[
			      [[x_i, x_k \inv y_k], [y_j, x_h \inv y_h]] = [[x_i, x_k \inv z_k], [y_j, x_h \inv y_h]]
		      \]
		      We thus obtain the last statement from the fact that $ x_i, x_k $ and $ z_k $ commute with $ [y_j, x_h \inv y_h] $.
	\end{enumerate}
\end{proof}

\begin{lemma}\label{lem:recycle-three-factors}
	Let $ \alpha \neq \beta \in \range 3 $, and consider the embedding $\iota \colon \K[3] \hookrightarrow \K $ given by sending $ x_i \mapsto \xii \alpha i, y_j \mapsto \xii \beta j $. There exists a constant $ C > 0 $ such that, if $ w \in \K[3]$, then
	\[
		\Area_{\R} (\iota(w)) \leq C \cdot \Area_{\R[r][3]} (w).
	\]
\end{lemma}%
\begin{proof}
	Note that all relations in $ \R[r][3] $ are either relations of $ \R $, or they are products of conjugates of relations in $ \R $ by \cref{lem:trivial-words-4}. We conclude by \cref{lem:area-bound-after-homo}.
\end{proof}

\subsection{Proof of \cref{prop:symmetry} for four factors}
The following proposition is a consequence of the previous section.

\begin{proposition}[\cref{prop:symmetry} for $n = 4$]\label{prop:ass:simmetry-4}
	There exists a constant $A_1>0$ such that the following holds:
	for all integers $r,r'\ge1$ and every homomorphism of free groups $\phi \colon \absfree r\rar\absfree {r'}$  with $\norma{\phi}\le 1$,
	\[
		\Area_{\R [r']}\left(\simmetriz\phi(\R) \right)\le A_1.
	\]
\end{proposition}

\begin{proof}
	Let $ R \in \R $. If the relation belongs to $ \simplecomms $ or $ \triplecomms $, then without loss of generality, up to a permutation of the indices, we may assume that $ R \in \R[r][3] $. The conclusion then follows from \cref{lem:recycle-three-factors}. If it belongs to $ \comms $, then the proof is identical to the proof of \cref{lem:same-index-n-grande}.
\end{proof}

\subsection{Proof of \cref{lem:basic-doubling} for four factors}
We now prove the doubling lemma.
\begin{proposition}[\cref{lem:basic-doubling} for $n = 4$]\label{prop:ass:doubling-4}
	There exists a constant $A_2>0$ such that the following holds:
	for every integer $r\ge1$, consider the homomorphism $\rho_r \colon \absfree r\rar \absfree {r+1}$ given by $\rho_r(\afgen 1)=\afgen1\afgen2$ and $\rho_r(\afgen i)={\afgen {i+1}}$ for $i \in \range [2]r$. Then we have
	\[
		\Area_{\R[r+1]}(\simmetriz\rho(\R))\le A_2.
	\]
\end{proposition}

\begin{proof}
	Let $ R \in \R$. If $ R $ is in $ \simplecomms $ or $ \triplecomms $, we conclude by combining \cref{lem:basic-doubling-3,lem:recycle-three-factors}.

	Otherwise, we claim that $ [x_j,y_iy_{i'}\inv z_{i'}\inv z_i] $ is trivial in $ \presentation\X\R $, for $ i, i', j \in \range r $.
	We have
	\[
		[x_j,y_iy_{i'}\inv z_{i'}\inv z_i]
		=[x_j,y_i \inv z_i [z_i,y_{i'} \inv z_{i'}]y_{i'}\inv z_{i'}],
	\]
	that becomes,  using $[x_j,y_i\inv z_i]$ and
	$[x_j,y_{i'}\inv z_{i'}]$,
	\[
		y_i\inv z_i([x_j,[z_i,y_{i'} \inv z_{i'}]])z_i\inv y_i
	\]
	which is trivial by \cref{lem:trivial-words-4}.

	Now, if $ R \in \comms $, either one of the indices involved is $1$ and the conclusion follows from the claim, or the conclusion is trivial.
\end{proof}

\subsection{Proof of \cref{prop:powers} for four factors}
In this section we aim to prove the following proposition.
\begin{proposition}[\cref{prop:powers} for $n = 4$]\label{prop:powers-4}
	There exists a constant  $C>0$ such that the following holds: let $r\geq 2$ and $N_1,\ldots,N_r\ge1$ be integers, and let
	\[\omega=\omega_{N_1,\ldots,N_r}\colon \absfree r\rar  \absfree r\]
	be the homomorphism given by $\omega(\afgen i)=\afgen i^{N_i}$ for $i \in \range r$. Then
	\[
		\Area_{\R}\left(\simmetriz\omega(\R)\right)\le C(\max\{\abs {N_1},\ldots, \abs {N_r}\})^{3} .
	\]
\end{proposition}

The proof is essentially the same as in the case $ n=3 $, with the difference that we can exploit the extra factor to obtain more efficient bounds, as follows.
\begin{lemma}\label{power:tc-4}
	There exists a constant $C$ (independent of $r$) such that, for all integers $\N,M,P$ and all $i,j,k\in \range r$, we have
	\[\Area_{\R}\left(\left[\left(\xii \alpha i\right)^N,\left[\left(\xii \beta j\right)^{M},\left(\xii \alpha k\right)^{P}\left(\mxii \beta k\right)^{P}\right]\right]\right)\le C\abs P\left(\abs P + \abs M^2  +\abs N\abs M\right).\]
\end{lemma}

\begin{proof}
	Without loss of generality, we assume $ x_i = \xii \alpha i, y_j = \xii \beta j $ and $ N,M,P >0 $ and we consider
	\[
		[x_i^N, [y_j^M, x_k^P \inv y_k^P]];
	\]
	the other cases are treated similarly.

	Using $O(P^2)$ times the relations $[x_k,y_k]$, $[x_k,z_k]$, $[x_k,z_k]$ and $O(PM)$ times the relation $ [y_j, x_k \inv z_k] $ we rewrite it as
	\[
		[x_i^N, [y_j^M, (z_k \inv y_k)^P]]
		=
		[x_i^N, ([y_j^M, z_k \inv y_k] z_k \inv y_k)^P (y_k \inv z_k)^P ].
	\]

	We then apply $ P$ times the identity $ [y_j^M, z_k \inv y_k] = [(y_j \inv z_j)^M, \inv y_k] $ (see \cref{thin:double-commutator-3} \cref{pow-swap-3-3}), which has area $O(M^2)$, to get
	\[
		[x_i^N, ([(y_j \inv z_j)^M, \inv y_k] z_k \inv y_k)^P (y_k \inv z_k)^P ] =	[x_i^N, \left((y_j \inv z_j)^M ([\inv y_k, z_j \inv y_j] z_j \inv y_j)^M z_k \inv y_k\right)^P (y_k \inv z_k)^P ]
	\]
	and finally $ x_i $ commutes with every factor on the right-hand side, so we conclude with $ O(NMP) $ additional relations.

\end{proof}

\begin{proof}[Proof of \cref{prop:powers-4}]
	If $ R \in \simplecomms $, the result is trivial, while if it belongs to $ \comms $, the proof is the same as the one of \cref{power:commutators-more-factors} (in both cases we get a quadratic upper bound). Finally, if $ R \in \triplecomms $, the result follows from \cref{power:tc-4}.
\end{proof}

\subsection{Proof of \cref{prop:thick} for four factors}
The goal of this section is to prove the following.
\begin{proposition}[\cref{prop:thick} for $n=4$]\label{prop:ass:thick-4}
	There is a constant $E$ (independent of $r$) such that, for every $\bq\in\bZ^r$ with $\abs{\bq}\leq m$, $\alpha\neq \beta\in \range 4$  and $i,j\in\range r$, the following holds:
	\[\Area_{\Rthick[m]}\left(\push_\bq \left(\left[\aii\alpha i,\aii\beta j \right]\right)\right)\le E . \]
\end{proposition}

Let $\bq=(q_1,\ldots,q_r)\in\ZZ^r$, $\abs{\bq}=\max\{\abs{q_1},\ldots,\abs{q_r}\}$, and let $\permut\colon\range {n-1}\to \range {n-1}$ be any map without fixed points.
We use the same terminology as in \cref{sec:push-more-factors}, so the push-down map is given by
\begin{align*}
	\push_\bq \left(\aii \alpha j\right) & =
	\paro {\permut (\alpha)}{j -1} \xii \alpha j\mparo {\permut (\alpha )}{j -1}, \\
	\push_\bq \left(\aii 4j\right)       & =
	\pa 2 \mxii 1j \mparo 2r[j]\xii 1j \mparo 2{j -1},
\end{align*}
for $ \alpha \in \range 3 $.

We proceed as in the case with at least $5$ factors.
\begin{lemma}\label{lem:cambio-lettera-4}%
	There exists a constant $C$ (independent of $r$) such that,
	for every ${\alpha,\beta,\gamma} \in\range {3}$ pairwise distinct, every positive integer $m$ and every $\bq, \bq' \in \ZZ$ satisfying $|\bq|\leq m$, it holds that
	\[
		\Area_{\Rthick[m]}\left(\left(\pa {\beta} \pa \alpha[\bq'] \mpa {\beta}\right) \left(\pa {\gamma} \pa \alpha[\bq'] \mpa {\gamma}\right)^{-1} \right) \leq C.
	\]
\end{lemma}

\begin{proof}
	Identical to \cref{lem:cambio-lettera-more-factors}.
\end{proof}
\begin{lemma}\label{lem:cambio-n-lettera-4}
	There exists a constant $C$ (independent of $r$) such that,
	for every $\alpha\ne\beta \in \range {3}$, every positive integer $m$ and every $\bq\in \ZZ$ satisfying $|\bq|\leq m$, it holds that
	\[
		\Area_{\Rthick[m]}\left(\left( \pa 2 \mxii 1j \mparo 2r[j]\xii 1j \mparo 2{j -1}, \right) \left( \pa \beta \mxii \alpha j \mparo \beta r[j]\xii \alpha j \mparo \beta {j -1} \right)^{-1} \right) \leq C.
	\]
\end{lemma}
\begin{proof}
	Identical to \cref{lem:cambio-n-lettera-more-factors}.
\end{proof}

We improve \cref{lem:recycle-three-factors} to thick relations.

\begin{lemma}\label{lem:recycle-three-factors-thick}
	There exists a constant $ C > 0 $ such that, if $ w \in \K[3] \subseteq \K $, then
	\[
		\Area_{\Rthick[m][4]} (w) \leq C \cdot \Area_{\Rthick[m][3]} (w)
	\]
\end{lemma}

\begin{proof}
	It is enough to prove the inequality when $\Area_{\Rthick[m][3]} (w)=1$.

	Let $ R \in \R[r+2][3] $, and consider $ w=\simmetriz\kappa_{\bq, \bq'}(R) \in \Rthick[m][3] $, where $ \abs \bq, \abs {\bq'} \leq m $. Using that $ \simmetriz\kappa_{\bq, \bq'}(\R[r+2][4]) \subseteq \Rthick[m][4]$ and \cref{lem:area-bound-after-homo} we have
	\[
		\Area_{\Rthick[m][4]} ( \simmetriz\kappa_{\bq, \bq'}(R) ) \leq \Area_{\simmetriz\kappa_{\bq, \bq'}(\R[r+2][4])} ( \simmetriz\kappa_{\bq, \bq'}(R) ) \leq \Area_{\R[r+2][4]} (R)
	\]\
	which is bounded by a constant by \cref{lem:recycle-three-factors}.
\end{proof}

Now we are ready to prove \cref{prop:thick} for $ n=4 $.
\begin{proof}
	If $ \alpha, \beta \neq 4 $, by \cref{lem:cambio-lettera-4} we may assume that $ \permut(\alpha)=\beta $ and $ \permut(\beta)=\alpha $. Then the push belongs to $ \ker(\freecopy r\alpha  \times \freecopy r\beta  \times \freecopy r4  \to \ZZ^r) \isom \K[3] $, and we have already proved in \cref{lem:push-ab-3} that it can be obtained by a constant number of thick relations of $ \Rthick[m][3] $. The conclusion follows from \cref{lem:recycle-three-factors-thick}.

	If $ \alpha \neq \beta = 4 $ the conclusion is similar by using \cref{lem:cambio-n-lettera-4} and then combining the proof of \cref{lem:push-ac-3} with \cref{lem:recycle-three-factors-thick}.
\end{proof}

%% file: main-uniform.bbl
\newcommand{\etalchar}[1]{$^{#1}$}
\begin{thebibliography}{{Dis}08b}

\bibitem[ABI{\etalchar{+}}25]{UniformBounds-25}
Dario Ascari, Federica Bertolotti, Giovanni Italiano, Claudio Llosa~Isenrich,
  and Matteo Migliorini.
\newblock {D}ehn functions of subgroups of products of free groups part {I}:
  uniform upper bounds.
\newblock {\em preprint}, 2025.

\bibitem[Art50]{Artin1950braid}
Emil Artin.
\newblock The theory of braids.
\newblock {\em American Scientist}, 38(1):112--119, 1950.

\bibitem[BBMS97]{baumslag}
Gilbert Baumslag, Martin~R. Bridson, Charles~F. Miller, III, and Hamish Short.
\newblock {Finitely Presented Subgroups of Automatic Groups and their
  Isoperimetric Functions}.
\newblock {\em Journal of the London Mathematical Society}, 56(2):292--304, 10
  1997.

\bibitem[BHMS09]{BHMS-09}
Martin~R. Bridson, James Howie, Charles~F. Miller, III, and Hamish Short.
\newblock Subgroups of direct products of limit groups.
\newblock {\em Ann. of Math. (2)}, 170(3):1447--1467, 2009.

\bibitem[BHMS13]{BHMS-13}
Martin~R. Bridson, James Howie, Charles~F. Miller, III, and Hamish Short.
\newblock On the finite presentation of subdirect products and the nature of
  residually free groups.
\newblock {\em Amer. J. Math.}, 135(4):891--933, 2013.

\bibitem[Bie81]{bieri}
Robert Bieri.
\newblock {\em Homological Dimension of Discrete Groups}.
\newblock Mathematics Department, Queen Mary College, University of London,
  1981.

\bibitem[BR09]{BriRil-09}
Martin~R. Bridson and Timothy~R. Riley.
\newblock Extrinsic versus intrinsic diameter for {R}iemannian filling-discs
  and van {K}ampen diagrams.
\newblock {\em J. Differential Geom.}, 82(1):115--154, 2009.

\bibitem[Bri]{BridsonPersonal}
Martin~R. Bridson.
\newblock Personal communication.

\bibitem[CF17]{carter2017stallings}
William Carter and Max Forester.
\newblock The {D}ehn functions of {S}tallings-{B}ieri groups.
\newblock {\em Mathematische Annalen}, 368, 06 2017.

\bibitem[Dis08a]{Dison-08-II}
Will Dison.
\newblock An isoperimetric function for {B}estvina-{B}rady groups.
\newblock {\em Bull. Lond. Math. Soc.}, 40(3):384--394, 2008.

\bibitem[{Dis}08b]{Dison-08}
Will {Dison}.
\newblock {\em {Isoperimetric functions for subdirect products and
  Bestvina-Brady groups}}.
\newblock PhD thesis, Imperial College London, October 2008.

\bibitem[Dis09]{Dison-09}
Will Dison.
\newblock A subgroup of a direct product of free groups whose {D}ehn function
  has a cubic lower bound.
\newblock {\em J. Group Theory}, 12(5):783--793, 2009.

\bibitem[Ger95]{gersten-95}
S.~M. Gersten.
\newblock Finiteness properties of asynchronously automatic groups.
\newblock In {\em Geometric group theory ({C}olumbus, {OH}, 1992)}, volume~3 of
  {\em Ohio State Univ. Math. Res. Inst. Publ.}, pages 121--133. de Gruyter,
  Berlin, 1995.

\bibitem[GS02]{GerSho-02}
Stephen~M. Gersten and Hamish Short.
\newblock Some isoperimetric inequalities for kernels of free extensions.
\newblock {\em Geom. Dedicata}, 92:63--72, 2002.
\newblock Dedicated to John Stallings on the occasion of his 65th birthday.

\bibitem[KLI22]{KrophollerLlosa}
Robert Kropholler and Claudio Llosa~Isenrich.
\newblock Dehn functions of coabelian subgroups of direct products of groups.
\newblock {\em Journal of the London Mathematical Society}, 107(1):123--152,
  2022.

\bibitem[Lic97]{lickorish1997introduction}
Raymond W.~B. Lickorish.
\newblock {\em An Introduction to Knot Theory}.
\newblock Graduate Texts in Mathematics. Springer New York, 1997.

\bibitem[LIT20]{LlosaTessera}
Claudio Llosa~Isenrich and Romain Tessera.
\newblock Residually free groups do not admit a uniform polynomial
  isoperimetric function.
\newblock {\em Proc. Amer. Math. Soc.}, 148(10):4203--4212, 2020.

\bibitem[LZ91]{lindstrom1991borromean}
Bernt Lindström and Hans-Olov Zetterström.
\newblock Borromean circles are impossible.
\newblock {\em The American Mathematical Monthly}, 98(4):340--341, 1991.

\bibitem[Mih68]{Mih-68}
K.~A. Mihailova.
\newblock The occurrence problem for free products of groups.
\newblock {\em Mathematics of the USSR-Sbornik}, 4(2):181, 1968.

\bibitem[Nan93]{nanyes1993borromean}
Ollie Nanyes.
\newblock An elementary proof that the borromean rings are non-splittable.
\newblock {\em The American Mathematical Monthly}, 100(8):786--789, 1993.

\bibitem[Pap96]{papasoglu1996asymptotic}
Panos Papasoglu.
\newblock {On the asymptotic cone of groups satisfying a quadratic
  isoperimetric inequality}.
\newblock {\em Journal of Differential Geometry}, 44(4):789 -- 806, 1996.

\bibitem[Sta63]{stallings}
John Stallings.
\newblock A finitely presented group whose 3-dimensional integral homology is
  not finitely generated.
\newblock {\em American Journal of Mathematics}, 85(4):541--543, 1963.

\end{thebibliography}


\newcommand{\etalchar}[1]{$^{#1}$}
\begin{thebibliography}{ABD{\etalchar{+}}13}

\bibitem[AB06]{AliBes-06}
Emina Alibegovi\'{c} and Mladen Bestvina.
\newblock Limit groups are {$\rm CAT(0)$}.
\newblock {\em J. London Math. Soc. (2)}, 74(1):259--272, 2006.

\bibitem[ABD{\etalchar{+}}13]{ABDDY-13}
Aaron Abrams, Noel Brady, Pallavi Dani, Moon Duchin, and Robert Young.
\newblock Pushing fillings in right-angled {A}rtin groups.
\newblock {\em J. Lond. Math. Soc. (2)}, 87(3):663--688, 2013.

\bibitem[Art50]{Artin1950braid}
Emil Artin.
\newblock The theory of braids.
\newblock {\em American Scientist}, 38(1):112--119, 1950.

\bibitem[BBMS97]{baumslag}
Gilbert Baumslag, Martin~R. Bridson, Charles~F. Miller, III, and Hamish Short.
\newblock {Finitely Presented Subgroups of Automatic Groups and their
  Isoperimetric Functions}.
\newblock {\em Journal of the London Mathematical Society}, 56(2):292--304, 10
  1997.

\bibitem[BHMS02]{BHMS-02}
Martin~R. Bridson, James Howie, Charles~F. Miller, III, and Hamish Short.
\newblock The subgroups of direct products of surface groups.
\newblock {\em Geometriae Dedicata}, 92:95--103, 2002.

\bibitem[BHMS09]{BHMS-09}
Martin~R. Bridson, James Howie, Charles~F. Miller, III, and Hamish Short.
\newblock Subgroups of direct products of limit groups.
\newblock {\em Ann. of Math. (2)}, 170(3):1447--1467, 2009.

\bibitem[BHMS13]{BHMS-13}
Martin~R. Bridson, James Howie, Charles~F. Miller, III, and Hamish Short.
\newblock On the finite presentation of subdirect products and the nature of
  residually free groups.
\newblock {\em Amer. J. Math.}, 135(4):891--933, 2013.

\bibitem[Bie81]{bieri}
Robert Bieri.
\newblock {\em Homological Dimension of Discrete Groups}.
\newblock Mathematics Department, Queen Mary College, University of London,
  1981.

\bibitem[BR84]{BauRos-84}
Gilbert Baumslag and James~E. Roseblade.
\newblock Subgroups of direct products of free groups.
\newblock {\em J. London Math. Soc. (2)}, 30(1):44--52, 1984.

\bibitem[BR09]{BriRil-09}
Martin~R. Bridson and Timothy~R. Riley.
\newblock Extrinsic versus intrinsic diameter for {R}iemannian filling-discs
  and van {K}ampen diagrams.
\newblock {\em J. Differential Geom.}, 82(1):115--154, 2009.

\bibitem[Bri]{BridsonPersonal}
Martin~R. Bridson.
\newblock Personal communication.

\bibitem[Bri99]{bridson-doubles}
Martin~R. Bridson.
\newblock Doubles, finiteness properties of groups, and quadratic isoperimetric
  inequalities.
\newblock {\em J. Algebra}, 214(2):652--667, 1999.

\bibitem[CF17]{carter2017stallings}
William Carter and Max Forester.
\newblock The {D}ehn functions of {S}tallings-{B}ieri groups.
\newblock {\em Mathematische Annalen}, 368, 06 2017.

\bibitem[DERY09]{BERY-09}
Will Dison, Murray Elder, Timothy~R. Riley, and Robert Young.
\newblock The {D}ehn function of {S}tallings' group.
\newblock {\em Geom. Funct. Anal.}, 19(2):406--422, 2009.

\bibitem[Dis08a]{Dison-08-II}
Will Dison.
\newblock An isoperimetric function for {B}estvina-{B}rady groups.
\newblock {\em Bull. Lond. Math. Soc.}, 40(3):384--394, 2008.

\bibitem[{Dis}08b]{Dison-08}
Will {Dison}.
\newblock {\em {Isoperimetric functions for subdirect products and
  Bestvina-Brady groups}}.
\newblock PhD thesis, Imperial College London, October 2008.

\bibitem[Dis09]{Dison-09}
Will Dison.
\newblock A subgroup of a direct product of free groups whose {D}ehn function
  has a cubic lower bound.
\newblock {\em J. Group Theory}, 12(5):783--793, 2009.

\bibitem[Geo08]{Geo-08}
Ross Geoghegan.
\newblock {\em Topological methods in group theory}, volume 243 of {\em
  Graduate Texts in Mathematics}.
\newblock Springer, New York, 2008.

\bibitem[Ger95]{gersten-95}
S.~M. Gersten.
\newblock Finiteness properties of asynchronously automatic groups.
\newblock In {\em Geometric group theory ({C}olumbus, {OH}, 1992)}, volume~3 of
  {\em Ohio State Univ. Math. Res. Inst. Publ.}, pages 121--133. de Gruyter,
  Berlin, 1995.

\bibitem[GS02]{GerSho-02}
S.~Gersten and H.~Short.
\newblock Some isoperimetric inequalities for kernels of free extensions.
\newblock {\em Geom. Dedicata}, 92:63--72, 2002.
\newblock Dedicated to John Stallings on the occasion of his 65th birthday.

\bibitem[KLI22]{KrophollerLlosa}
Robert Kropholler and Claudio Llosa~Isenrich.
\newblock Dehn functions of coabelian subgroups of direct products of groups.
\newblock {\em Journal of the London Mathematical Society}, 107(1):123--152,
  2022.

\bibitem[Koc10]{Kochloukova-10}
Dessislava~H. Kochloukova.
\newblock On subdirect products of type {${\rm FP}_m$} of limit groups.
\newblock {\em J. Group Theory}, 13(1):1--19, 2010.

\bibitem[Kuc14]{Kuckuck-14}
Benno Kuckuck.
\newblock Subdirect products of groups and the {$n$}-{$(n+1)$}-{$(n+2)$}
  conjecture.
\newblock {\em Q. J. Math.}, 65(4):1293--1318, 2014.

\bibitem[LI20]{Llo-20}
Claudio Llosa~Isenrich.
\newblock K\"{a}hler groups and subdirect products of surface groups.
\newblock {\em Geom. Topol.}, 24(2):971--1017, 2020.

\bibitem[LI24]{Llo-24}
Claudio Llosa~Isenrich.
\newblock From the second {BNSR} invariant to {D}ehn functions of coabelian
  subgroups.
\newblock {\em arXiv preprint arXiv:2404.12334}, 2024.

\bibitem[Lic97]{lickorish1997introduction}
Raymond W.~B. Lickorish.
\newblock {\em An Introduction to Knot Theory}.
\newblock Graduate Texts in Mathematics. Springer New York, 1997.

\bibitem[LIT20]{LlosaTessera}
Claudio Llosa~Isenrich and Romain Tessera.
\newblock Residually free groups do not admit a uniform polynomial
  isoperimetric function.
\newblock {\em Proc. Amer. Math. Soc.}, 148(10):4203--4212, 2020.

\bibitem[LZ91]{lindstrom1991borromean}
Bernt Lindström and Hans-Olov Zetterström.
\newblock Borromean circles are impossible.
\newblock {\em The American Mathematical Monthly}, 98(4):340--341, 1991.

\bibitem[Mih68]{Mih-68}
K.~A. Mihailova.
\newblock The occurrence problem for free products of groups.
\newblock {\em Mathematics of the USSR-Sbornik}, 4(2):181, 1968.

\bibitem[Nan93]{nanyes1993borromean}
Ollie Nanyes.
\newblock An elementary proof that the borromean rings are non-splittable.
\newblock {\em The American Mathematical Monthly}, 100(8):786--789, 1993.

\bibitem[Pap96]{papasoglu1996asymptotic}
Panos Papasoglu.
\newblock {On the asymptotic cone of groups satisfying a quadratic
  isoperimetric inequality}.
\newblock {\em Journal of Differential Geometry}, 44(4):789 -- 806, 1996.

\bibitem[Sta63]{stallings}
John Stallings.
\newblock A finitely presented group whose 3-dimensional integral homology is
  not finitely generated.
\newblock {\em American Journal of Mathematics}, 85(4):541--543, 1963.

\end{thebibliography}
